\documentclass[11pt]{article}

\usepackage[a4paper,margin=1.12in]{geometry}
\usepackage{amsmath,amssymb,amsthm,mathtools}
\usepackage{enumitem}
\usepackage{xcolor}
\usepackage[colorlinks=true,linkcolor=blue!45!black,citecolor=blue!45!black,urlcolor=blue!55!black,hypertexnames=false]{hyperref}
\usepackage{mathrsfs}
\usepackage{bbm}
\usepackage{microtype}
\usepackage{tikz}

\numberwithin{equation}{section}

\newtheorem{theorem}{Theorem}[section]
\newtheorem{proposition}[theorem]{Proposition}
\newtheorem{lemma}[theorem]{Lemma}
\newtheorem{corollary}[theorem]{Corollary}

\theoremstyle{definition}
\newtheorem{definition}[theorem]{Definition}
\newtheorem*{definition*}{Definition}
\newtheorem{remark}[theorem]{Remark}

\newcommand{\R}{\mathbb R}
\newcommand{\Z}{\mathbb Z}
\newcommand{\Pp}{\mathbb P}
\newcommand{\Ee}{\mathbb E}
\newcommand{\Pbb}{\Pp}
\newcommand{\1}{\mathbf{1}}
\newcommand{\Capa}{\operatorname{cap}}
\newcommand{\diam}{\operatorname{diam}}
\newcommand{\dist}{\operatorname{dist}}
\newcommand{\dimH}{\dim_{\mathrm H}}
\newcommand{\LL}{\mathcal L}
\newcommand{\BL}{\LL}
\newcommand{\Bcal}{\mathcal B}
\newcommand{\Loop}{\mathrm{loop}}
\newcommand{\muLoop}{\mu^{\Loop}_{\R^3}}
\newcommand{\dint}{\,\mathrm d}

\newcommand{\ZMG}{\widetilde{\mathbb Z}^{3}}

\title{On loop-soup clusters in three dimensions and higher}
\author{Yifan Gao\thanks{Institute for Theoretical Sciences, Westlake University,
\href{mailto:gaoyifan75@westlake.edu.cn}{\textcolor{black}{gaoyifan75@westlake.edu.cn}}.}}
\date{}

\begin{document}

\maketitle

\begin{abstract}
We study clusters of the Brownian loop soup on
\(\mathbb R^3\) and of random walk and metric-graph loop soups on
\(\mathbb Z^d\) and \(\widetilde{\mathbb Z}^{d}\) within a unified framework.
For the continuum model, when the one-arm exponent \(\xi(\alpha)\) is positive,
we prove that almost surely every cluster has Hausdorff dimension \(3-\xi(\alpha)\).
A key step is to show that, conditioned on the one-arm event, a crossing cluster
has capacity comparable to the crossing scale with uniformly positive probability.
This gives one-arm estimates of order
\(\alpha r^{-\xi(\alpha)}\), uniformly down to zero intensity, and shows that
\(\xi(\alpha)\) is strictly decreasing and left-continuous.
Moreover, \(1-\xi(\alpha)\asymp\alpha\) as \(\alpha\downarrow0\).
As another application of the one-arm estimates,  we prove that several critical
values coincide, resolving a conjecture of Jego and Lupu~\cite{JL26}.

The proof combines a decoupling technique with a multiscale recursive argument
to control the combined effects of large and small loops.
These methods also extend to loop soups on \(\mathbb Z^d\) and \(\widetilde{\mathbb Z}^{d}\)
for \(d=3,4,5\),
yielding capacity lower bounds of order \(r\), \(r^2/\log r\), and \(r^2\)
for the cluster of the origin, conditioned on the one-arm event,
throughout the small-loop decay regime.
In particular, these bounds imply strict polynomial comparisons between
one-arm probabilities at different intensities in dimension three.
We also show that the small-loop decay regime on \(\widetilde{\mathbb Z}^{d}\) has upper endpoint
\(\frac12\) by adapting the OSSS sharpness framework to the truncated loop soup.

\end{abstract}

\tableofcontents

\section{Introduction}

The Brownian loop soup, introduced by Lawler and Werner~\cite{LW04}, is a
Poisson ensemble of Brownian loops.
In two dimensions, it is conformally invariant and plays a central role in
conformal random geometry, with close connections to Schramm--Loewner evolution
and conformal restriction measures~\cite{LSW03}.
The loop-soup clusters, defined as the maximal connected components of loops,
have remarkable geometric properties.
The outer boundaries of the outermost clusters form a simple conformal loop ensemble~\cite{SW12}.

In this paper, we study the geometry of continuum clusters in dimension three,
as well as clusters of loop soups on \(\mathbb Z^d\) and
\(\widetilde{\mathbb Z}^{d}\) for \(d\ge3\).
Dimension three is the only non-planar dimension in which distinct Brownian
loops can intersect and form non-trivial clusters~\cite{MP10,Wer21,JL26}.
Moreover, the three-dimensional Brownian loop soup has a non-trivial percolation
phase transition~\cite{JL26}.
In the regime when the one-arm probability decays, our main result
(Theorem~\ref{thm:all-cluster-dimension}) identifies the Hausdorff dimension of every cluster as
\(3-\xi(\alpha)\), where \(\xi(\alpha)\in(0,1)\) is the one-arm exponent
whose existence was established in~\cite{CD25a,JL26}.

The dimension predicted by the one-arm exponent has a simple heuristic explanation.
If an \(\varepsilon\)-ball is hit by a macroscopic cluster with
probability of order \(\varepsilon^{\xi(\alpha)}\), then the expected volume of
the cluster's \(\varepsilon\)-neighborhood suggests the dimension \(3-\xi(\alpha)\).
To make this heuristic rigorous, we first establish up-to-constants
estimates for the one-arm probability (Theorem~\ref{thm:utc}), sharpening the estimate
\(\varepsilon^{\xi(\alpha)+o(1)}\) in~\cite{CD25a,JL26}.
It turns out that the usual decoupling technique alone is
insufficient, and the key is to control the geometry of a crossing cluster and show
that its capacity is comparable to the crossing scale.
In this respect, the problem is similar to obtaining upper bounds for the critical
one-arm probability on the metric graph, where capacity estimates also play an
essential role~\cite{DW20,DPR23,CD24,DPR24}.
In our setting, however, the isomorphism with the Gaussian free field (GFF)
used in those works is unavailable.

To overcome this difficulty, we devise a multiscale recursive argument that
combines the effects of small and large loops.
This yields the desired capacity estimate for crossing clusters throughout the
subcritical regime (Theorem~\ref{thm:continuum-robust-summary}).
The estimate then allows us to glue crossing clusters and obtain quasi-multiplicativity of
the one-arm probabilities.
Furthermore, using Russo's formula, we show in Theorem~\ref{thm:xi-strict-monotone} that the one-arm exponent \(\xi(\alpha)\) is strictly decreasing.
As another application of our improved one-arm estimates,
we show in Theorem~\ref{thm:continuum-thresholds} that several critical
values of the
three-dimensional Brownian loop soup coincide, resolving
\cite[Conjecture~1.6]{JL26}.
Along the way, we prove the sharp phase transition
for the unit-truncated loop soup via the OSSS decision-tree framework (Theorem~\ref{thm:unit-truncated-sharpness}),
which is of independent interest.

The recursive argument also extends to loop soups on discrete and metric graphs.
Throughout the small-loop decay regime, we obtain a lower bound for the capacity
of the cluster containing the origin, conditional on the one-arm event
(Theorem~\ref{thm:rwls-robust-summary}).
In particular, we obtain strict polynomial comparisons between one-arm probabilities at
different intensities in dimension three (Theorem~\ref{thm:rwls-strict-summary}).
The analysis relies on the decay of crossing probabilities for small loops.
We therefore introduce an auxiliary threshold \(\alpha_\ast\) that captures this behavior,
which is shown to coincide with $\frac12$ on the metric graph \(\widetilde{\mathbb Z}^{d}\) for \(d=3,4,5\).

\subsection{Hausdorff dimension of clusters}
Let
\(\mu_{\R^3}^{\Loop}\) denote the Brownian loop measure on \(\R^3\).  The
Brownian loop soup \(\LL_{\R^3}^\alpha\) of intensity \(\alpha>0\) is a
Poisson point process on loops with intensity measure
\(\alpha\mu_{\R^3}^{\Loop}\). See Section~\ref{sec:pre} for details.
Below, we use $\Pp_\alpha$ to denote the probability measure of \(\LL_{\R^3}^\alpha\).

Let \(\mathcal L\) be a collection of loops in \(\R^3\).  Define an equivalence relation
\(\sim\) on \(\mathcal L\) by declaring that
\(\ell\sim\ell'\) if and only if there exist \(n\ge1\) and
\(\ell_1,\ldots,\ell_n\in\mathcal L\), with \(\ell_1=\ell\) and
\(\ell_n=\ell'\), such that \(\ell_i\cap\ell_{i+1}\ne\varnothing\) for every
\(1\le i<n\).  For each equivalence class
\(E\in\mathcal L/\!\sim\), define its \emph{cluster} to be
\(\mathcal C_E:=\overline{\bigcup_{\ell\in E}\ell}\), where the closure is
taken in \(\R^3\).  We write
\(\mathscr C(\mathcal L):=\{\mathcal C_E:E\in\mathcal L/\!\sim\}\)
for the collection of all clusters of \(\mathcal L\).
Two sets \(A,B\subset\R^3\) are said to be connected in \(\mathcal L\),
denoted by
\(
        A \xleftrightarrow{\mathcal L} B
\),
if the union of the loops in some equivalence class of \(\mathcal L\)
intersects both \(A\) and \(B\).

For \(x\in\R^3\) and \(r>0\), let \(B(x,r)\) be the ball of radius \(r\) centered at \(x\).
For \(0<s<r\), let
\(A(x,s,r):=B(x,r)\setminus\overline{B(x,s)}\) be the annulus with radii $s$ and $r$.
Abbreviate \(B_r:=B(0,r)\), \(S_r:=\partial B_r\), and
\(A(s,r):=A(0,s,r)\).

For \(0<s<r\), we define the \emph{one-arm (or crossing) event}
\(
        \mathcal E_{s,r}:=\{S_s \xleftrightarrow{\LL_{\R^3}^\alpha} S_r\}
\).
For \(s=1\), we use the following abbreviation
\begin{equation}\label{eq:one-arm-prob-def}
        \mathcal E_r:=\mathcal E_{1,r} \quad  \text{ and } \quad p_r^{\alpha}:=\Pp_\alpha(\mathcal E_r).
\end{equation}
The Brownian loop soup is scale invariant, and hence
\(\Pp_\alpha(\mathcal E_{s,r})=\Pp_\alpha(\mathcal E_{1,r/s})=p_{r/s}^{\alpha}\).
The existence of the one-arm exponent for every intensity \(\alpha>0\)
was established in~\cite{CD25a,JL26}.
More precisely, there exists \(\xi=\xi(\alpha)\ge0\) such that
\begin{equation}\label{eq:one-arm-log-asymp}
    p_r^{\alpha}=r^{-\xi+o(1)}
    \qquad\text{as }r\to\infty.
\end{equation}
These works also show that \(\xi<1\), where \(1\) is the exponent for a single crossing loop.
Define the one-arm decay regime by
\begin{equation}\label{eq:I0}
        \mathrm{I}_0:=\{ \alpha>0: \lim_{r\to\infty} p_r^{\alpha}=0 \}.
\end{equation}
By~\cite{JL26}, we have
$\sup \mathrm{I}_0=\alpha_{\mathrm{B2B}}<\infty$, where $\alpha_{\mathrm{B2B}}$ is the ball-to-ball threshold introduced in
\cite[(1.5)]{JL26}, and $\xi>0$ for all $\alpha\in \mathrm{I}_0$.
It is not known whether the critical value \(\alpha_{\mathrm{B2B}}\) belongs to \(\mathrm{I}_0\).
To allow for this possibility, we state our results throughout \(\mathrm{I}_0\),
rather than only on \((0,\alpha_{\mathrm{B2B}})\).
 In fact, Theorem~\ref{cor:bdelta} below identifies
$\alpha_{\mathrm{B2B}}$ with other natural critical values
associated with the Brownian loop soup (see Section~\ref{subsec:equality}).

Our first main result identifies the Hausdorff dimension of every
three-dimensional cluster as \(3-\xi\in(2,3)\) for every \(\alpha\in\mathrm{I}_0\).

\begin{theorem}[Dimension]
\label{thm:all-cluster-dimension}
For every \(\alpha\in \mathrm{I}_0\), we have
\[
        \Pp_\alpha\left(
        \forall \mathcal C\in\mathscr C(\LL_{\R^3}^\alpha),\ 
        \dimH(\mathcal C)=3-\xi
        \right)=1 .
\]
\end{theorem}

Since \(0<\xi<1\), every cluster has dimension strictly between \(2\) and
\(3\), i.e., larger than the dimension of a Brownian loop and smaller than
that of the ambient space.
Recall that a cluster is the closure of the union of the loops in an equivalence class.
Hence, the excess dimension above \(2\) can be regarded as an ultraviolet accumulation effect arising from loops at arbitrarily small scales.

We now discuss the implication at intensity \(\alpha=\frac12\).
It was shown in~\cite{CDL24} that \(\frac12\) is the critical intensity
for loop-soup percolation on the metric graph \(\widetilde{\mathbb Z}^{3}\)
(and, more generally, on every quasi-transitive transient metric graph);
see Section~\ref{sec:rwls-results}.
Whether \(\frac12\) is also the critical value for
\(\mathcal L_{\mathbb R^3}^{\alpha}\) remains open; see the question following
\cite[Remark~1.7]{CD25a}.
Moreover, the authors of~\cite{CD25a} prove that \(\xi(\frac12)>\frac12\),
where \(\frac12\) is the one-arm exponent for the critical loop soup on
\(\widetilde{\mathbb Z}^{3}\)~\cite{CD24,DPR24}.
They also show that the upper box-counting dimension, and hence the
Hausdorff dimension, of every cluster of \(\mathcal L_{\mathbb R^3}^{1/2}\)
is at most \(3-\xi(\frac12)<\frac52\).
The matching lower bound was left open in~\cite[Remark~1.3]{CD25a}.
Since \(\frac12\in\mathrm{I}_0\) by~\cite{CD25a},
Theorem~\ref{thm:all-cluster-dimension} gives this lower bound, and every cluster of
\(\mathcal L_{\mathbb R^3}^{1/2}\) has Hausdorff dimension
\(3-\xi(\frac12)\in(2,\frac52)\).

\subsection{One-arm estimates}

The proof of Theorem~\ref{thm:all-cluster-dimension} relies on the following up-to-constants
estimate for the one-arm probability. It sharpens
\eqref{eq:one-arm-log-asymp} by removing any subpolynomial multiplicative correction,
and the bounds are uniform in \(0<\alpha\le\bar\alpha\) for each fixed
\(\bar\alpha\in\mathrm I_0\).

\begin{theorem}[Up-to-constants estimate]
\label{thm:utc}
For every \(\bar\alpha\in\mathrm{I}_0\), there exist constants
\(0<c<C<\infty\), depending only on \(\bar\alpha\), such that
for all \(0<\alpha\le\bar\alpha\) and \(r\ge2\),
\[
 c\alpha r^{-\xi(\alpha)}\le p_r^{\alpha}\le C\alpha r^{-\xi(\alpha)}.
\]
\end{theorem}

We briefly explain the key idea of the proof.
One route to the existence of the one-arm exponent in \eqref{eq:one-arm-log-asymp}
is to prove submultiplicativity, \(p_{rs}^{\alpha}\lesssim p_r^{\alpha}p_s^{\alpha}\)\footnote{For nonnegative quantities \(f,g\), we write \(f\lesssim g\) if \(f\le Cg\) for a positive constant. We write \(f\gtrsim g\) if \(g\lesssim f\), and \(f\asymp g\) if both hold.}, and then apply Fekete's lemma.
This approach was carried out in~\cite{CD25a}, while~\cite{JL26} also established an
integrated version of the inequality.
In Section~\ref{sec:submultiplicativity}, we give another proof using the loop-splitting
technique developed in~\cite{GNQ24,GNQ25}
(although those papers mainly concern two-dimensional loop soups, the technique
extends to higher dimensions).
As we will see, loop splitting and resampling recur throughout the paper.

The reverse inequality, \(p_{rs}^{\alpha}\gtrsim p_r^{\alpha}p_s^{\alpha}\), is more delicate because it requires control of the capacity of crossing clusters.
Suppose that crossings from \(S_1\) to \(S_r\) and from \(S_r\) to \(S_{rs}\)
each carry capacity of order \(r\) near scale \(r\), with uniformly positive
probability conditional on the respective one-arm events. Then, a fresh Brownian loop at scale \(r\) can weld them
with positive probability,
yielding the desired supermultiplicativity.
This route was also outlined in~\cite[Section~7]{JL26}, which identified control of
the geometry of the crossing cluster as the main difficulty.
Theorem~\ref{thm:continuum-robust-summary} below provides this capacity estimate.

\begin{theorem}
\label{thm:continuum-robust-summary}
There exists \(\kappa_0>0\) such that, for every
\(\bar\alpha\in\mathrm{I}_0\), there is \(c>0\), depending only on
\(\bar\alpha\), for which the following holds simultaneously for
\(0<\alpha\le\bar\alpha\) and \(r\ge2\):
\begin{equation}\label{eq:continuum-robust-summary}
       \Pp_\alpha\Bigl(
 \exists\text{ cluster } \mathcal C \text{ across } A(1,r)
 \text{ such that }
 \Capa\bigl(\mathcal C\cap B_{2r}\bigr)\ge\kappa_0r
 \Bigr)
 \ge c\,p_r^{\alpha},
\end{equation}
where \(\Capa\) denotes the Newtonian capacity in \(\R^3\).
\end{theorem}

The key step in the proof of Theorem~\ref{thm:continuum-robust-summary} is to show that,
conditioned on the one-arm event, a crossing cluster \(\mathcal C\) contains a loop
of diameter of order \(r\) with uniformly positive probability.
Note that the earlier bound \(\xi<1\) shows that a single loop is unlikely to realize the entire
crossing. Our observation is that a crossing made entirely of small loops is also unlikely,
where ``small'' is measured relative to the scale at which a loop lies.

Theorem~\ref{thm:continuum-robust-summary} allows us to construct pivotal
loops for the one-arm event.
Together with Russo's formula, this construction yields the quantitative strict
monotonicity of $\xi(\alpha)$ below.
Moreover, we prove the left continuity of the one-arm exponent throughout
\(\mathrm I_0\) and a two-sided bound for small intensities.
In particular, \(\xi\) is continuous at \(\alpha=0\), upon setting
\(\xi(0):=1\), the one-arm exponent for a single Brownian loop.

\begin{theorem}[Strict monotonicity]
\label{thm:xi-strict-monotone}
For any compact interval \(I\subset \mathrm{I}_0\), there exists a constant
\({c}>0\) such that, for all \(\alpha,\beta\in I\) with
\(\alpha<\beta\), one has
\begin{equation}\label{eq:xi-linear-separation}
    \xi(\alpha)-\xi(\beta)\ge {c}(\beta-\alpha).
\end{equation}
{
Moreover, there exist constants \(0<c<C<\infty\) and
\(\alpha_0\in(0,1)\) such that, for all \(0<\alpha\le\alpha_0\),
\begin{equation}\label{eq:xi-0}
        c\alpha\le1-\xi(\alpha)\le C\alpha.
\end{equation}
The function \(\xi\) is left-continuous on \(\mathrm I_0\).
}
Consequently, \(\xi\) is strictly decreasing on \(\mathrm{I}_0\) and
\(\lim_{\alpha\downarrow0}\xi(\alpha)=\xi(0)=1\).
\end{theorem}

The bound \(\xi(\alpha)<1\) was proved
in~\cite{CD25a,JL26}. The quantitative estimate \eqref{eq:xi-0} follows
from recursions for the capacity; see
Theorem~\ref{imp:thm:linear}.
Combining Theorem~\ref{thm:xi-strict-monotone} with
Theorem~\ref{thm:all-cluster-dimension}, we conclude that the cluster dimension
is strictly increasing in \(\alpha\) on \(\mathrm{I}_0\)
and tends to \(2\) as \(\alpha\downarrow0\).

\begin{remark}[Continuity]
Theorem~\ref{imp:thm:left-continuity} proves left continuity of \(\xi\)
on \(\mathrm I_0\), while right continuity at positive intensities remains open.
In fact, any interior discontinuity must be a downward jump to the right, and there
are at most countably many such points. Proving the continuity of \(\xi\) would give
a family of random fractals in \(\mathbb R^3\) whose Hausdorff dimension
varies continuously with the intensity.
\end{remark}

\begin{remark}[Minkowski content]
The results and methods of this paper can also be used to establish the existence
of Minkowski content for three-dimensional loop-soup clusters; see~\cite{GaoMinkowski}. Note that the two-dimensional analogue has been proved
in~\cite{JLQ23} using the one-arm estimate of~\cite{JLQ25}.
\end{remark}

\subsection{Equality of critical values and sharpness}\label{subsec:equality}

We next identify various critical values of the Brownian loop soup introduced in
\cite{JL26}.
Following the terminology of \cite{JL26}, we collect the percolation,
uniqueness, ball-to-ball, and lower-truncation thresholds below:
\[
\begin{aligned}
 \alpha_c
 &:=\inf\Bigl\{\alpha>0:
 \mathbb P_\alpha\bigl(\mathcal L_{\mathbb R^3}^{\alpha}
 \text{ has an unbounded cluster}\bigr)>0\Bigr\},\\
 \alpha_u
 &:=\inf\Bigl\{\alpha>0:
 \mathbb P_\alpha\bigl(\mathcal L_{\mathbb R^3}^{\alpha}
 \text{ has a unique unbounded cluster}\bigr)>0\Bigr\},\\
 \alpha_{\mathrm{B2B}}
 &:=\inf\Bigl\{\alpha>0:
 \mathbb P_\alpha\bigl(\forall r>1,\,
 B_1\xleftrightarrow{\mathcal L_{\mathbb R^3}^{\alpha}}S_r\bigr)>0\Bigr\},\\
 \alpha_{\mathrm{tr}}
 &:=\inf\Bigl\{\alpha>0:
 \mathbb P_\alpha\bigl(\mathcal L_{\mathbb R^3}^{\alpha,\ge1}
 \text{ has an unbounded cluster}\bigr)>0\Bigr\}.
\end{aligned}
\]
Here, \(\mathcal L_{\mathbb R^3}^{\alpha,\ge1}\) denotes the restriction of
\(\mathcal L_{\mathbb R^3}^{\alpha}\) to loops of diameter at least \(1\).
It was proved in \cite[Theorem~1.2]{JL26} that
\begin{equation}\label{eq:four-critical}
        \alpha_{\mathrm{B2B}}\le\alpha_c\le\alpha_u\le\alpha_{\mathrm{tr}},
\end{equation}
and all of these critical values are conjectured to be equal \cite[Conjecture~1.6]{JL26}.
We are also interested in the truncation from above, which captures
the effect of small loops. For this,
let \(\LL_{\R^3}^{\alpha,\le1}\) be the loop soup obtained by
retaining only loops of diameter at most \(1\),
and define the associated critical value by
\begin{equation}\label{eq:alpha-star-def}
        \alpha_\ast:=\sup\Bigl\{\alpha>0: \lim_{r\to\infty}\Pp_\alpha\Bigl(
        S_r
        \xleftrightarrow{\LL_{\R^3}^{\alpha,\le1}}
        S_{2r}
    \Bigr)=0\Bigr\}.
\end{equation}
The next theorem identifies all these thresholds, resolving
\cite[Conjecture~1.6]{JL26}.

\begin{theorem}[Equality of critical values]
\label{cor:bdelta}
\label{thm:continuum-thresholds}
For the Brownian loop soup on \(\mathbb R^3\),
\begin{equation}\label{eq:four-critical-qual}
 \alpha_\ast=\alpha_{\mathrm{B2B}}
 =\alpha_c=\alpha_u=\alpha_{\mathrm{tr}}.
\end{equation}
Moreover, for every \(\beta>\alpha_{\mathrm{B2B}}\) and every fixed
nonempty connected open set \(D\subset\mathbb R^3\), almost surely all loops of
\(\mathcal L_D^\beta\) belong to a single cluster.
\end{theorem}

Recall that \(\alpha_{\mathrm{B2B}}=\sup\mathrm I_0\).
Hence, all the preceding results stated for $\alpha\in \mathrm I_0$
remain true for any $\alpha$ below  any one of the critical values in \eqref{eq:four-critical-qual}.
Note that most of the results in~\cite{GLL26} are likewise stated for the
three-dimensional Brownian loop soup with \(\alpha\in\mathrm I_0\),
including the up-to-constants estimates for generalized non-intersection probabilities
and the Hausdorff dimension for local cut points.
The same extension therefore applies to these results.

For the proof of \eqref{eq:four-critical-qual}, we will first use a
zero--one law argument to get the equality \(\alpha_\ast=\alpha_{\mathrm{B2B}}\).
Then, we use the one-arm up-to-constants estimates in Theorem~\ref{thm:utc}
to show $\alpha_{\mathrm{B2B}}\ge \alpha_{\mathrm{tr}}$.  This, combined with
\eqref{eq:four-critical}, proves the remaining equalities in \eqref{eq:four-critical-qual}.

We also study the sharp phase transition for the unit-truncated loop soup \(\LL_{\R^3}^{\alpha,\le1}\),
which is of independent interest.
With such a truncation, the model is well suited to the OSSS decision-tree
 framework; see, e.g.,~\cite{OSSS05,DRT17,DRT18,LPY23}.
Applying this powerful technique, we can derive the following sharpness result.

\begin{theorem}[Sharpness]\label{thm:unit-truncated-sharpness}
The following properties hold for \(\LL_{\R^3}^{\alpha,\le1}\).
\begin{enumerate}[label=\textnormal{(\roman*)}]
\item If \(0<\alpha<\alpha_\ast\), then there are \(C,c\in(0,\infty)\) such that
\(\Pp_\alpha(S_1\xleftrightarrow{\LL_{\R^3}^{\alpha,\le1}}S_r)\le C\,e^{-cr}\) for all \(r>1\).
\item If \(\alpha>\alpha_\ast\), then
\(\Pp_\alpha(\forall r>1,\,S_1\xleftrightarrow{\LL_{\R^3}^{\alpha,\le1}}S_r)=1\).
\end{enumerate}
\end{theorem}

We will give a detailed proof of
Theorem~\ref{thm:unit-truncated-sharpness} in Section~\ref{sec:unit-truncated-sharpness},
which  will guide our subsequent study of the truncated loop soup on the metric graph.

\subsection{Discrete and metric-graph results}
\label{sec:rwls-results}

We now turn to the companion results on \(\Z^d\) and on its metric graph
\(\widetilde{\mathbb Z}^{d}\), where \(d\ge3\).
The latter is obtained by replacing each lattice edge by a unit interval.
To state the results for both models, we write
\(\mathbb X\in\{\Z^d,\widetilde{\mathbb Z}^{d}\}\),
and for \(x\in\mathbb X\) and \(0<s<r\), we set
\(\mathsf{B}(x,r):=x+[-r,r]^d\cap\mathbb X\),
\(\mathsf{S}(x,r):=\partial \mathsf{B}(x,r)\),
and \(
 \mathsf A(x,s,r)
 :=\mathsf B(x,r)\setminus\mathsf B(x,s)
\).
For sets centered at the origin, write
\(\mathsf{B}_r:=\mathsf{B}(0,r)\),
\(\mathsf{S}_r:=\mathsf{S}(0,r)\) and $\mathsf A(s,r):= \mathsf A(0,s,r)$ for brevity.
Let \(\mu_{\mathbb X}^{\Loop}\) be the loop measure on $\mathbb X$,
and let \(\mathcal L_{\mathbb X}^\alpha\) be the Poisson point process of loops with
intensity measure \(\alpha\mu_{\mathbb X}^{\Loop}\) (introduced in~\cite{LTF07,Lup16}); see Section~\ref{rwls:sec:inputs}
for the precise definitions.
For $D\subseteq \mathbb X$, let \(\mathcal L_{D}^\alpha\) be \(\mathcal L_{\mathbb X}^\alpha\) restricted
to loops contained in $D$.
With a slight abuse of notation, we still use \(\Pp_\alpha\) to
denote the law of the loop soup \(\mathcal L_{\mathbb X}^\alpha\).

For \(r\ge1\), we define the one-arm event and one-arm probability by
\[
 \mathcal E_r^{\mathbb X}
 :=\{0\xleftrightarrow{\mathcal L_{\mathbb X}^\alpha}
 \mathsf{S}_{r}\},
 \qquad
 p_r^{\mathbb X,\alpha}:=\Pp_\alpha(\mathcal E_r^{\mathbb X}).
\]
Unlike the continuum event \(\mathcal E_r\) in \eqref{eq:one-arm-prob-def}, the event
\(\mathcal E_r^{\mathbb X}\) is rooted at the origin.
The percolation threshold is given by
\[
 \alpha_{\rm c}^{\mathbb X}
 :=\inf\left\{\alpha>0:
 \lim_{r\to\infty} p_r^{\mathbb X,\alpha}>0\right\}{.}
\]

\begingroup
\relax
It is now known that Lupu's coupling of the metric-graph loop soup at intensity \(\frac12\)
with the Gaussian free field gives
\(\alpha_{\rm c}^{\widetilde{\mathbb Z}^{d}}\ge\frac12\)~\cite{Lup16}.
The reverse inequality was proved in~\cite{CDL24}; thus
\(\alpha_{\rm c}^{\widetilde{\mathbb Z}^{d}}=\frac12\) for every \(d\ge3\).
\endgroup
It was recently proved in~\cite{BDKM26} that
\(\alpha_{\rm c}^{\mathbb Z^d}>\alpha_{\rm c}^{\widetilde{\mathbb Z}^{d}}=\frac12\)
for \(d\ge5\).
However, for \(d=3,4\), it remains open whether \(\alpha_{\rm c}^{\mathbb Z^d}>\frac12\).

We now introduce the analogue of \eqref{eq:alpha-star-def} on
\(\mathbb X\). It captures the effect of small loops and is a key parameter in
our recursive argument.
For \(L>0\), we again use the superscript \(\le L\) to denote the restriction to loops whose
diameter is at most \(L\).
For \(0<a<b<\infty\), let
\begin{equation}
\label{rwls:eq:rescaled-metric-crossing}
\chi_{a,b}^{\mathbb X,\alpha}(\delta)
 :=
 \sup_{r\ge2}
 \Pp_\alpha\bigl(
 \mathsf{S}_{ar}
 \xleftrightarrow{\mathcal L_{\mathbb X}^{\alpha,\le\delta r}}
 \mathsf{S}_{br}\bigr).
\end{equation}
We are concerned with the following regime,
defined in analogy with \eqref{eq:I0},
\begin{equation}\label{eq:I0X}
\mathrm{I}_0^{\mathbb X}:=\left\{\alpha>0:
 \lim_{\delta\downarrow0}\chi_{a,b}^{\mathbb X,\alpha}(\delta)=0
 \text{ for every }0<a<b<\infty\right\},
\end{equation}
and the corresponding threshold is set as
\(\alpha_\ast^{\mathbb X}:=\sup\mathrm{I}_0^{\mathbb X}\).
On the metric graph, we can identify \(\alpha_\ast^{\widetilde{\mathbb Z}^{d}}\) with
the percolation threshold for \(d\in\{3,4,5\}\).

\begin{theorem}
\label{thm:metric-critical-values}
For every \(d\in\{3,4,5\}\),
$
 \alpha_\ast^{\widetilde{\mathbb Z}^{d}}
 =\alpha_{\rm c}^{\widetilde{\mathbb Z}^{d}}
 =\frac12
$.
\end{theorem}

The upper bound \(\alpha_\ast^{\widetilde{\mathbb Z}^{d}}\le\frac12\)
follows from~\cite{CDL24}.
For the lower bound, we combine the OSSS framework used in
Theorem~\ref{thm:unit-truncated-sharpness} with the upper bound for the critical
crossing probability in~\cite[Theorem~1.2]{CD24}.
We believe Theorem~\ref{thm:metric-critical-values} also holds for all $d\ge 6$
although the current proof strategy does not work.

Since the discrete loop soup is dominated by the metric-graph loop soup, we obtain
\(\alpha_\ast^{\mathbb Z^d}\ge\frac12\) for
\(d\in\{3,4,5\}\).  On the other hand,
\cite[Theorem~1.1]{Cha17} implies
\(\alpha_\ast^{\mathbb Z^d}\le\alpha_{\rm c}^{\mathbb Z^d}\).  Therefore,
we have \(\frac12\le\alpha_\ast^{\mathbb Z^d}
\le\alpha_{\rm c}^{\mathbb Z^d}\) for
\(d\in\{3,4,5\}\).

For \(d\in\{3,4,5\}\) and \(r\ge2\), put
\begin{equation}
\label{eq:rwls-capacity-scale}
 a_d(r):=r\1_{\{d=3\}}+r^2(\log r)^{-1}\1_{\{d=4\}}+r^2\1_{\{d=5\}}.
\end{equation}
Let \(\operatorname{cap}_{\mathbb X}\)
denote the capacity on \(\mathbb X\), defined
precisely in \eqref{eq:capX}.
Let \(\mathcal C(0)\) be the
cluster of the origin in \(\mathcal L_{\mathbb X}^\alpha\).
The following analogue of Theorem~\ref{thm:continuum-robust-summary}
gives a lower bound for the capacity of \(\mathcal C(0)\), conditioned on the one-arm event.

\begin{theorem}
\label{thm:rwls-robust-summary}
For \(d\in\{3,4,5\}\),
\(\mathbb X\in\{\mathbb Z^d,\widetilde{\mathbb Z}^{d}\}\), and
\(\alpha\in\mathrm{I}_0^{\mathbb X}\), there are
\(c,\kappa_0>0\) such that for all \(r\ge2\),
\begin{equation}\label{eq:rwls-robust-summary}
 \Pp_\alpha\Bigl(
 \mathcal C(0)\cap \mathsf{S}_r\neq\emptyset,
 \ \operatorname{cap}_{\mathbb X}
 \bigl(\mathcal C(0)\cap\mathsf{B}_{2r}\bigr)\ge\kappa_0a_d(r)
 \Bigr)
 \ge c\,p_r^{\mathbb X,\alpha}.
\end{equation}
\end{theorem}

The proof uses the same recursive argument as in the proof of
Theorem~\ref{thm:continuum-robust-summary}.
In dimension three, the estimate \eqref{eq:rwls-robust-summary} was previously known only for the critical metric-graph
loop soup \(\mathcal L_{\widetilde{\mathbb Z}^{3}}^{1/2}\), where the isomorphism
with the GFF plays an essential role~\cite{CD24,DPR24}.
Much less is known away from criticality or in the discrete setting,
where no exact two-point connection formula of the type in~\cite{Lup16} is available.
Note that if \eqref{eq:rwls-robust-summary} holds at \(\alpha=\frac12\),
or $\frac12\in \mathrm{I}_0^{\widetilde{\mathbb Z}^{3}}$,
then one can obtain the one-arm upper bound
\(p_r^{\widetilde{\mathbb Z}^{3},1/2}\lesssim r^{-1/2}\) immediately; see also~\cite{DW20,DPR23}.
Hence, investigating the behavior of \(\chi_{a,b}^{\widetilde{\mathbb Z}^{d},\alpha}(\delta)\)
as \(\delta\downarrow0\) is important for obtaining one-arm estimates.
In this paper, we focus on the loop soup in a unified framework without using the GFF,
and leave the critical case to future work.

We now return to dimension three.
The power-law asymptotics of the one-arm probability are known only for
\(\mathcal L_{\widetilde{\mathbb Z}^{3}}^{1/2}\), while the general case
\(\mathcal L_{\mathbb X}^{\alpha}\) remains open.
Nevertheless, the capacity estimate in Theorem~\ref{thm:rwls-robust-summary}
yields a strict comparison between one-arm probabilities at different intensities.
If these probabilities have power-law exponents, the comparison implies that those
exponents are strictly decreasing in \(\alpha\), as in Theorem~\ref{thm:xi-strict-monotone}.

\begin{theorem}[Strict comparison]
\label{thm:rwls-strict-summary}
Fix \(\mathbb X\in\{\Z^3,\ZMG\}\).  For
\(0<\alpha<\beta<\alpha_\ast^{\mathbb X}\),
there exist constants
\(\rho,c>0\) such that for all \(r\ge2\),
\begin{equation}\label{eq:rwls-strict-summary}
 p_r^{\mathbb X,\alpha}\le cr^{-\rho}p_r^{\mathbb X,\beta}.
\end{equation}
Moreover, there exist constants \(\alpha_0\in(0,\tfrac12)\) and
\(0<c<C<\infty\) such that, for both
\(\mathbb X\in\{\mathbb Z^3,\widetilde{\mathbb Z}^{3}\}\), all
\(0<\alpha\le\alpha_0\), and all \(r\ge2\),
\begin{equation}\label{rwls:eq:small-intensity-bounds}
 c\alpha r^{-1+c\alpha}
 \le p_r^{\mathbb X,\alpha}
 \le C\alpha r^{-1+C\alpha}.
\end{equation}
\end{theorem}

The comparison
\eqref{eq:rwls-strict-summary} follows from
Theorem~\ref{thm:rwls-robust-summary} and Russo's formula.
The bounds~\eqref{rwls:eq:small-intensity-bounds} are proved in
Section~\ref{rwls:sec:robust-reduction}  using the capacity recursion
from the proof of \eqref{eq:xi-0}.
The lower bound in~\eqref{rwls:eq:small-intensity-bounds} quantifies a phenomenon
already seen in \(\mathcal L_{\mathbb Z^3}^{\alpha}\) and
\(\mathcal L_{\mathbb R^3}^{\alpha}\): realizing the one-arm event with a single loop
is strictly more costly than using a chain of loops~\cite{CS16,CD25a,JL26}.
\begingroup
In fact, for the random walk loop soup,  the following bounds have already been established (see \cite[Theorems~1.3(b) and~1.5]{CS16}),
\begin{equation}\label{rwls:eq:cs16-one-arm-bounds}
 c(\alpha)r^{-1+\varepsilon(\alpha)}
 \le p_r^{\mathbb Z^3,\alpha}
 \le C(\alpha)r^{-1+C(\log\frac{1}{\alpha})^{-1}},
\end{equation}
for all sufficiently small \(\alpha>0\) and all \(r\ge2\), where
\(\varepsilon(\alpha)>0\) is non-explicit.
These estimates are now refined by~\eqref{rwls:eq:small-intensity-bounds}.
\endgroup

\subsection{Proof strategy}

We first explain the proof of Theorem~\ref{thm:continuum-robust-summary}. Fix \(\alpha\in \mathrm{I}_0\).  For \(0<a<b\) and \(0<\delta<1\), set
\begin{equation}\label{eq:fixed-ratio-crossing-prob}
        \chi_{a,b}^{\alpha}(\delta)
        :=\Pp_\alpha\Bigl(
        S_{a}
        \xleftrightarrow{\LL_{\R^3}^{\alpha,\le\delta}}
        S_{b}
        \Bigr)=
        \Pp_\alpha\Bigl(
        S_{a\delta^{-1}}
        \xleftrightarrow{\LL_{\R^3}^{\alpha,\le1}}
        S_{b\delta^{-1}}
        \Bigr).
\end{equation}
A zero--one law gives \(\lim_{\delta\to0}\chi_{a,b}^{\alpha}(\delta)=0\)
for every \(0<a<b\); see Lemma~\ref{lem:fixed-ratio-bdelta}.
Let \(g_s^{\alpha}\) be the probability that the loop soup \(\mathcal L_{B_{2s}}^\alpha\)
contains a cluster connecting \(S_1\) to \(S_s\) and containing a loop of capacity
of order \(s\).
We split the one-arm event \(\mathcal E_s\) into two cases.
If the crossing uses only loops of diameter less than \(\delta s\), then scale
invariance gives the upper bound \(\chi_{5/8,3/4}^{\alpha}(\delta)p_{s/2}^{\alpha}\).
Otherwise, the crossing contains a loop of diameter at least \(\delta s\).
Resampling this large loop bounds the probability of the second case by \(C_\delta g_s^{\alpha}\).
Combining the two cases gives the recursion
\begin{equation}\label{eq:induction}
p_s^{\alpha}\le C_\delta g_s^{\alpha}+
\chi_{5/8,3/4}^{\alpha}(\delta)p_{s/2}^{\alpha}.
\end{equation}
Together with \(g_{2r}^{\alpha}\gtrsim g_r^{\alpha}\) and
\(\chi_{a,b}^{\alpha}(\delta)\downarrow0\), this recursion yields
\(g_r^{\alpha}\gtrsim p_r^{\alpha}\) for all \(r\ge2\), proving
Theorem~\ref{thm:continuum-robust-summary}.
Indeed, the recursion \eqref{eq:induction} resembles the recursive estimates used to prove
separation lemmas for percolation and non-intersecting Brownian motions;
see e.g.\ \cite{GPS13,GNQ24}.
Now, using the welding argument in Section~\ref{imp:sec:uniform},  we obtain the supermultiplicativity of the
one-arm probability
\begin{equation}\label{eq:super}
  \alpha p_{rs}^{\alpha}\gtrsim g_{r/2}^{\alpha}g_{s/4}^{\alpha} \gtrsim p_{r/2}^{\alpha}p_{s/4}^{\alpha}.
\end{equation}

For submultiplicativity, the BK inequality reduces the remaining
estimate to the event \(\mathcal E_{rs}\setminus
(\mathcal E_r\circ\mathcal E_{4r,rs})\).
On this event, the two crossings cannot occur disjointly, so there must be a
welding loop across \(A(r,4r)\).
We split this loop into two smaller loops realizing \(\mathcal E_r\) and
\(\mathcal E_{4r,rs}\), respectively, and obtain \(\alpha p_{rs}^{\alpha}\lesssim p_r^{\alpha}p_{s/4}^{\alpha}\).
Together with \eqref{eq:super},
this yields Theorem~\ref{thm:utc}
with bounds uniform in \(0<\alpha\le\bar\alpha\); see
Section~\ref{imp:sec:uniform}.
We then use  the standard second-moment method to establish the Hausdorff dimension of all
clusters in Theorem~\ref{thm:all-cluster-dimension}.
Finally, to prove strict monotonicity of the one-arm exponent in
Theorem~\ref{thm:xi-strict-monotone}, we construct pivotal loops at dyadic scales.
The welding estimate \eqref{eq:super} gives a uniformly positive probability for
these constructions, and Russo's formula yields
\(\frac{\mathrm d}{\mathrm d\lambda}\log p_r^\lambda\gtrsim\log r\).
Integrating this inequality proves the strict monotonicity.

For Theorem~\ref{thm:continuum-thresholds}, the key additional observation is
 that, almost surely, Brownian segments can be connected by an
independent Brownian loop soup of any small positive intensity
with uniformly positive
probability. We prove this property by a second-moment argument using \(\xi<1\) in Lemma~\ref{lem:threshold-attachment}.
Then for $\beta>\alpha_{\mathrm{B2B}}$, we split \(\mathcal L_{\mathbb R^3}^{\beta}\)
 into the union of two independent loop soups \(\mathcal L_{\mathbb R^3}^{\alpha}\) and
\(\mathcal L_{\mathbb R^3}^{\varepsilon}\)
with $\alpha>\alpha_{\mathrm{B2B}}$ and $\varepsilon>0$.
Now, \(\mathcal L_{\mathbb R^3}^{\alpha}\) provides connections between
neighborhoods of any pair of loops, while \(\mathcal L_{\mathbb R^3}^{\varepsilon}\)
provides the local connections to the actual loops by the previous observation.
This proves local uniqueness and hence the
threshold identity.

Next, we turn to the unit-truncated loop soup. Let \(\theta_r(\alpha)\) be the one-arm
probability for the loop soup \(\LL_{\R^3}^{\alpha,\le1}\), and set
\(\Sigma_r(\alpha):=\sum_{j=0}^{\lfloor r\rfloor}\theta_j(\alpha)\).
Combining Russo's formula with the OSSS inequality for a
random-sphere exploration algorithm, we obtain
\[
 \frac{\dint}{\dint\alpha}\theta_r(\alpha)
 \ge \frac{c}{\alpha}\frac{r}{\Sigma_r(\alpha)}
 \theta_r(\alpha)\bigl(1-\theta_r(\alpha)\bigr).
\]
The sharpness result in Theorem~\ref{thm:unit-truncated-sharpness} then follows from
an analytic dichotomy result (Lemma~\ref{ut:lem:analytic}).
Note that there are infinitely many small loops.
To make the argument rigorous, we first impose a positive lower
diameter cutoff and then remove it, which requires some additional work.
Theorem~\ref{thm:metric-critical-values} follows from an analogous sharpness argument,
with the metric-graph estimates in~\cite{CD24,CDL24}.

The recursion \eqref{eq:induction} also holds on \(\mathbb X\in\{\Z^3,\ZMG\}\).
For \(\alpha\in\mathrm{I}_0^{\mathbb X}\), defined in \eqref{eq:I0X},
we have \(\chi_{5/8,3/4}^{\mathbb X,\alpha}(\delta)\downarrow0\), and hence the same induction
proves Theorem~\ref{thm:rwls-robust-summary}.
In dimension three, the capacity scale in \eqref{eq:rwls-capacity-scale} is linear.
We can therefore weld crossings and apply Russo's formula again to prove
Theorem~\ref{thm:rwls-strict-summary}.

\medskip
\noindent\textbf{Organization.}
Section~\ref{sec:pre} gives the preliminaries.
Section~\ref{sec:submultiplicativity} develops loop-splitting estimates and proves
submultiplicativity of the one-arm probability.
In Section~\ref{sec:robust-crossing}, we establish the capacity lower bound
(Theorem~\ref{thm:continuum-robust-summary}) and use it to prove the up-to-constants
estimates (Theorem~\ref{thm:utc}).
Section~\ref{sec:multiplicativity-to-dimension} proves the cluster dimension
(Theorem~\ref{thm:all-cluster-dimension}), strict monotonicity of the one-arm
exponent (Theorem~\ref{thm:xi-strict-monotone}),
and the equality of thresholds
(Theorem~\ref{thm:continuum-thresholds}).
Section~\ref{sec:unit-truncated-sharpness} proves sharpness for the unit-truncated
loop soup (Theorem~\ref{thm:unit-truncated-sharpness}).
Finally, Section~\ref{sec:rwls-companions} proves the results on the discrete and metric graphs.

\medskip
\noindent\textbf{Acknowledgments.}
We thank Zhenhao Cai for drawing our attention to the OSSS inequality used in the
proof of sharpness. We are grateful to Zhenhao Cai, Ruixuan Li, and Runsheng Liu
for valuable discussions on the Hausdorff dimension of clusters in three dimensions, and to Gefei Cai, Jian Ding,
Xinyi Li, and Baojun Wu for helpful discussions and comments.

\emph{AI use statement.} AI tools were used to refine the exposition, polish the
language, and assist with proofreading and figure generation. The author takes full responsibility for
the content of this paper.

\section{Preliminaries}
\label{sec:pre}

We begin by fixing the normalization of the Brownian loop measure used
throughout.  Let \(D\subset\R^3\) be a domain.  Write
\[
        p_t(x,y)
        =
        \frac{1}{(2\pi t)^{3/2}}
        \exp\left(-\frac{|x-y|^2}{2t}\right)
\]
for the three-dimensional heat kernel, and let \(\Pbb_t^{x,x}\) be the law of
a Brownian bridge from \(x\) to \(x\) of duration \(t\).  The Brownian loop
measure on \(D\) is
\[
        \mu_D^{\Loop}(d\ell)
        =
        \int_0^\infty \int_D
        \frac{1}{t}\,p_t(x,x)
        \Pbb^{x,x}_t(d\ell;\ell\subset D)
        \,dx\,dt,
\]
regarded as a measure on loops modulo circular time-shift.  We use the same
notation \(\ell\) for a parametrized representative and its time-shift class.

For \(\alpha>0\), let \(\LL_D^\alpha\) be the Brownian loop soup on \(D\),
namely the Poisson point process on loops with intensity measure
\(\alpha\mu_D^{\Loop}\).

We next recall Brownian path measures, capacity, Poisson inequalities, Palm's
formula, and the two Russo formulae used below; see
\cite{LW04,HLLS22,JL26,GLL26}.
For $x,y\in D$, let
$p_t^D(x,y)$ be the transition density of Brownian motion killed on exiting
$D$, and let $\Pbb^{x,y}_{t,D}$ denote the corresponding Brownian bridge from
$x$ to $y$ in time $t$, conditioned to stay in $D$.  The Brownian path measure
from $x$ to $y$ in $D$ is
\[
        \nu^D_{x,y}(d\omega)
        :=
        \int_0^\infty
        p_t^D(x,y)\,
        \Pbb^{x,y}_{t,D}(d\omega)
        \,dt .
\]
When
$D=\R^3$ we write $\nu_{x,y}:=\nu^{\R^3}_{x,y}$.
We write \(\lvert\nu^D_{x,y}\rvert\) for the total mass of the path measure.
For \(x,y\in D\), this mass is given by Green's function in \(D\),
\[
        \lvert\nu^D_{x,y}\rvert
        =
        G_D(x,y)
        :=
        \int_0^\infty p_t^D(x,y)\,dt .
\]
Thus \(G_D(x,y)<\infty\) for \(x\ne y\), whereas \(G_D(x,x)=\infty\).
When \(D=\R^3\), we write \(G(x,y):=G_{\R^3}(x,y)=\frac{1}{2\pi |x-y|}\).

If \(x\in D\) and \(y\) lies on a smooth piece of \(\partial D\), let
\(n_y\) be the inward unit normal at \(y\).  We use the same notation for the
bulk-to-boundary Brownian path measure
\begin{equation}\label{eq:bulk-to-boundary-path-measure}
        \nu^D_{x,y}
        :=
        \lim_{\varepsilon\downarrow0}
        \frac{1}{2\varepsilon}\,
        \nu^D_{x,y+\varepsilon n_y}.
\end{equation}
Its total mass is the Poisson kernel \(H_D(x,y)\).  By time reversal, the
same convention applies when \(x\in\partial D\) and \(y\in D\).

If \(x,y\) are distinct and lie on smooth pieces of \(\partial D\), we use
\(\nu^D_{x,y}\) for the boundary-to-boundary Brownian path measure (or excursion measure),
obtained by taking the corresponding inward-normal limits at both endpoints.
The total mass of this excursion measure is given by the boundary Poisson
kernel in \(D\), still denoted by \(H_D(x,y)\).
Whether \(x,y\) lie in \(D\) or on \(\partial D\) will always be clear from
the context, so this notation causes no ambiguity.

It is useful to collect some facts from classical potential theory \cite{MP10,PS78}.
If \(K\subset D\) is compact, write \(\mathcal P(K)\) for the set of Borel probability
measures supported on \(K\).  The relative capacity of \(K\) in \(D\) is
\begin{equation}\label{eq:relative-capacity}
        \Capa_D(K)
        :=
        \left[
        \inf_{\nu\in\mathcal P(K)}
        \int_K\int_K G_D(x,y)\,\nu(dx)\nu(dy)
        \right]^{-1}.
\end{equation}
When \(\Capa_D(K)>0\), the relative
equilibrium measure \(e_K^D\) satisfies
\begin{equation}\label{eq:relative-equilibrium-measure-hitting}
        e_K^D(K)=\Capa_D(K),
        \qquad
        \Pp_x(\tau_K<\tau_{\partial D})
        =
        \int_K G_D(x,y)\,e_K^D(dy),
        \quad x\in D\setminus K .
\end{equation}
The above also holds when \(D=\R^3\), with \(\tau_{\partial D}=\infty\),
and we will simply write \(\Capa(K)\) and \(e_K\) in this case.
From \eqref{eq:relative-capacity}, it follows immediately that
$
    \Capa(\lambda K)=\lambda\,\Capa(K)
$ for all \(\lambda>0\), and \(\Capa_D(K)\ge\Capa(K)\) since \(G_D\le G\).

\subsection{Useful properties of the Brownian loop soup}
The Brownian loop measure satisfies the restriction property and conformal invariance: if \(D'\subset D\), then \(\mu^{\Loop}_{D'}\) is the restriction
of \(\mu^{\Loop}_D\) to loops contained in \(D'\), while for every conformal
map \(f:D\to\widetilde D\), the image of \(\mu^{\Loop}_D\) under
\(\ell\mapsto f(\ell)\) is \(\mu^{\Loop}_{\widetilde D}\).

The following standard facts hold for general Poisson point processes, but we
state them only for the Brownian loop soup.
We partially order loop configurations by inclusion, counting multiplicities.
An event \(A\) is called increasing if \(\mathcal L\in A\) and
\(\mathcal L\subseteq\mathcal L'\) imply \(\mathcal L'\in A\).
The following lemma is the BK inequality for Poisson point processes~\cite{VDB96}, applied to the Brownian loop soup.

\begin{lemma}
\label{lem:poisson-bk}
Let \(D\subset\R^3\) be a domain, and let \(\LL_D^\alpha\) be a Brownian
loop soup on \(D\).  If \(A\) and \(B\) are increasing events depending on
\(\LL_D^\alpha\), and \(A\circ B\) denotes their disjoint occurrence, then
\[
        \Pp_\alpha(A\circ B)\le \Pp_\alpha(A)\,\Pp_\alpha(B).
\]
\end{lemma}

The FKG inequality for Poisson point processes
\cite[Theorem~20.4]{LP17} yields the following positive association property.

\begin{lemma}
\label{lem:poisson-fkg}
Let \(D\subset\R^3\) be a domain, and let \(\LL_D^\alpha\) be a Brownian
loop soup on \(D\).  If \(A\) and \(B\) are increasing events depending on
\(\LL_D^\alpha\), then
\[
        \Pp_\alpha(A\cap B)\ge \Pp_\alpha(A)\,\Pp_\alpha(B).
\]
\end{lemma}

By complementation, the FKG inequality also holds for decreasing events.
We next record Palm's formula for loop soups.

\begin{lemma}
\label{lem:loop-soup-palm}
Let \(D\subset\R^3\) be a domain, and let \(\LL_D^\alpha\) be a Brownian
loop soup on \(D\) with intensity \(\alpha\mu_D^{\Loop}\).  For every
nonnegative measurable function \(F\) of a loop \(\gamma\) and a loop
configuration,
\begin{equation}
        \Ee\left[
        \sum_{\gamma\in\LL_D^\alpha}
        F(\gamma,\LL_D^\alpha)
        \right]
        =
        \alpha\int
        \Ee\left[
        F(\gamma,\LL_D^\alpha\cup\{\gamma\})
        \right]
        \mu_D^{\Loop}(d\gamma).
        \label{eq:loop-soup-palm}
\end{equation}
Consequently, if \(A(\gamma,\mathcal L)\) is a measurable event and
$\Pp_\alpha\bigl(A(\gamma,\LL_D^\alpha\cup\{\gamma\})\bigr)=1$ for \(\mu_D^{\Loop}\)-almost every \(\gamma\), then
\begin{equation}\label{eq:palm}
        \Pp_\alpha\bigl(
        A(\gamma,\LL_D^\alpha)\text{ holds for every }
        \gamma\in\LL_D^\alpha
        \bigr)=1 .
\end{equation}
\end{lemma}

\begin{proof}
This is the Mecke equation~\cite[Theorem~4.1]{LP17} for the Poisson point process
\(\LL_D^\alpha\) on loop space with intensity
\(\alpha\mu_D^{\Loop}\).  The final assertion follows by applying
\eqref{eq:loop-soup-palm} to \(F=\mathbf 1_{A^c}\).
\end{proof}

\subsection{Russo's formulae for loop soups}

Throughout this subsection, fix
\(\mathbb X\in\{\R^3,\Z^3,\ZMG\}\).  When \(\mathbb X=\R^3\), set
\(\mu_{\mathbb X}^{\Loop}:=\muLoop\) and
\(\mathcal L_{\mathbb X}^\alpha:=\LL_{\R^3}^\alpha\); for the other two choices, use
the loop measures and loop soups introduced in
Section~\ref{sec:rwls-results}.  Thus \(\mu_{\mathbb X}^{\Loop}\) is the
corresponding unrooted loop measure, and \(\mathcal L_{\mathbb X}^\alpha\) is
the Poisson point process with intensity \(\alpha\mu_{\mathbb X}^{\Loop}\).
We regard loop configurations as counting measures and write \(\delta_\ell\)
for the unit mass at \(\ell\).

We will use two versions of Russo's formula for these three loop soups.
Closely related ones also appear in~\cite{DRT18,CDL24,JL26}.
Throughout, if \(A\) is an increasing
event and \(\ell\) is a loop, write
\begin{equation}\label{eq:pivotal}
        \operatorname{Piv}_\ell A
        :=
        \{A(\mathcal L_{\mathbb X}^\alpha+\delta_\ell)\text{ occurs but }
        A(\mathcal L_{\mathbb X}^\alpha)
        \text{ does not occur}\}.
\end{equation}
For cutoff soups, the same notation means insertion into the cutoff soup under
consideration.
Applying \cite[Equation~(19.2)]{LP17} to the Poisson point process on the loop
space of \(\mathbb X\) yields the following standard version of Russo's formula.

\begin{lemma}\label{ut:lem:russo}
Let \(\mathscr D\) be a measurable loop class with
\(\mu_{\mathbb X}^{\Loop}(\mathscr D)<\infty\), and let \(A\) be an increasing
event depending only on the restricted soup
\(\mathcal L_{\mathbb X}^\alpha|_{\mathscr D}\).  Then
\(\theta_A(\alpha):=\Pp_\alpha(A)\) is continuously differentiable and
\begin{equation}\label{ut:eq:russo}
        \frac{\dint}{\dint\alpha}\theta_A(\alpha)
        =
        \int_{\mathscr D}
        \Pp_\alpha(\operatorname{Piv}_\ell A)\,
        \mu_{\mathbb X}^{\Loop}(\dint\ell).
\end{equation}
\end{lemma}

We shall also use the following version when the event may depend on
the full loop soup but the sprinkling is confined to a finite loop-measure
class.

\begin{lemma}\label{lem:restricted-russo}
Let \(\mathscr D\) be a measurable loop class with
\(\mu_{\mathbb X}^{\Loop}(\mathscr D)<\infty\), and let \(A\) be an increasing
loop-soup event.  At every differentiability point of
\(\theta(\alpha):=\Pp_\alpha(A)\),
\begin{equation}\label{eq:restricted-russo}
        \theta'(\alpha)
        \ge
        \int_{\mathscr D}
        \Pp_\alpha(\operatorname{Piv}_\ell A)\,
        \mu_{\mathbb X}^{\Loop}(\dint\ell).
\end{equation}
\end{lemma}

\begin{proof}
Let \(\widehat{\mathcal L}^{\varepsilon}_{\mathbb X,\mathscr D}\) be an
independent loop soup restricted to \(\mathscr D\), with intensity
\(\varepsilon\mu_{\mathbb X}^{\Loop}|_{\mathscr D}\).  Sprinkling only inside
\(\mathscr D\) is
dominated by increasing the whole soup intensity from \(\alpha\) to
\(\alpha+\varepsilon\).  Hence,
\[
\begin{aligned}
        \theta(\alpha+\varepsilon)-\theta(\alpha)
        &\ge
        \Pp_\alpha\bigl(A(\mathcal L_{\mathbb X}^\alpha+
        \widehat{\mathcal L}^{\varepsilon}_{\mathbb X,\mathscr D})
        \setminus A(\mathcal L_{\mathbb X}^\alpha)\bigr)\\
        &\ge
        e^{-\varepsilon\mu_{\mathbb X}^{\Loop}(\mathscr D)}
        \varepsilon
        \int_{\mathscr D}
        \Pp_\alpha(\operatorname{Piv}_\ell A)\,
        \mu_{\mathbb X}^{\Loop}(\dint\ell).
\end{aligned}
\]
Dividing by \(\varepsilon\) and letting \(\varepsilon\downarrow0\) yields the desired
lower bound.
\end{proof}

\section{Splitting and submultiplicativity}
\label{sec:submultiplicativity}

This section proves the submultiplicativity of the one-arm probability.
The proof combines a splitting lemma with the BK inequality.

\subsection{Splitting lemmas}
We first present some useful surgeries on Brownian loops in three dimensions.
Several two-dimensional analogs or similar ideas also appeared in \cite{GNQ24,GNQ25},
which are concerned with discrete and continuous models, respectively.

To derive these results, we will combine the crossing-loop decomposition from
\cite[Proposition~4.7]{JL26} with standard properties of Brownian path measures.
For compact sets $K_1,K_2\subset D$, define the total mass of loops in $D$ that intersect both $K_1$ and $K_2$ by
\begin{equation}\label{eq:loop-cut-def}
        \Bcal(K_1,K_2;D)
        :=
        \mu^{\Loop}_D
        \bigl(
        \ell:\ell\cap K_1\ne\varnothing,\,
        \ell\cap K_2\ne\varnothing
        \bigr).
\end{equation}
Let \(\sigma_r\) be the surface measure on the sphere \(S_r=\partial B_r\).
The following lemma cuts a loop at fixed spheres, retains the pieces that hit
prescribed sets, and turns loop measure into products of Brownian path measures.

\begin{lemma}\label{lem:loop-cut}
For every \(b>1\), uniformly over \(r>0\) and \(K\subset B_r\),
\begin{equation}\label{eq:loop-cut-one-set}
        \Bcal(K,S_{br};\R^3)
        \asymp r^{-2}\int_{S_{br}}\int_{S_r}
        \nu^{B_{br}}_{x,y}\bigl(\omega\cap K\ne\varnothing\bigr)\,
        \sigma_r(dx)\sigma_{br}(dy),
\end{equation}
where the implicit constants depend only on \(b\).
Moreover, if $K_1\subset B_r$ and $K_2\subset B_{br}^c$, then for any $1<a_1<a_2<b$,
\begin{equation}\label{eq:loop-cut-two-set}
\begin{aligned}
        \Bcal(K_1,K_2;\R^3)
        \asymp r^{-4}\int_{S_{a_2r}}\int_{S_{a_1r}}&
        \nu^{B_{a_2r}}_{x_1,y_1}\bigl(\omega\cap K_1\ne\varnothing\bigr)\,
        \sigma_{a_1r}(dx_1)\sigma_{a_2r}(dy_1)\\
        &  \times\,\int_{S_{a_1r}}\int_{S_{a_2r}} \nu^{B_{a_1r}^c}_{y_2,x_2}\bigl(\omega\cap K_2\ne\varnothing\bigr)\,\sigma_{a_2r}(dy_2)\sigma_{a_1r}(dx_2),
\end{aligned}
\end{equation}
where the implicit constants depend only on \(a_1,a_2,b\).
\end{lemma}

\begin{proof}
By Brownian scaling, it is enough to prove the estimates for \(r=1\).  We
therefore work in the fixed annulus \(A(1,b)\).  We recall
\cite[Proposition~4.7]{JL26} in this fixed-annulus case, using our notation.
Let \(\mathrm{Cross}_{1,b}^{(n)}\) be the event that a loop crosses
\(A(1,b)\) exactly \(n\) times.  Then, for every
nonnegative measurable function \(F\) on the loop space, writing
\(x_{n+1}=x_1\) and
\begin{equation}\label{eq:loop-cut-Qj}
        \mathcal Q_j
        :=
        \int \nu^{B_b}_{x_j,y_j}(de_{2j-1})
        \int \nu^{B_1^c}_{y_j,x_{j+1}}(de_{2j}),
\end{equation}
we have
\begin{equation}
\begin{aligned}
        &\int F(\ell)\,
        \mathbf 1_{\mathrm{Cross}_{1,b}^{(n)}}(\ell)\,
        \muLoop(d\ell)                                                   \\
        &\quad =
        \frac1n
        \int_{(S_b)^n}\int_{(S_1)^n}
        \prod_{j=1}^n \mathcal Q_j\,
        F([e_1\oplus\cdots\oplus e_{2n}]) \,
        \prod_{j=1}^n\sigma_1(dx_j)\sigma_b(dy_j),
        \label{eq:JL-prop-47}
\end{aligned}
\end{equation}
where the measures in \eqref{eq:loop-cut-Qj} are the bulk-to-boundary Brownian path measures
given in \eqref{eq:bulk-to-boundary-path-measure}.
In our convention, the odd pieces stay in
\(B_b\), going from \(S_1\) to \(S_b\), and the even pieces stay in
\(B_1^c\), going from \(S_b\) to \(S_1\).

According to \eqref{eq:JL-prop-47}, the total mass of loops that connect $K$ and $S_b$,
and cross \(A(1,b)\) exactly once, is given by
\begin{equation}\label{eq:cross-lower-bound}
        \int_{S_b}\int_{S_1}
        \nu^{B_b}_{x,y}(\omega\cap K\ne\varnothing)\,
        H_{B_1^c}(y,x)\,
        \sigma_1(dx)\sigma_b(dy),
\end{equation}
which gives a natural lower bound for $\Bcal(K,S_b;\R^3)$.
Since \(b\) is fixed, the Poisson kernel
 \(H_{B_1^c}(y,x)\) is uniformly bounded below by a positive constant
for \(x\in S_1\), \(y\in S_b\).  Hence this one-crossing contribution gives
 the lower bound in \eqref{eq:loop-cut-one-set}.

For the upper bound, decompose each loop counted by \(\Bcal(K,S_b;\R^3)\)
according to its number \(n\ge1\) of crossings of
\(A(1,b)\). Note that at least
one of the odd-indexed pieces must hit \(K\).  We use a union bound to mark one
such piece and then apply \eqref{eq:JL-prop-47}.  Summing over the \(n\)
possible locations of the marked piece cancels the prefactor \(\frac1n\).  By
cyclic symmetry, we may take the
marked piece to be \(e_1\).  It remains to verify that the total mass of the
unmarked crossing pieces is summable, that is,
\begin{equation}\label{eq:remaining-crossings-summable}
\sum_{n\ge1}\,\int_{(S_b)^{n-1}}\int_{(S_1)^{n-1}}
        \prod_{j=1}^{n-1} {H_{B_1^c}(y_j,x_{j+1})}\, H_{B_b}(x_{j+1},y_{j+1})\,
               H_{B_1^c}(y_n,x_1) \,
        \prod_{j=1}^{n-1}\sigma_1(dx_{j+1})\sigma_b(dy_{j+1})
\end{equation}
is finite. For \(n=1\), the summand reduces to \(H_{B_1^c}(y_1,x_1)\).
We first bound $H_{B_1^c}(y_n,x_1)$ by a constant $C_b<\infty$ for any $x_1\in S_1$, $y_n\in S_b$.
Then using the following two identities
\[
        \int_{S_b} H_{B_b}(x_{j+1},y_{j+1})\,\sigma_b(dy_{j+1})=1,
        \qquad x_{j+1}\in S_1,
\]
and
\[
        {\int_{S_1} H_{B_1^c}(y_j,x_{j+1})\,\sigma_1(dx_{j+1})}
        =
        \Pp_{y_j}(\tau_{S_1}<\infty)
        =
        \frac1b,
        \qquad y_j\in S_b,
\]
we conclude that \eqref{eq:remaining-crossings-summable} is bounded by $C_b\sum_{n\ge1} b^{-(n-1)}<\infty$.
This implies  the upper bound in \eqref{eq:loop-cut-one-set}.

The two-set estimate is obtained in the same way, with two marked pieces.
However, in order to decouple the two marked pieces, we leave some space in between to apply the Harnack principle.
Since the strategy is quite similar to the one-set case, we only give a sketch of the proof.
First, we decompose a loop in \(\mathrm{Cross}_{a_1,a_2}^{(n)}\) as in \eqref{eq:JL-prop-47}, and mark one \(B_{a_2}\)-piece hitting \(K_1\) and one \(B_{a_1}^c\)-piece hitting \(K_2\).
Then, by the Harnack principle,
\[
\nu^{B_{a_2}}_{x,y}(\omega\cap K_1\ne\varnothing) \asymp
\int_{S_{a_2}}\int_{S_{a_1}}
        \nu^{B_{a_2}}_{x,y}\bigl(\omega\cap K_1\ne\varnothing\bigr)\,
        \,
        \sigma_{a_1}(dx)\sigma_{a_2}(dy),
\]
\[
\nu^{B_{a_1}^c}_{y,x}\bigl(\omega\cap K_2\ne\varnothing\bigr)\asymp \int_{S_{a_1}}\int_{S_{a_2}} \nu^{B_{a_1}^c}_{y,x}\bigl(\omega\cap K_2\ne\varnothing\bigr)\,\sigma_{a_2}(dy)\sigma_{a_1}(dx).
\]
Note that the remaining crossings are summable for the same reason as before.
This concludes the proof of \eqref{eq:loop-cut-two-set} for \(r=1\).
Scaling back gives the result for general $r>0$. We thus finish the proof of Lemma~\ref{lem:loop-cut}.
\end{proof}

The following estimates are consequences of Lemma~\ref{lem:loop-cut}.

{
\begin{lemma}
\label{lem:two-capacity-welding}
Uniformly for \(r>0\) and compact sets \(K\subset B_r\),
\begin{equation}\label{eq:splitting-capacity}
 \Bcal(K,S_{3r};\R^3)\asymp\frac{\Capa(K)}r.
\end{equation}
Moreover, if \(K_1\subset A(\frac r4,r)\) and
\(K_2\subset A(2r,8r)\) are compact, then
\begin{equation}\label{eq:two-capacity-welding}
 \Bcal(K_1,K_2;\R^3)
 \gtrsim r^{-2}\,\Capa(K_1)\,\Capa(K_2).
\end{equation}
\end{lemma}

\begin{proof}
By Brownian scaling, it is enough to take \(r=1\).
Fix \(1<a_1<a_2\).
By \eqref{eq:bulk-to-boundary-path-measure} and \eqref{eq:relative-equilibrium-measure-hitting},
we claim that uniformly in \(K\subset B_1\) and \(x\in S_{a_1}\),
\begin{equation}\label{eq:inner-path-capacity}
 \int_{S_{a_2}}\nu^{B_{a_2}}_{x,y}
       (\omega\cap K\ne\varnothing)\,\sigma_{a_2}(dy)
 =\Pp_x(\tau_K<\tau_{S_{a_2}})
 =\int_K G_{B_{a_2}}(x,z)\,e_K^{B_{a_2}}(dz)
 \asymp\Capa(K).
\end{equation}
For the lower bound \(\Pp_x(\tau_K<\tau_{S_{a_2}})\gtrsim\Capa(K)\)
in \eqref{eq:inner-path-capacity}, note that \(G_{B_{a_2}}(x,z)\) is uniformly bounded below on
\(S_{a_1}\times B_1\), and
\(\Capa_{B_{a_2}}(K)\ge\Capa(K)\).
For the upper bound, we use
\(\Pp_x(\tau_K<\tau_{S_{a_2}})\le\Pp_x(\tau_K<\infty)\)
and \eqref{eq:relative-equilibrium-measure-hitting}.
Taking \(a_1=2\), \(a_2=3\), integrating over the starting point
\(x\in S_2\), and applying
\eqref{eq:loop-cut-one-set} proves \eqref{eq:splitting-capacity}.

For \eqref{eq:two-capacity-welding}, fix \(1<a_1<a_2<2\) and apply
\eqref{eq:loop-cut-two-set} with \(b=2\).
The first double integral in \eqref{eq:loop-cut-two-set} counts paths
that start on \(S_{a_1}\), hit \(K_1\), and stop at their first exit from
\(B_{a_2}\).  It is comparable to \(\Capa(K_1)\) by
\eqref{eq:inner-path-capacity}, integrated over \(x\in S_{a_1}\).
The second double integral counts paths that start on \(S_{a_2}\),
hit \(K_2\) before their first hit of \(S_{a_1}\), and stop there.
The strong Markov property at \(\tau_{K_2}\) gives, uniformly for
the starting point \(y\in S_{a_2}\),
\[
\begin{aligned}
 \int_{S_{a_1}}\nu^{B_{a_1}^c}_{y,x}
       (\omega\cap K_2\ne\varnothing)\,\sigma_{a_1}(dx)
 &=\Pp_y(\tau_{K_2}<\tau_{S_{a_1}}<\infty)\\
 &\gtrsim\Pp_y(\tau_{K_2}<\tau_{S_{a_1}})
 =\int_{K_2}G_{B_{a_1}^c}(y,z)\,e_{K_2}^{B_{a_1}^c}(dz)
 \gtrsim\Capa(K_2).
\end{aligned}
\]
The last inequality follows from the uniform lower bound on
\(G_{B_{a_1}^c}\) on \(S_{a_2}\times A(2,8)\) and
\(\Capa_{B_{a_1}^c}(K_2)\ge\Capa(K_2)\).
Integrating over the starting point \(y\in S_{a_2}\) and substituting
the bounds for both double integrals into \eqref{eq:loop-cut-two-set} proves
\eqref{eq:two-capacity-welding} at unit scale.
Scale invariance of Brownian loop measure and
\(\Capa(rK)=r\Capa(K)\) give both estimates for general \(r>0\).
\end{proof}
}

We now turn a loop joining an inner set to an outer set into two local
loops, by closing each of its marked pieces in a fixed annulus.

\begin{lemma}
\label{lem:splitting}
Uniformly for \(r>0\) and compact sets
\(K_1\subset B_r\) and \(K_2\subset B_{4r}^c\),
\begin{equation}\label{eq:splitting-lemma}
\begin{aligned}
 \Bcal(K_1,K_2;\R^3)
 &\asymp\Bcal(K_1,S_{3r};B_{4r})\,
        \Bcal(S_{2r},K_2;B_r^c)\\
 &\asymp\Bcal(K_1,S_{3r};\R^3)\,
        \Bcal(S_{2r},K_2;\R^3).
\end{aligned}
\end{equation}
\end{lemma}

\begin{proof}
By Brownian scaling, it is enough to take \(r=1\). We write
\[
 J(K_2):=
 \int_{S_2}\int_{S_3}
 \nu^{B_2^c}_{z,w}(\omega\cap K_2\ne\varnothing)\,
 \sigma_3(dz)\sigma_2(dw).
\]
Applying \eqref{eq:loop-cut-two-set} with \(a_1=2\), \(a_2=3\),
and \(b=4\), and using \eqref{eq:inner-path-capacity} for the inner
factor, gives
\begin{equation}\label{eq:splitting-integral-product}
 \Bcal(K_1,K_2;\R^3)\asymp\Capa(K_1)\,J(K_2).
\end{equation}
It remains to compare these factors with the two local loop masses.

For the inner factor, use the one-crossing construction in
\eqref{eq:cross-lower-bound} with cutting spheres \(S_2,S_3\).
Close each path from \(x\in S_2\) to \(y\in S_3\) that hits \(K_1\)
by a path from \(y\) to \(x\) in \(A(2,4)\).
 The closing Poisson kernel \(H_{A(2,4)}(y,x)\) is uniformly bounded below,
so \eqref{eq:inner-path-capacity} yields
\begin{equation}\label{eq:IK1}
 \Capa(K_1)\lesssim\Bcal(K_1,S_3;B_4).
\end{equation}
By monotonicity and \eqref{eq:splitting-capacity}, we get
\[
 \Bcal(K_1,S_3;B_4)\asymp
 \Bcal(K_1,S_3;\R^3)\asymp\Capa(K_1).
\]

For the outer factor, close each path counted by \(J(K_2)\) from
\(z\in S_3\) to \(w\in S_2\) by a path from \(w\) to \(z\) in
\(A(1,3)\).  Since \(H_{A(1,3)}(w,z)\) is uniformly bounded below,
 this construction implies that
\[
 J(K_2)\lesssim\Bcal(S_2,K_2;B_1^c).
\]
The one-set crossing decomposition with the exterior piece marked gives
\(\Bcal(S_2,K_2;\R^3)\asymp J(K_2)\).
Domain monotonicity therefore yields
\[
 \Bcal(S_2,K_2;B_1^c)\asymp
 \Bcal(S_2,K_2;\R^3)\asymp J(K_2).
\]
Combining these comparisons with \eqref{eq:splitting-integral-product}
and scaling back proves \eqref{eq:splitting-lemma}.
\end{proof}

\subsection{Submultiplicativity and locality}

Recall the definition of the one-arm probability $p_r^{\alpha}$ in \eqref{eq:one-arm-prob-def}. Next, we show that
it satisfies the submultiplicative property by using the splitting lemma.

\begin{theorem}
\label{thm:submultiplicativity}
For every \(\bar\alpha>0\), there is
\(C_{\bar\alpha}<\infty\) such that, simultaneously for
\(0<\alpha\le\bar\alpha\), \(r>1\), and \(s>4\),
\[
    \alpha\, p_{rs}^{\alpha}\le C_{\bar\alpha}\,p_r^{\alpha}\,p_{s/4}^{\alpha}.
\]
\end{theorem}

\begin{proof}[Proof of Theorem~\ref{thm:submultiplicativity}]
Fix \(\bar\alpha>0\) and let
\(0<\alpha\le\bar\alpha\). All implicit constants in this proof depend
only on \(\bar\alpha\).
Fix $r>1$ and $s>4$.  Clearly,
$
        \mathcal E_{rs}
        \subset
        \mathcal E_r\cap \mathcal E_{4r,rs}
$,
and we can further split it according to whether the two crossing events occur
disjointly or not and obtain
\[
        \mathcal E_{rs}
        \subset
        \bigl(\mathcal E_r\circ \mathcal E_{4r,rs}\bigr)
        \cup
        \bigl(\mathcal E_{rs}\setminus
        (\mathcal E_r\circ \mathcal E_{4r,rs})\bigr).
\]
For the first event, we use Lemma~\ref{lem:poisson-bk} to obtain that
\begin{equation}\label{eq:submultiplicativity-disjoint}
        \Pbb_\alpha(\mathcal E_r\circ \mathcal E_{4r,rs})
        \le
        \Pbb_\alpha(\mathcal E_r)\,\Pbb_\alpha(\mathcal E_{4r,rs}).
\end{equation}
Hence, $\Pbb_\alpha(\mathcal E_r\circ \mathcal E_{4r,rs})
        \le
        p_r^{\alpha}\,p_{s/4}^{\alpha}$ by scale invariance of the Brownian loop soup.

It remains to upper bound the non-disjoint part $\{\mathcal E_{rs}\setminus
        (\mathcal E_r\circ \mathcal E_{4r,rs})\}$.  Let $\BL_{r,4r}$ be the collection
of loops in $\LL_{\R^3}^\alpha$ that intersect both $S_r$ and $S_{4r}$.
Define the inner random set
\[
        K_1
        :=
        S_1
        \cup
        \bigl\{
        \text{all clusters of }\LL^\alpha_{B_r}
        \text{ that intersect }S_1
        \bigr\},
\]
and the outer random set
\[
        K_2
        :=
        S_{rs}
        \cup
        \bigl\{
        \text{all clusters of }\LL^\alpha_{B_{4r}^c}
        \text{ that intersect }S_{rs}
        \bigr\}.
\]
The loop family $\BL_{r,4r}$ is independent of the loop collections used to
construct $K_1$ and $K_2$.  Moreover \(K_1\subset B_r\) and
\(K_2\subset B_{4r}^c\) by construction.
On the event
$
        \mathcal E_{rs}\setminus
        (\mathcal E_r\circ \mathcal E_{4r,rs}),
$
the crossing from $S_1$ to $S_{rs}$ cannot be realized by
two disjoint subfamilies producing the events $\mathcal E_r$ and
$\mathcal E_{4r,rs}$.
Consequently, there must be at least one loop in $\BL_{r,4r}$ that connects
the inner set \(K_1\) to the outer set \(K_2\).
Thus,
\begin{equation}\label{eq:submult-bridge-event-inclusion}
        \mathcal E_{rs}\setminus
        (\mathcal E_r\circ \mathcal E_{4r,rs})
        \subset
        \{
        \exists \ell\in\BL_{r,4r}:
        \ell\cap K_1\ne\varnothing,\,
        \ell\cap K_2\ne\varnothing
        \}.
\end{equation}
Since any loop intersecting both \(K_1\)
and \(K_2\) must intersect both \(S_r\) and \(S_{4r}\), we
have
\begin{equation}\label{eq:submult-bridge-poisson}
\begin{aligned}
        \Pbb_\alpha
        \bigl(
        \exists \ell\in\BL_{r,4r}:
        \ell\cap K_1\ne\varnothing,\,
        \ell\cap K_2\ne\varnothing
        \,\big|\,K_1,K_2
        \bigr)
        &=
        1-\exp\left\{
        -\alpha\Bcal(K_1,K_2;\R^3)
        \right\}                                          \\
        &\le
        \alpha\Bcal(K_1,K_2;\R^3).
\end{aligned}
\end{equation}
Applying Lemma~\ref{lem:splitting}, we get
\begin{equation}\label{eq:submult-splitting-bound}
        \Bcal(K_1,K_2;\R^3)
        \lesssim
        \Bcal(K_1,S_{3r};B_{4r})\,
        \Bcal(S_{2r},K_2;B_{r}^c).
\end{equation}
Conditionally on \(K_1\), the loops in \(B_{4r}\)
 that intersect both \(K_1\) and \(S_{3r}\) form a Poisson process with total intensity
\(\alpha\Bcal(K_1,S_{3r};B_{4r})\), which is uniformly bounded.
For \(0\le m\le M\),
\(e^{-\bar\alpha M}\alpha m\le1-e^{-\alpha m}\le\alpha m\).
Consequently,
$\Bcal(K_1,S_{3r};B_{4r})\lesssim \alpha^{-1}\Pbb_\alpha
        \bigl(
        K_1
        \xleftrightarrow{\LL^\alpha_{B_{4r}}}
        S_{3r}
        \,\big|\,K_1
        \bigr)$.
The same observation
applies to \(\Bcal(S_{2r},K_2;B_{r}^c)\).
Combining these with \eqref{eq:submult-bridge-poisson} and \eqref{eq:submult-splitting-bound} gives
\[
\begin{aligned}
        &\Pbb_\alpha
        \bigl(
        \exists \ell\in\BL_{r,4r}:
        \ell\cap K_1\ne\varnothing,\,
        \ell\cap K_2\ne\varnothing
        \,\big|\,K_1,K_2
        \bigr)                                            \\
        &\qquad\lesssim \alpha^{-1}
        \Pbb_\alpha
        \bigl(
        K_1
        \xleftrightarrow{\LL^\alpha_{B_{4r}}}
        S_{3r}
        \,\big|\,K_1
        \bigr)
        \,
        \Pbb_\alpha
        \bigl(
        S_{2r}
        \xleftrightarrow{\LL^\alpha_{B_r^c}} K_2
        \,\big|\,K_2
        \bigr).
\end{aligned}
\]
Integrating over $K_1$ and $K_2$ and plugging back into \eqref{eq:submult-bridge-event-inclusion} yields
\[
        \Pbb_\alpha\bigl(\mathcal E_{rs}\setminus
        (\mathcal E_r\circ \mathcal E_{4r,rs})\bigr)
        \lesssim \alpha^{-1}
        \Pbb_\alpha\bigl(S_1\xleftrightarrow{\LL^\alpha_{B_{4r}}}S_{3r}\bigr)\,
        \Pbb_\alpha\bigl(S_{2r}\xleftrightarrow{\LL^\alpha_{B_r^c}}S_{rs}\bigr)
        \lesssim \alpha^{-1}p_r^{\alpha}\,p_{s/4}^{\alpha},
\]
where we have used the facts that
\(\{S_1\xleftrightarrow{\LL^\alpha_{B_{4r}}}S_{3r}\}\subset\mathcal E_r\),
\(\{S_{2r}\xleftrightarrow{\LL^\alpha_{B_r^c}}S_{rs}\}
\subset\mathcal E_{4r,rs}\), and scale invariance, which gives
\(\Pbb_\alpha(S_{2r}\xleftrightarrow{\LL^\alpha_{B_r^c}}S_{rs})\le p_{s/4}^{\alpha}\).
 Combining this with \eqref{eq:submultiplicativity-disjoint}
and \(\alpha\le\bar\alpha\) gives the desired bound.
\end{proof}

Using the same splitting strategy, we can prove the following locality property
for the one-arm probability, which will be used later.

\begin{lemma}\label{lem:local}
        For all $\alpha>0$, $r>1$, and $D\supset B_{2r}$,
        \[
        \Pp_\alpha(S_1\xleftrightarrow{\LL^{\alpha}_D}S_r)\asymp p_r^{\alpha},
        \]
        where the implied constant only depends on $\alpha$.
\end{lemma}

\begin{proof}
        We only need to show that there exists a constant $c(\alpha)$ such that
        \begin{equation}\label{eq:local}
                \Pp_\alpha\Big(\{S_1\xleftrightarrow{\LL_{\R^3}^\alpha}S_r\}\setminus\{S_1\xleftrightarrow{\LL^{\alpha}_D}S_r\}\Big)
                \le c(\alpha)\, \Pp_\alpha(S_1\xleftrightarrow{\LL^{\alpha}_D}S_r).
        \end{equation}
        Let $K$ be the union of $S_1$ and all the clusters in $\LL^{\alpha}_D$ that intersect it.
        On the event $\{S_1\xleftrightarrow{\LL_{\R^3}^\alpha}S_r\}\setminus\{S_1\xleftrightarrow{\LL^{\alpha}_D}S_r\}$, $K\subset B_r$,
        and there is a loop that is not fully contained in $D$ and intersects $K$.
        The total mass of such loops is bounded by \(\Bcal(K,S_{2r};\R^3)\).
        The inner-factor comparison in the proof of Lemma~\ref{lem:splitting},
        with cutting spheres \(S_{5r/4},S_{3r/2}\) and closing paths kept in
        \(A(5r/4,2r)\), gives
        \(\Bcal(K,S_{2r};\R^3)\lesssim\Bcal(K,S_{3r/2};B_{2r})\).
        Therefore, we can replace the previous loop with a loop in \(B_{2r}\subset D\) that connects $K$ and $S_r$ with positive probability (depending only on $\alpha$).
        This gives \eqref{eq:local} and proves the lemma.
\end{proof}

\section{Capacity bounds and one-arm estimates}
\label{sec:robust-crossing}

In this section, we prove the capacity lower bound
in Theorem~\ref{thm:continuum-robust-summary} and the up-to-constants
 estimates in Theorem~\ref{thm:utc} in the two subsections, respectively.
The key is to find a macroscopic piece of a crossing cluster
near the outer scale whose capacity is comparable to that scale;
see \eqref{eq:robust-event-def} below for a precise formulation.

\subsection{Capacity bounds for crossing clusters}\label{subsec:robust-crossing}

We prove Theorem~\ref{thm:main}, which gives a capacity
lower bound for a crossing cluster with positive conditional probability,
uniformly in the scale and for \(0<\alpha\le\bar\alpha\) with any fixed
\(\bar\alpha\in\mathrm I_0\).

\begingroup
For \(\delta,\kappa>0\), \(r\ge2\), and a domain \(D\supset B_{2r}\),
define the event
\begin{equation}\label{eq:robust-event-def}
\mathsf H_{r,D}^{\delta,\kappa}
:=
\left\{
\begin{gathered}
\text{ there is a cluster $\mathcal C$ in $\LL_D^\alpha$ across } A(1,r) \text{ which contains }
\\ J
\text{ such that }J\subset A(\frac r2,2r),\,
\diam(J)\ge\frac{\delta r}{4},\,
\Capa(J)\ge\kappa r
\end{gathered}
\right\},
\end{equation}
and abbreviate
\(\mathsf H_r^{\delta,\kappa}:=
\mathsf H_{r,\mathbb R^3}^{\delta,\kappa}\).
\endgroup

\begin{theorem}
\label{thm:main}
\begingroup
There exists \(\kappa_0\in(0,1)\) such that,
for every \(\bar\alpha\in\mathrm I_0\), there exists
\(c>0\), depending only on \(\bar\alpha\), such that, for all
\(0<\alpha\le\bar\alpha\),
\(r\ge2\), and domains
\(D\supset B_{2r}\),
\begin{equation}\label{eq:local_version}
    \Pp_\alpha\bigl(\mathsf H_{r,D}^{2,\kappa_0}\bigr)
    \ge c\,
    \Pp_\alpha\bigl(S_1\xleftrightarrow{\LL_D^\alpha}S_r\bigr).
\end{equation}
\endgroup
\end{theorem}

\begin{proof}[Proof of Theorem~\ref{thm:continuum-robust-summary}]
Fix \(\bar\alpha\in\mathrm I_0\) and
\(0<\alpha\le\bar\alpha\). Applying Theorem~\ref{thm:main} with
\(D=\mathbb R^3\), we get $\Pp_\alpha\bigl(\mathsf H_{r}^{2,\kappa_0}\bigr)
    \ge c\, p_r^{\alpha}$.
The monotonicity of capacity gives
\(\Capa(\mathcal C\cap B_{2r})\ge\Capa(J)\ge\kappa_0r\), and the
claim \eqref{eq:continuum-robust-summary} follows.
\end{proof}

We will prove Theorem~\ref{thm:main} by using a key recursion formula derived in
Proposition~\ref{lem:residual-extension}.
We begin with the following useful lemma, based on Kolmogorov's zero--one law.
\begin{lemma}\label{lem:fixed-ratio-bdelta}
Fix \(\alpha\in \mathrm{I}_0\) and \(0<a<b\).
Then \(\lim_{\delta\downarrow0}\chi_{a,b}^{\alpha}(\delta)=0\).
\end{lemma}

\begin{proof}
For \(k\ge0\), let \(X_k\) be the restriction of
\(\LL_{\R^3}^{\alpha}\) to loops with diameter in
\((2^{-(k+1)},2^{-k}]\).
By Poisson restriction, these $X_k$'s are independent.
If we define the event
\[
        E_\infty:=\bigcap_{n\ge0} \left\{
        S_a\xleftrightarrow{\LL_{\R^3}^{\alpha,\le2^{-n}}}S_b
        \right\},
\]
then continuity from above gives
$\lim_{n\to\infty}\chi_{a,b}^{\alpha}(2^{-n})
        =\Pp_\alpha(E_\infty)$.  Monotonicity in the cutoff extends the dyadic
limit to \(\delta\downarrow0\), so it suffices to show
$\Pp_\alpha(E_\infty)=0$.
Since \(\alpha\in \mathrm{I}_0\), we have \(p_r^{\alpha}\to0\). Choose \(\lambda>1\) such that \(p_\lambda^{\alpha}<1\).
Put
\(\rho:=\frac{b-a}{2\lambda}\).  Choose \(x_1,\ldots,x_N\in S_a\) so that
\(S_a\subset\bigcup_{i=1}^N B(x_i,\frac\rho2)\).  For \(1\le i\le N\), let
\[
        A_{\infty,i}:=\bigcap_{n\ge0}
        \left\{
        B(x_i,\frac\rho2)
        \xleftrightarrow{\LL_{\R^3}^{\alpha,\le2^{-n}}}S_b
        \right\}.
\]
A finite-pigeonhole argument gives
$E_\infty\subset\bigcup_{i=1}^N A_{\infty,i}$.
Translation and scale invariance now yield
\[
        \Pp_\alpha(A_{\infty,i})
        \le
        \Pp_\alpha\Bigl(
        \partial B(x_i,\rho)
        \xleftrightarrow{\LL_{\R^3}^{\alpha}}
        \partial B(x_i,\lambda\rho)
        \Bigr)
        =p_\lambda^{\alpha}<1.
\]
Each \(A_{\infty,i}\) is also measurable with respect to every diameter-tail
sigma-field \(\sigma(X_k:k\ge N)\).  Hence the tail zero--one law forces
\(\Pp_\alpha(A_{\infty,i})=0\) for every \(i\), which
gives \(\Pp_\alpha(E_\infty)=0\) and finishes the proof.
\end{proof}

We next compare the mass of macroscopic loops hitting a set $K$ to $\Capa(K)$.  For
\(\kappa>0\), \(s>0\), and compact \(K\subset B_s\), define
\begin{equation}
\label{eq:terminal-loop-class}
 \mathscr T_s^\kappa(K)
 :=
 \left\{
 \begin{gathered}
 \ell\subset B_{2s}:\ \ell\cap K\ne\varnothing\text{ and there is }
 J\subset\ell\cap A(s,2s)\\
 \text{joining }S_{5s/4}\text{ to }S_{7s/4}
 \text{ such that }\Capa(J)\ge\kappa s
 \end{gathered}
 \right\}.
\end{equation}
In particular, every such path satisfies \(\diam(J)\ge s/2\).

\begin{lemma}
\label{lem:one-sided-surgery}
There are constants \(\kappa_0,c,C>0\) such that, for every \(s\ge2\) and
compact \(K\subset B_s\),
\begin{equation}
\label{eq:terminal-loop-comparison}
 c\frac{\Capa(K)}s
 \le
 m_s(K):=\mu_{B_{2s}}^{\Loop}(\mathscr T_s^{\kappa_0}(K))
 \le
 C\frac{\Capa(K)}s.
\end{equation}
Moreover, for every \(\delta\in(0,\frac12)\), there is
 \(C_\delta<\infty\) such that, uniformly over \(s\ge2\), domains
\(D\supset B_{2s}\), and \(K\subset B_s\),
\begin{equation}
\label{eq:large-loop-terminal-comparison}
 \mu_D^{\Loop}\bigl(
 \ell\cap K\ne\varnothing,\ \diam(\ell)>\delta s
 \bigr)
 \le C_\delta\frac{\Capa(K)}s
 \le C_\delta m_s(K).
\end{equation}
\end{lemma}

\begin{proof}
We first deal with the lower bound in
\eqref{eq:terminal-loop-comparison}.
By Brownian scaling, it is enough to prove the lower bound for \(s=1\).
By the one-crossing decomposition in Lemma~\ref{lem:loop-cut}, we restrict
to loops which can be decomposed into three pieces as
\[
 \ell=[\omega^{\rm in}\oplus\omega^{\rm mid}\oplus\omega^{\rm out}],
\]
where \(\omega^{\rm in}\) is a path from \(x\in S_{9/8}\) to
\(y\in S_{5/4}\) that intersects
\(K\), \(\omega^{\rm mid}\) is a path from \(y\) to
\(z\in S_{7/4}\), and
\(\omega^{\rm out}\) is a path from \(z\) back to \(x\).  The symbol \(\oplus\) denotes chronological
concatenation of paths, and the brackets denote the resulting unrooted loop.
For the lower bound, we restrict the middle piece to
\[
 \Gamma_{\rm mid}=\Gamma_{\rm mid}(\kappa_0)
 :=\left\{\omega^{\rm mid}:\
 \omega^{\rm mid}\subset A(9/8,7/4),\
 \Capa(\omega^{\rm mid})\ge\kappa_0\right\}.
\]
When \(\omega^{\rm mid}\in\Gamma_{\rm mid}(\kappa_0)\), the loop \(\ell\)
is contained in \(B_{15/8}\subset B_2\), intersects \(K\), and contains a
path in \(A(1,2)\) joining \(S_{5/4}\) to \(S_{7/4}\) with capacity at
least \(\kappa_0\).  Hence \(\ell\in\mathscr T_1^{\kappa_0}(K)\).
Then, the one-crossing lower bound in Lemma~\ref{lem:loop-cut} yields
\begin{equation}
\label{eq:terminal-one-crossing-lower-bound}
\begin{aligned}
 m_1(K)&\ge
 \int_{S_{9/8}}\int_{S_{5/4}}
 \nu^{B_{5/4}}_{x,y}(\omega^{\rm in}\cap K\ne\varnothing)\\
 &\hspace{1cm}\times\left(
 \int_{S_{7/4}}
 \nu^{A(9/8,7/4)}_{y,z}(\Gamma_{\rm mid})
 H_{A(9/8,15/8)}(z,x)\,\sigma_{7/4}(dz)
 \right)\sigma_{5/4}(dy)\sigma_{9/8}(dx),
\end{aligned}
\end{equation}
where \(H_{A(9/8,15/8)}(z,x)\) is the total mass of
\(\omega^{\rm out}\).

For the inner piece, uniformly in \(x\in S_{9/8}\), we have
\begin{equation}
\label{eq:terminal-inner-piece}
 \int_{S_{5/4}}\nu^{B_{5/4}}_{x,y}(\omega^{\rm in}\cap K\ne\varnothing)\,\sigma_{5/4}(dy)
 =\Pp_x(\tau_K<\tau_{S_{5/4}})
 =\int_K G_{B_{5/4}}(x,w)\,e_K^{B_{5/4}}(dw)
 \ge c\Capa(K),
\end{equation}
where we use \eqref{eq:relative-equilibrium-measure-hitting},
the uniform lower bound \(G_{B_{5/4}}(x,w)\ge c\), and
\(\Capa_{B_{5/4}}(K)\ge\Capa(K)\).
For the middle piece, by \cite[Theorem~1.1]{PPS}, there are
\(\kappa_0, q>0\) such that
\begin{equation}
\label{eq:terminal-middle-piece}
 \inf_{y\in S_{5/4}}\int_{S_{7/4}}
 \nu^{A(9/8,7/4)}_{y,z}(\Gamma_{\rm mid})\,
 \sigma_{7/4}(dz)\ge q.
\end{equation}
Using the estimate
\(\inf_{z\in S_{7/4},\,x\in S_{9/8}}H_{A(9/8,15/8)}(z,x)\ge c'\),
and plugging the bounds \eqref{eq:terminal-inner-piece} and
\eqref{eq:terminal-middle-piece} into
\eqref{eq:terminal-one-crossing-lower-bound} gives
\(m_1(K)\ge cc'q\Capa(K)\).  Scaling proves the lower bound in
\eqref{eq:terminal-loop-comparison}.

We next prove
\eqref{eq:large-loop-terminal-comparison}.
Cover \(B_s\) by
\(N_\delta=O_\delta(1)\) balls \(B(x_i,\delta s/32)\), and set
\(K_i:=K\cap\overline{B(x_i,\delta s/32)}\).  If a loop of diameter larger than
\(\delta s\) intersects \(K_i\), then it also intersects
\(\partial B(x_i,\delta s/4)\).
By Lemma~\ref{lem:loop-cut} and
\eqref{eq:relative-equilibrium-measure-hitting}, we have
\[
\begin{aligned}
 \mu_D^{\Loop}
 \bigl(\ell\cap K_i\ne\varnothing,\ \diam(\ell)>\delta s\bigr)
 &\le C(\delta s)^{-2}
 \int_{S_{\delta s/16}}
 \Pp_x(\tau_{K_i}<\infty)\,\sigma_{\delta s/16}(dx)\\
 &=C(\delta s)^{-2}
 \int_{S_{\delta s/16}}\int_{K_i}
 G(x,w)\,e_{K_i}(dw)\,\sigma_{\delta s/16}(dx)\\
 &\le C(\delta s)^{-1}\Capa(K_i).
\end{aligned}
\]
Indeed, \(K_i\subset B_{\delta s/32}\), and hence
\(G(x,w)\lesssim (\delta s)^{-1}\) for
\(x\in S_{\delta s/16}\) and \(w\in K_i\), while
\(\sigma_{\delta s/16}(S_{\delta s/16})\asymp(\delta s)^2\).
Summing over \(i\), and using \(\Capa(K_i)\le\Capa(K)\) and
\(N_\delta=O_\delta(1)\), proves the first inequality in
\eqref{eq:large-loop-terminal-comparison}.  Every loop in
\(\mathscr T_s^{\kappa_0}(K)\) has diameter at least \(s/2\).  Taking
\(\delta=\frac14\) gives the upper bound in
\eqref{eq:terminal-loop-comparison}, while its lower bound gives the last
inequality in \eqref{eq:large-loop-terminal-comparison}.
We complete the proof.
\end{proof}

For \(s\ge2\), define
\begin{equation}\label{eq:seeded-cluster-def}
 \mathcal K_s^{\alpha}
 :=S_1\cup
 \bigl\{
 \text{all clusters of }\LL^\alpha_{B_s}
 \text{ that intersect }S_1
 \bigr\}.
\end{equation}
For every \(D\supset B_{2s}\), set
\begin{equation}
\label{eq:terminal-event-probability}
 g_s^{\alpha}
 :=\Pp_\alpha\bigl(
 \mathcal L_D^\alpha\cap
 \mathscr T_s^{\kappa_0}(\mathcal K_s^{\alpha})\ne\varnothing
 \bigr).
\end{equation}

\begin{lemma}
\label{lem:robust-prop}
Fix \(\bar\alpha\in\mathrm I_0\).
Uniformly for \(0<\alpha\le\bar\alpha\) and \(s\ge2\),
\begin{equation}
\label{eq:terminal-capacity-expectation}
 g_s^{\alpha}\asymp \alpha s^{-1}\Ee_\alpha[\Capa(\mathcal K_s^{\alpha})],
\end{equation}
and moreover, there is \(c_\star\in(0,1)\) such that
\begin{equation}
\label{eq:terminal-propagation}
 g_{2s}^{\alpha}\ge c_\star g_s^{\alpha}
 \quad \text{ and }  \quad  \inf_{2\le\rho<4}g_\rho^{\alpha}
 \ge c\alpha.
\end{equation}
{
In addition,
\begin{equation}\label{imp:eq:initial-scale}
 p_2^{\alpha}\le C\alpha,
 \qquad p_\rho^{\alpha}\le C'g_\rho^{\alpha}\quad(2\le\rho<4).
\end{equation}
All constants above depend only on \(\bar\alpha\).
}
\end{lemma}

\begin{proof}
The loops forming \(\mathcal K_s^{\alpha}\) and those in
\(\mathscr T_s^{\kappa_0}(\mathcal K_s^{\alpha})\) belong to disjoint loop classes.
Hence, conditionally on \(\mathcal K_s^{\alpha}\),
\[
\Pp_\alpha(\mathcal L_D^\alpha\cap
\mathscr T_s^{\kappa_0}(\mathcal K_s^{\alpha})\ne\varnothing
\mid\mathcal K_s^{\alpha})=1-\exp\{-\alpha m_s(\mathcal K_s^{\alpha})\}.
\]
Since
\(\Capa(\mathcal K_s^{\alpha})\le\Capa(B_s)\le Cs\),
Lemma~\ref{lem:one-sided-surgery} gives
\(m_s(\mathcal K_s^{\alpha})\le C\), which implies that
\(1-e^{-\alpha m_s(\mathcal K_s^{\alpha})}
\asymp\alpha m_s(\mathcal K_s^{\alpha})\) uniformly for \(0<\alpha\le\bar\alpha\).
Applying \eqref{eq:terminal-loop-comparison} and taking
expectations proves
\eqref{eq:terminal-capacity-expectation}.
The inclusion \(\mathcal K_s^{\alpha}\subset\mathcal K_{2s}^{\alpha}\), capacity monotonicity,
and the uniform comparison in \eqref{eq:terminal-capacity-expectation} give
\(g_{2s}^{\alpha}\ge c\alpha(2s)^{-1}
\Ee_\alpha[\Capa(\mathcal K_{2s}^{\alpha})]\ge c_\star g_s^{\alpha}\).
Finally, since \(S_1\subset\mathcal K_\rho^{\alpha}\) for \(2\le\rho<4\),
\eqref{eq:terminal-capacity-expectation} and capacity monotonicity give
\[
 g_\rho^{\alpha}
 \ge c\alpha\rho^{-1}\Capa(S_1)
 \ge \frac{c\alpha}{4}\Capa(S_1),
\]
uniformly for \(0<\alpha\le\bar\alpha\).
This proves \eqref{eq:terminal-propagation}.

{
To prove \eqref{imp:eq:initial-scale}, we first show that \(p_2^{\bar\alpha}<1\).
By Lemma~\ref{lem:fixed-ratio-bdelta}, choose \(\delta\in(0,1)\) such that
\[
 \chi_{1,2}^{\bar\alpha}(\delta)
 =\Pp_{\bar\alpha}\Bigl(
 S_1\xleftrightarrow{\mathcal L_{\mathbb R^3}^{\bar\alpha,\le\delta}}S_2
 \Bigr)<1.
\]
Let
$
 \mathscr D_\delta
 :=\bigl\{\ell:\ \diam(\ell)>\delta,\
 \ell\cap\overline A(1,2)\ne\varnothing\bigr\}
$.
By~\eqref{eq:large-loop-terminal-comparison}, with
\(s=4\), \(K=\overline A(1,2)\), and cutoff \(\delta/4\), we have
\(\mu_{\mathbb R^3}^{\Loop}(\mathscr D_\delta)<\infty\).
The restrictions \(\mathcal L_{\mathbb R^3}^{\bar\alpha,\le\delta}\) and
\(\mathcal L_{\mathbb R^3}^{\bar\alpha}\cap\mathscr D_\delta\) are independent.
Hence,
\[
\begin{aligned}
 1-p_2^{\bar\alpha}
 &\ge \bigl(1-\chi_{1,2}^{\bar\alpha}(\delta)\bigr)
 \Pp_{\bar\alpha}\bigl(
 \mathcal L_{\mathbb R^3}^{\bar\alpha}\cap\mathscr D_\delta=\varnothing
 \bigr)\\
 &=\bigl(1-\chi_{1,2}^{\bar\alpha}(\delta)\bigr)
 \exp\bigl\{-\bar\alpha\mu_{\mathbb R^3}^{\Loop}(\mathscr D_\delta)\bigr\}>0.
\end{aligned}
\]
For \(\alpha\le\bar\alpha/2\), superpose
\(N=\lfloor\bar\alpha/\alpha\rfloor\) independent loop soups of intensity
\(\alpha\). Then, monotonicity gives
\[
 1-(1-p_2^\alpha)^N\le p_2^{N\alpha}\le p_2^{\bar\alpha},
 \qquad
 p_2^\alpha\le\frac{-\log(1-p_2^{\bar\alpha})}{N}
 \le C_{\bar\alpha}\alpha.
\]
For \(\alpha>\bar\alpha/2\),
use \(p_2^{\alpha}\le1\le2\alpha/\bar\alpha\). Radius monotonicity and
\eqref{eq:terminal-propagation} then give \(p_\rho^{\alpha}\le C'g_\rho^{\alpha}\).
}
\end{proof}

For \(s\ge2\) and
\(D\supset B_{2s}\), define
\begin{equation}\label{eq:domain-one-arm-def}
 \mathcal E_{s,D}:=\{S_1\xleftrightarrow{\mathcal L_D^\alpha}S_s\},
 \qquad
 p_{s,D}^{\alpha}:=\Pp_\alpha(\mathcal E_{s,D}).
\end{equation}
Then, we have the following recursive formula.

\begin{proposition}[Recursion]
\label{lem:residual-extension}
For
every \(\delta\in(0,\frac18)\), there is \(C_\delta\ge1\) such that
for all \(0<\alpha\le\bar\alpha\), \(s\ge4\), and \(D\supset B_{2s}\),
\begin{equation}
\label{eq:residual-extension}
 p_{s,D}^{\alpha}
 \le C_\delta g_s^{\alpha}
 +\chi_{5/8,3/4}^{\bar\alpha}(\delta)p_{s/2,D}^{\alpha}.
\end{equation}
\end{proposition}

\begin{proof}
Let \(\mathcal L_D^{\alpha,\le\delta s}\) be the restriction of \(\mathcal L_D^\alpha\) to
loops of diameter at most \(\delta s\).
Define the event \(\mathcal E_{s,D}^{\le\delta s}:=\{S_1\xleftrightarrow{\mathcal L_D^{\alpha,\le\delta s}}S_s\}\), on
which \(\mathcal E_{s/2,D}^{\le\delta s}\) and $\{S_{5s/8}\xleftrightarrow{\mathcal L_D^{\alpha,\le\delta s}}S_{3s/4}\}$ are realized by disjoint collections of loops. Therefore,
\begin{equation}
\label{eq:small-loop-recursion-part}
 \Pp_\alpha(\mathcal E_{s,D}^{\le\delta s})
 \le
 \chi_{5/8,3/4}^{\bar\alpha}(\delta)p_{s/2,D}^{\alpha}.
\end{equation}

On \((\mathcal E_{s,D}^{\le\delta s})^{\rm c}\), let
\(\mathcal K^{\le\delta s}\) be the union of \(S_1\) and all loops of
\(\mathcal L_D^{\alpha,\le\delta s}\) connected to \(S_1\).  Then
\(\mathcal K^{\le\delta s}\subset B_s\).
Conditionally on \(\mathcal L_D^{\alpha,\le\delta s}\), the loops of diameter larger than
\(\delta s\) form an independent Poisson process.
Thus, by \eqref{eq:large-loop-terminal-comparison}, on the event
\((\mathcal E_{s,D}^{\le\delta s})^{\rm c}\), we have
\[
\begin{aligned}
 \Pp_\alpha\bigl(
 \mathcal E_{s,D}\setminus\mathcal E_{s,D}^{\le\delta s}
 \mathrel{\big|}\mathcal L_D^{\alpha,\le\delta s}\bigr)
 &\le
 1-\exp\left\{-\alpha\mu_D^{\Loop}\bigl(
 \ell\cap\mathcal K^{\le\delta s}\ne\varnothing,\
 \diam(\ell)>\delta s\bigr)\right\}\\
 &\le
 C_\delta\bigl(1-\exp\{-\alpha m_s(\mathcal K^{\le\delta s})\}\bigr).
\end{aligned}
\]
Since
\(\mathcal K^{\le\delta s}\subset\mathcal K_s^{\alpha}\) on
\((\mathcal E_{s,D}^{\le\delta s})^{\rm c}\), taking expectations above yields
\[
\begin{aligned}
 \Pp_\alpha(\mathcal E_{s,D}\setminus\mathcal E_{s,D}^{\le\delta s})
 &\le C_\delta\Ee_\alpha\left[
 \mathbf 1_{(\mathcal E_{s,D}^{\le\delta s})^{\rm c}}
 \bigl(1-e^{-\alpha m_s(\mathcal K^{\le\delta s})}\bigr)\right]\\
 &\le C_\delta\Ee_\alpha\left[
 1-e^{-\alpha m_s(\mathcal K_s^{\alpha})}\right]
 =C_\delta g_s^{\alpha}.
\end{aligned}
\]
Combining this estimate with \eqref{eq:small-loop-recursion-part} proves
\eqref{eq:residual-extension}.
\end{proof}

We now complete the proof by a dyadic iteration.
\begin{proof}[Proof of Theorem~\ref{thm:main}]
Fix \(\bar\alpha\in\mathrm I_0\), and let \(c_\star\) be the
constant in Lemma~\ref{lem:robust-prop}.  By
Lemma~\ref{lem:fixed-ratio-bdelta} applied at \(\bar\alpha\), choose
\(\delta\in(0,\frac18)\) so that
\(\eta:=\chi_{5/8,3/4}^{\bar\alpha}(\delta)<c_\star/2\).
Fix \(0<\alpha\le\bar\alpha\) and \(D\supset B_{2r}\), write \(r=\rho2^M\) with
\(\rho\in[2,4)\) and \(M\ge0\), and set
\(P_m:=p_{\rho2^m,D}^{\alpha}\) and \(G_m:=g_{\rho2^m}^{\alpha}\).
For \(M\ge1\), the domain assumption holds at every scale
\(s=\rho2^m\le r\).  Applying \eqref{eq:residual-extension} successively at
these scales and using \(G_{M-j}\le c_\star^{-j}G_M\) gives
\begin{equation}
\label{eq:terminal-dyadic-iteration}
 P_M
 \le C_\delta\sum_{j=0}^{M-1}\eta^jG_{M-j}+\eta^MP_0
 \le C_\delta G_M\sum_{j=0}^{M-1}
 \left(\frac{\eta}{c_\star}\right)^j+\eta^M P_0.
\end{equation}
{
By \eqref{imp:eq:initial-scale}, \(P_0\le CG_0\). Iterating
\eqref{eq:terminal-propagation} gives \(G_0\le c_\star^{-M}G_M\). Thus, we have
\[
 \eta^MP_0\le C(\eta/c_\star)^MG_M,
 \qquad
 \sum_{j=0}^{M-1}(\eta/c_\star)^j\le2.
\]
Substituting the above into \eqref{eq:terminal-dyadic-iteration} gives
\(P_M\le CG_M\), uniformly in \(0<\alpha\le\bar\alpha\). The case \(M=0\)
is exactly \eqref{imp:eq:initial-scale}.
}
Since the event defining \(g_r^{\alpha}\) is
contained in \(\mathsf H_{r,D}^{2,\kappa_0}\) (recall \eqref{eq:terminal-loop-class}), we have
\(\Pp_\alpha(\mathsf H_{r,D}^{2,\kappa_0})
\ge g_r^{\alpha}\ge c p_{r,D}^{\alpha}\).
This proves \eqref{eq:local_version}, and hence Theorem~\ref{thm:main}.
The same argument with \(D=\R^3\) also gives
\(p_r^{\alpha}\le Cg_r^{\alpha}\), while \(g_r^{\alpha}\le p_{r,D}^{\alpha}\le p_r^{\alpha}\) for all
\(D\supset B_{2r}\).
\end{proof}

\subsection{Supermultiplicativity and up-to-constants estimates}
\label{imp:sec:uniform}

Fix
\(\bar\alpha\in\mathrm I_0\) and let \(0<\alpha\le\bar\alpha\).
All comparison constants in this subsection depend only on \(\bar\alpha\),
unless another dependence is indicated.
Put \(\varkappa:=\Capa(S_1)=\Capa(\overline B_1)\), so that
\(\Capa(\overline B_r)=\varkappa r\).

\begin{lemma}\label{imp:lem:fixed-ratio}
For each fixed \(L\ge1\), uniformly for \(r\ge2\) and \(1\le t\le L\),
\begin{equation}\label{imp:eq:fixed-ratio}
 p_{tr}^{\alpha}\asymp_L p_r^{\alpha}.
\end{equation}
For each fixed \(L\ge2\), also
\(p_r^{\alpha}\asymp_L\alpha\) uniformly for \(2\le r\le L\).
\end{lemma}

\begin{proof}
The proof of Theorem~\ref{thm:main} gives
\(p_r^{\alpha}\asymp g_r^{\alpha}\).
Together with \eqref{eq:terminal-propagation}, this yields
\(p_{2r}^{\alpha}\gtrsim p_r^{\alpha}\).
Iterating at most \(\lceil\log_2 L\rceil\) times and using
monotonicity proves \eqref{imp:eq:fixed-ratio}.
For \(2\le r\le L\), \eqref{imp:eq:initial-scale} gives
\(p_r^{\alpha}\le p_2^{\alpha}\lesssim\alpha\).
On the other hand, \(S_1\subset\mathcal K_r^{\alpha}\) and
\eqref{eq:terminal-capacity-expectation} yield
\[
 p_r^{\alpha}\ge g_r^{\alpha}
 \gtrsim\frac{\alpha}{r}\Capa(S_1)
 \ge\frac{\varkappa}{L}\alpha.
\]
We complete the proof.
\end{proof}

Next, we combine Theorem~\ref{thm:submultiplicativity},
Lemma~\ref{imp:lem:fixed-ratio} and Theorem~\ref{thm:main}  to obtain the following multiplicativity property for
one-arm probabilities.

\begin{theorem}
\label{imp:thm:normalized-product}
There are constants \(0<c\le C<\infty\) such
that, uniformly for \(0<\alpha\le\bar\alpha\) and \(r,s\ge2\),
\begin{equation}\label{imp:eq:normalized-product}
 \frac{c}{\alpha}p_r^{\alpha}p_s^{\alpha}
 \le p_{rs}^{\alpha}
 \le\frac{C}{\alpha}p_r^{\alpha}p_s^{\alpha}.
\end{equation}
\end{theorem}

\begin{proof}
Fix \(\bar\alpha\in\mathrm I_0\) and let \(0<\alpha\le\bar\alpha\).
It suffices to consider \(r,s\ge8\).
The upper bound follows from Theorem~\ref{thm:submultiplicativity} and
Lemma~\ref{imp:lem:fixed-ratio},
\[
 p_{rs}^{\alpha}\lesssim\alpha^{-1}p_r^{\alpha}p_{s/4}^{\alpha}
 \asymp\alpha^{-1}p_r^{\alpha}p_s^{\alpha}.
\]

For the lower bound, let \(K_1,K_2,\BL_{r,4r}\) be as in the proof of
Theorem~\ref{thm:submultiplicativity}.
By \eqref{eq:submult-bridge-poisson}, the uniform bound
\(\Bcal(K_1,K_2;\R^3)\le\Bcal(S_r,S_{4r};\R^3)<\infty\), and
Lemma~\ref{lem:splitting},
\begin{equation}\label{imp:eq:expected-bridge-mass}
 p_{rs}^{\alpha}
 \gtrsim\alpha\Ee_\alpha[\Bcal(K_1,K_2;\R^3)]
 \asymp\alpha\Ee_\alpha[\Bcal(K_1,S_{3r};\R^3)]\,
        \Ee_\alpha[\Bcal(S_{2r},K_2;\R^3)].
\end{equation}
Since \(K_1=\mathcal K_r^{\alpha}\), \eqref{eq:splitting-capacity},
\eqref{eq:terminal-capacity-expectation}, and the comparison
\(p_r^{\alpha}\asymp g_r^{\alpha}\) from the proof of Theorem~\ref{thm:main}
give
\begin{equation}\label{imp:eq:inner-factor}
 \Ee_\alpha[\Bcal(K_1,S_{3r};\R^3)]
 \asymp\frac{\Ee_\alpha[\Capa(\mathcal K_r^{\alpha})]}r
 \asymp\frac{p_r^{\alpha}}\alpha.
\end{equation}
Similarly, we have
\(\Ee_\alpha[\Bcal(S_{2r},K_2;\R^3)]
\asymp\frac{p_{s/4}^{\alpha}}\alpha\).
Combining these estimates, we obtain
\begin{equation}\label{imp:eq:product-lower}
 p_{rs}^{\alpha}
 \gtrsim\alpha^{-1}p_r^{\alpha}p_{s/4}^{\alpha}
 \asymp\alpha^{-1}p_r^{\alpha}p_s^{\alpha},
\end{equation}
where the last comparison follows from Lemma~\ref{imp:lem:fixed-ratio}.
This completes the proof.
\end{proof}

Recall the definition of $\mathcal K_r^{\alpha}$ in \eqref{eq:seeded-cluster-def}.
Now, we can deduce the following uniform one-arm estimates.
\begin{theorem}
\label{imp:thm:uniform-sharp}
Uniformly for \(0<\alpha\le\bar\alpha\) and \(r\ge2\),
\begin{equation}\label{imp:eq:uniform-sharp}
 p_r^{\alpha}\asymp\alpha r^{-\xi(\alpha)},
 \qquad \Ee_\alpha[\Capa(\mathcal K_r^{\alpha})]\asymp r^{1-\xi(\alpha)}.
\end{equation}
For every \(D\supset B_{2r}\), both \(p_{r,D}^{\alpha}\) and
\(\Pp_\alpha(\mathsf H_{r,D}^{2,\kappa_0})\) are also comparable to
\(\alpha r^{-\xi(\alpha)}\), uniformly in \(0<\alpha\le\bar\alpha\) and \(r\ge2\).
\end{theorem}

\begin{proof}[Proof of Theorems~\ref{thm:utc} and~\ref{imp:thm:uniform-sharp}]
Set \(q_\alpha(r):=p_r^{\alpha}/\alpha\).
Theorem~\ref{imp:thm:normalized-product} gives a constant \(C\ge1\),
depending only on \(\bar\alpha\), such that
\[
 C^{-1}q_\alpha(r)q_\alpha(s)
 \le q_\alpha(rs)\le Cq_\alpha(r)q_\alpha(s).
\]
For fixed \(\alpha>0\) and \(r\ge2\), iteration yields
\[
 C^{-(n-1)}q_\alpha(r)^n
 \le q_\alpha(r^n)\le C^{n-1}q_\alpha(r)^n.
\]
The existing exponent asymptotic \eqref{eq:one-arm-log-asymp} implies
\(q_\alpha(r^n)^{1/n}\longrightarrow r^{-\xi(\alpha)}\).
Taking \(n\)-th roots and then letting \(n\to\infty\), we obtain
\[
 C^{-1}r^{-\xi(\alpha)}\le q_\alpha(r)
 \le Cr^{-\xi(\alpha)}.
\]
Although the limit is taken at each fixed intensity, the comparison
constant depends only on \(\bar\alpha\), so the conclusion is uniform
on \((0,\bar\alpha]\).
The capacity and localized estimates now follow from
\eqref{eq:terminal-capacity-expectation}, \eqref{eq:local_version}, and
the comparisons at the end of the proof of Theorem~\ref{thm:main}.
\end{proof}

\section{Cluster dimension and equality of critical values}
\label{sec:multiplicativity-to-dimension}
In Section~\ref{sec:dimension}, we prove
Theorem~\ref{thm:all-cluster-dimension}, which identifies \(3-\xi(\alpha)\) as the Hausdorff dimension of every cluster.
In Section~\ref{sec:strict-monotonicity}, we prove Proposition~\ref{prop:xi-linear-separation}, showing that \(\xi(\alpha)\) decreases at a uniform linear rate
on compact subintervals of \(\mathrm{I}_0\).
We then prove the bounds for small intensities and left
continuity of $\xi(\alpha)$, completing the proof of
Theorem~\ref{thm:xi-strict-monotone}.
Finally, Section~\ref{sec:continuum-thresholds} proves the
equality of critical values in Theorem~\ref{thm:continuum-thresholds}.

\subsection{Hausdorff dimension of clusters}
\label{sec:dimension}

To deal with the dimension of all clusters simultaneously, we will use Palm's formula (Lemma~\ref{lem:loop-soup-palm}) to mark a loop and
 then calculate the dimension of the cluster attached to it; see Lemma~\ref{lem:marked-loop-lower-bound}.
The main technical input is a localized first- and second-moment estimate for the cluster attached to a
positive-capacity set, which is given in Lemma~\ref{lem:seeded-moments} below.
Then, we use a standard second-moment method to prove a positive-probability lower bound for
the Hausdorff dimension of the cluster attached to a set with positive capacity; see Lemma~\ref{lem:localized-annular-trial}.
To upgrade the resulting positive-probability bound to an almost-sure
statement, we find linearly many
disjoint positive-capacity pieces along a marked Brownian loop (Lemma~\ref{lem:radial-capacity-boxes}), which enables us
to perform infinitely many trials and bootstrap the result.

We first give the localized first- and second-moment estimate,
which is a direct consequence of the up-to-constants one-arm estimate (Theorem~\ref{imp:thm:uniform-sharp}).
\begin{lemma}
\label{lem:seeded-moments}
For every \(\alpha\in \mathrm{I}_0\) and \(\rho,L>0\), there is
\(C=C(\alpha,\rho,L)<\infty\)
with the following property.  Let \(D\subset\R^3\) be a domain, let
\(K\subset D\) be compact and connected with \(\Capa(K)>0\), and let
\(U\subset D\setminus K\) be bounded and open.  Suppose that
\[
        \dist(\overline U,K)\ge\rho,\qquad
        \dist(\overline U,D^{\mathrm c})\ge\rho,\qquad
        \diam(U)\le L{.}
\]
Then, for all
\(x\in U\) and \(0<\varepsilon\le\frac\rho{100}\),
\begin{equation}
        \Pp_\alpha\bigl(B(x,\varepsilon)
        \xleftrightarrow{\LL_D^\alpha}K\bigr)
        \le
        C\,\varepsilon^\xi.
        \label{eq:seeded-first-moment}
\end{equation}
Moreover, for all \(x,y\in U\) and
\(0<\varepsilon_1,\varepsilon_2\le\frac\rho{100}\),
\begin{equation}
        \Pp_\alpha\bigl(
        B(x,\varepsilon_1)\xleftrightarrow{\LL_D^\alpha}K,\,
        B(y,\varepsilon_2)\xleftrightarrow{\LL_D^\alpha}K
        \bigr)
        \le
        C\,\varepsilon_1^\xi\varepsilon_2^\xi
        \bigl(|x-y|\vee\varepsilon_1\vee\varepsilon_2\bigr)^{-\xi}.
        \label{eq:seeded-second-moment}
\end{equation}
\end{lemma}

\begin{proof}
A crossing in $\LL_D^\alpha$ from \(B(x,\varepsilon)\) to \(K\)
contains a subcrossing to distance \(\frac\rho2\).  Thus
Theorem~\ref{imp:thm:uniform-sharp} gives, uniformly over all the sets in the statement,
\[
        \Pp_\alpha\bigl(B(x,\varepsilon)
        \xleftrightarrow{\LL_D^\alpha}K\bigr)
        \lesssim \varepsilon^\xi .
\]
For the two-point estimate, put \(s=|x-y|\) and
\(\delta=\varepsilon_1\vee\varepsilon_2\). If \(s\le2\delta\), apply
the first bound to the smaller ball. If
\(2\delta<s\le\frac\rho4\), the
 splitting strategy used in the proof of
Theorem~\ref{thm:submultiplicativity}, followed by
Theorem~\ref{imp:thm:uniform-sharp}, bounds the joint probability by
\[
        C\left(\frac{\varepsilon_1}{s}\right)^\xi
        \left(\frac{\varepsilon_2}{s}\right)^\xi s^\xi
        =C\,\varepsilon_1^\xi\varepsilon_2^\xi s^{-\xi},
\]
with a constant depending only on \(\alpha,\rho,L\).  If \(s>\frac\rho4\),
the same splitting at a fixed scale comparable to \(\rho\) gives
\(C\varepsilon_1^\xi\varepsilon_2^\xi\). Since \(s\le L\), this is at most
\(C L^\xi\varepsilon_1^\xi\varepsilon_2^\xi s^{-\xi}\), which proves
\eqref{eq:seeded-second-moment}.
\end{proof}

We next prove a positive-probability
lower bound for the Hausdorff dimension of the cluster attached to a positive-capacity set.

\begin{lemma}
\label{lem:localized-annular-trial}
For every \(\alpha\in \mathrm{I}_0\) and \(\kappa>0\), there is
\(q=q(\alpha,\kappa)>0\) such that the following holds.
Let \(z\in\R^3\), \(r>0\), and let
\(K\subset A(z,2r,4r)\) be compact and connected, with
\(\Capa(K)\ge \kappa r\).  Let
\(\mathcal C_z^r(K)\) be the cluster attached to \(K\) using an independent
soup restricted to loops contained in \(B(z,8r)\).  Then
\[
        \Pp_\alpha\bigl(
        \dimH(\mathcal C_z^r(K)\cap B(z,\frac r4))\ge3-\xi
        \bigr)
        \ge q(\alpha,\kappa).
\]
\end{lemma}

\begin{proof}
By translation and scaling, take \(z=0\) and \(r=1\), and fix
\(U=B(0,\frac18)\).  We first prove that there exists
\(c_{\alpha,\kappa}>0\), uniformly over the admissible \(K\), such that
\begin{equation}\label{eq:localized-seeded-lower-bound}
        \Pp_\alpha\bigl(B(x,\varepsilon)\cap\mathcal C_0^1(K)\ne\varnothing\bigr)
        \ge c_{\alpha,\kappa}\varepsilon^\xi
\end{equation}
for every \(x\in U\) and \(0<\varepsilon\le\frac1{100}\).
For these $x$ and $\varepsilon$, Theorem~\ref{imp:thm:uniform-sharp} shows that, with
probability at least \(a_\alpha\varepsilon^\xi\), there exist a
cluster
\(\mathcal C_{x,\varepsilon}\) of \(\LL^\alpha_{B(x,3/2)}\)
across \(A(x,\frac\varepsilon2,\frac12)\) and a compact
connected set \(J\) such that
\[
        J\subset \mathcal C_{x,\varepsilon}\cap A(x,\frac14,1),
        \qquad
        \diam(J)\ge\frac14,
        \qquad
        \Capa(J)\ge\frac{\kappa_0}2.
\]
Since \(x\in B(0,\frac18)\), we have
\(J\subset B(0,\frac98)\), whereas \(K\subset A(0,2,4)\).  Restricting the proof
of Lemma~\ref{lem:two-capacity-welding} to \(B(0,8)\), with cutting spheres
\(S_{5/4}\) and \(S_{3/2}\), therefore gives, uniformly in \(x,J,K\),
\[
        \muLoop(\ell:\ell\subset B(0,8),\,
        \ell\cap J\ne\varnothing,\,
        \ell\cap K\ne\varnothing)
        \ge c\,\Capa(J)\Capa(K)
        \ge \frac{c\,\kappa_0\kappa}{2} .
\]
Every loop counted here intersects \(K\), which is disjoint from \(B(x,\frac32)\), so
this loop class is independent of the local soup that determines
\(\mathcal C_{x,\varepsilon}\) and \(J\).
Hence, given \(\mathcal C_{x,\varepsilon}\) and \(J\), we can insert a loop joining
\(J\) to \(K\) with probability at least
\(1-\exp\{-\frac{\alpha c\kappa_0\kappa}{2}\}>0\).
Consequently, \eqref{eq:localized-seeded-lower-bound} holds with
\(c_{\alpha,\kappa}
:=a_\alpha\bigl(1-\exp\{-\frac{\alpha c\kappa_0\kappa}{2}\}\bigr)>0\).
Moreover, Lemma~\ref{lem:seeded-moments}, applied with
\(D=B(0,8)\) and \(\rho=L=1\), gives the required one- and two-point upper
bounds uniformly over the admissible \(K\).
By the standard second-moment criterion \cite[Theorem~10.43]{MP10}, for some
\(q=q(\alpha,\kappa)>0\),
\[
        \Pp_\alpha\bigl(
        \dimH(\mathcal C_0^1(K)\cap U)\ge3-\xi
        \bigr)\ge q,
\]
which proves the
claim.
\end{proof}

The next lemma obtains many disjoint positive-capacity pieces of
a bridge from crossings of separated annuli.

\begin{lemma}
\label{lem:radial-capacity-boxes}
For every \(\delta>0\) there exist \(\kappa_1>0\) and \(c_1>0\) such that, for
every \(x\in\R^3\) and \(t>0\),
\[
        \lim_{r\to0}\Pbb^{x,x}_t\bigl(\diam(\gamma)\ge\delta,\,
        N_r(\gamma)<\frac{c_1\delta}{r}\bigr)=0 .
\]
Here \(N_r(\gamma)\) is the largest cardinality of a collection of pairs
\((z_i,K_i)\) such that the balls \(B(z_i,8r)\) are pairwise disjoint,
\(K_i\subset\gamma\cap A(z_i,2r,4r)\), and \(\Capa(K_i)\ge\kappa_1r\).
\end{lemma}

\begin{proof}
On \(\{\diam(\gamma)\ge\delta\}\), the law of the bridge before exiting \(B(x,\frac\delta3)\)
is absolutely continuous with respect to the law of \(Z([0,\tau_{\delta/3}])\), where \(Z\) is standard Brownian motion
started at \(x\) and stopped on first exiting \(B(x,\frac\delta3)\).
Here \(\tau_a:=\inf\{u\ge0:|Z_u-x|=a\}\) for \(a>0\).
It therefore suffices to prove the claim for
\(Z([0,\tau_{\delta/3}])\).
Fix \(M\ge40\). For
\(1\le j\le n(r):=\lfloor(\delta-24r)/(3Mr)\rfloor\), set
\(a_j=jMr\).  Define
\[
        v_j:=\inf\{s\ge\tau_{a_j}:\ |Z_s-Z_{\tau_{a_j}}|=\frac{7r}{2}\},
        \qquad
        u_j:=\sup\{s\le v_j:\ |Z_s-Z_{\tau_{a_j}}|=\frac{5r}{2}\}.
\]
Set \(K_j:=Z([u_j,v_j])\).  Then
\(v_j\le\tau_{a_j+4r}\le\tau_{\delta/3}\), and
\(K_j\subset Z([0,\tau_{\delta/3}])\cap A(Z_{\tau_{a_j}},2r,4r)\) is compact and connected.
By the strong Markov property, scaling, and~\cite[Theorem~1.1]{PPS},
there are \(\kappa_1,p_*>0\), independent of \(j,r\), such
that
\[
        \Pp\bigl(\Capa(K_j)\ge\kappa_1 r
        \,\bigm|\,\mathcal F_{\tau_{a_j}}\bigr)\ge p_* .
\]
The Chernoff bound gives
\[
        \Pp\left(
        \sum_{j=1}^{n(r)}
        \mathbf 1_{\{\Capa(K_j)\ge\kappa_1r\}}
        \le \frac{p_*}{2}n(r)
        \right)
        \le \exp\{-c(p_*)n(r)\}.
\]
Here \(n(r)\asymp\frac\delta{Mr}\), while the separation of the radii \(a_j\) and
\(M\ge40\) make the balls \(B(Z_{\tau_{a_j}},8r)\) disjoint.  Thus, after decreasing a
deterministic \(c_1>0\), the stopped Brownian path contains at least
\(\frac{c_1\delta}{r}\) admissible boxes with probability tending to one. This completes the proof.
\end{proof}

The loop soups restricted to the disjoint balls \(B(z_i,8r)\) are independent, so a fixed
positive success probability amplifies to an almost-sure statement by sending $r\to 0$.
For a compact set \(K\), let \(\mathcal C(K)\) be the closed union of \(K\)
and all loop-soup clusters that intersect \(K\).

\begin{lemma}
\label{lem:marked-loop-lower-bound}
Let \(\alpha\in \mathrm{I}_0\), and let \(\LL_{\R^3}^\alpha\) be an independent loop
soup.  Then, for
\(\muLoop\)-almost every Brownian loop \(\gamma\),
\[
        \Pp_\alpha\bigl(\dimH\mathcal C(\gamma)\ge 3-\xi
        \,\bigm|\,\gamma\bigr)=1 .
\]
\end{lemma}

\begin{proof}
Fix \(\delta>0\) and a bridge law \(\Pbb_t^{x,x}\).  Let
\(\kappa_1,c_1\) come from Lemma~\ref{lem:radial-capacity-boxes} and put
\(q=q(\alpha,\kappa_1)\) as in Lemma~\ref{lem:localized-annular-trial}.
Conditional on \(N_r(\gamma)\ge \frac{c_1\delta}{r}\), choose a collection
\((z_i,K_i)_{1\le i\le N_r(\gamma)}\) from Lemma~\ref{lem:radial-capacity-boxes}.  Independence of the soups restricted
to the disjoint balls \(B(z_i,8r)\), together with
Lemma~\ref{lem:localized-annular-trial}, gives
\[
\begin{aligned}
        &\Pp_\alpha\bigl(
        \dimH(\mathcal C_{z_i}^r(K_i)\cap B(z_i,\frac r4))<3-\xi
        \text{ for every }1\le i\le N_r(\gamma)
        \,\bigm|\,\gamma\bigr)\\
        &\qquad\le (1-q)^{N_r(\gamma)}
        \le (1-q)^{c_1\delta/r}.
\end{aligned}
\]
Consequently,
\[
\begin{aligned}
        &(\Pbb_t^{x,x}\otimes\Pp_\alpha)\bigl(\diam(\gamma)\ge\delta,\,
        \dimH\mathcal C(\gamma)<3-\xi\bigr)  \\
        &\qquad\le
        \Pbb^{x,x}_t\bigl(\diam(\gamma)\ge\delta,\,
        N_r(\gamma)<\frac{c_1\delta}{r}\bigr)
        +(1-q)^{c_1\delta/r}.
\end{aligned}
\]
Both terms vanish as \(r\to 0\).  Integrating the resulting bridge
statement in the definition of \(\muLoop\), then taking the union over
rational \(\delta>0\), proves the claim.
\end{proof}

Finally, we finish the proof of Theorem~\ref{thm:all-cluster-dimension} by using Palm's formula.

\begin{proof}[Proof of Theorem~\ref{thm:all-cluster-dimension}]
Fix \(\alpha\in \mathrm{I}_0\).
Since every cluster contains a loop,
Lemma~\ref{lem:marked-loop-lower-bound} and
Lemma~\ref{lem:loop-soup-palm} (Palm's formula; in particular \eqref{eq:palm} applied in \(\R^3\)) show that
\[
        \Pp_\alpha\left(
        \forall \mathcal C\in\mathscr C(\LL_{\R^3}^\alpha),\quad
        \dimH(\mathcal C)\ge 3-\xi
        \right)=1 .
\]

We turn to the upper bound.
For \(\muLoop\)-almost every \(\gamma\), its
trace is compact, connected, has positive capacity, and has Hausdorff
dimension \(2\).
Let \(\LL_{\R^3}^\alpha\) be an independent loop soup, and let \(\mathcal C(\gamma)\) be the cluster of \(\LL_{\R^3}^\alpha\) attached to \(\gamma\).
If \(U\Subset\R^3\setminus\gamma\) is bounded, the first moment bound
\eqref{eq:seeded-first-moment} gives \(\dimH(\mathcal C(\gamma)\cap U)\le3-\xi\) almost surely.
A countable
exhaustion of \(\R^3\setminus\gamma\) further yields \(\dimH(\mathcal C(\gamma)\setminus\gamma)\le3-\xi\).
Since \(\xi<1\), as recalled in the introduction,
\(\dimH\gamma=2<3-\xi\), and we conclude that
\[
        \Pp_\alpha\bigl(\dimH \mathcal C(\gamma)\le 3-\xi\,\bigm|\,\gamma\bigr)=1
        \qquad\text{for }\muLoop\text{-a.e. }\gamma .
\]
One more application of Lemma~\ref{lem:loop-soup-palm} gives the upper bound
simultaneously for all clusters, completing the proof.
\end{proof}

\subsection{Proof of Theorem~\ref{thm:xi-strict-monotone}}
\label{sec:strict-monotonicity}
\label{imp:sec:exponent}

We first prove the quantitative strict monotonicity in Proposition~\ref{prop:xi-linear-separation}.
We will work on a compact interval
\(I=[\alpha_-,\alpha_+]\subset \mathrm{I}_0\) and allow the constants to depend on \(I\).
Recall that \(p_r^\alpha=\Pp_\alpha(\mathcal E_r)\).

{

We first extend the definition~\eqref{eq:robust-event-def} in the following way.
\begin{definition}[Good pieces]\label{def:good-piece}
For \(0<a<b\), consider a compact connected set \(K\)
contained in a cluster across \(A(a,b)\).
We call \(K\) a good outer piece for \(\mathcal E_{a,b}\) if
\(K\subset A(\frac b2,2b)\), \(\diam(K)\ge \frac b2\), and
\(\Capa(K)\ge \kappa_0 b\).
We call \(K\) a good inner piece if
\(K\subset A(\frac a2,2a)\), \(\diam(K)\ge\frac a8\), and
\(\Capa(K)\ge\frac{\kappa_0a}{4}\).
We denote by \(\mathsf H^+(a,b)\) and \(\mathsf H^-(a,b)\) the events that
such a good outer piece and a good inner piece exist, respectively.
\end{definition}
Note that $\mathsf H^+(1,r)=\mathsf H_{r}^{2,\kappa_0}$ by definition \eqref{eq:robust-event-def}.

\begin{corollary}
\label{lem:localized-two-sided}
For every compact interval
\(I=[\alpha_-,\alpha_+]\subset\mathrm{I}_0\), there exists a constant
\(c>0\) such that for all
\(\alpha\in I\) and \(0<a\le \frac b2\),
\[
    \Pp_\alpha\bigl(\mathsf H^+(a,b)\bigr)
    \ge c\,p_{b/a}^{\alpha},
    \qquad
    \Pp_\alpha\bigl(\mathsf H^-(a,b)\bigr)
    \ge c\,p_{b/a}^{\alpha}.
\]
\end{corollary}

\begin{proof}
After scaling \(a\) to \(1\), Theorem~\ref{thm:main} gives
$
        \Pp_\alpha(\mathsf H^+(a,b))\ge c\,p_{b/a}^{\alpha}
$.
For $\mathsf H^-(a,b)$, consider the spherical inversion
\(\iota_{a,b}(x):=\frac{ab}{\lvert x\rvert^2}\,x\).
It exchanges \(S_a\) and \(S_b\), and maps \(A(\frac b2,2b)\) onto
\(A(\frac a2,2a)\).  By the inversion invariance of Brownian loop measure
\cite[Lemma~4.1]{JL26}, the image of the Brownian loop soup under \(\iota_{a,b}\)
has the same law, while intersections, and hence crossing clusters, are
preserved.
{For \(x,y\in A(\frac b2,2b)\),
\[
 |\iota_{a,b}(x)-\iota_{a,b}(y)|
 =\frac{ab}{|x||y|}|x-y|
 \ge\frac{a}{4b}|x-y|.
\]
Hence, for every good outer piece \(K\),
\[
 \diam(\iota_{a,b}(K))\ge\frac{a}{4b}\diam(K)\ge\frac a8,
 \qquad
 \Capa(\iota_{a,b}(K))\ge\frac{a}{4b}\Capa(K)
 \ge\frac{\kappa_0a}{4}.
\]
Thus its image satisfies exactly the good inner piece conditions in
Definition~\ref{def:good-piece}, proving the second estimate.}
\end{proof}
}
For a domain
\(D\), let \(\mathsf H_D^\pm(a,b)\) denote \(\mathsf H^\pm(a,b)\) with
\(\LL_{\R^3}^\alpha\) replaced by \(\LL_D^\alpha\).
Note that
\(\mathsf H_D^+(1,s)=\mathsf H_{s,D}^{2,\kappa_0}\).
Then, we have the following localized version of Corollary~\ref{lem:localized-two-sided}.

\begin{lemma}\label{lem:localized-robust-crossing}
There is \(c>0\) such that for every \(\alpha\in I\),
\(s\ge16\) and \(r\ge64s\),
\[
        \Pp_\alpha\bigl(\mathsf H_{B_s}^+(1,\frac s2)\bigr)
        \ge c\,p_s^\alpha,
        \qquad
        \Pp_\alpha\bigl(\mathsf H_{B_{2s}^c}^-(4s,r)\bigr)
        \ge c\,p_{r/s}^\alpha.
\]
\end{lemma}

\begin{proof}
The first estimate follows from
\eqref{eq:local_version}, applied with outer radius \(s/2\) and domain
\(B_s\), together with Lemma~\ref{lem:local} and monotonicity in the radius.
The comparison in Lemma~\ref{lem:local} is uniform for \(\alpha\in I\), as
is immediate from its proof.  The second estimate follows by the inversion
used in Corollary~\ref{lem:localized-two-sided}.
\end{proof}

Recall the definition of the pivotal event from \eqref{eq:pivotal}.
We next lower bound the mass of pivotal loops with diameter comparable to \(s\) for the one-arm event \(\mathcal E_r\).

\begin{lemma}\label{lem:annular-pivotal-lower}
There is \(c>0\) such that the following holds.  Define
\begin{equation}\label{eq:annular-pivotal-loop-class}
        \mathscr P_s:=\{\ell:\ell\subset B_{16s},\ s\le\diam(\ell)\le32s\}.
\end{equation}
Then, for every \(\alpha\in I\), \(r\ge1024\), and
\(16\le s\le \frac r{64}\),
\begin{equation}\label{eq:annular-pivotal-lower}
        \int_{\mathscr P_s}
        \Pp_\alpha(\operatorname{Piv}_\ell \mathcal E_r)\,
        \muLoop(\dint\ell)
        \ge c\,p_r^\alpha .
\end{equation}
\end{lemma}

\begin{proof}
For \(\delta>0\) and \(s\ge1\), let \(\mathsf F_s^\delta\) be the event that
\(S_{5s/4}\) is connected to \(S_{3s/2}\) using only loops of \(\LL_{\R^3}^\alpha\)
that are contained in \(A(s,2s)\) and have diameter at most \(\delta s\).
Since \(\alpha_+\in \mathrm{I}_0\), Lemma~\ref{lem:fixed-ratio-bdelta} and monotonicity
in the intensity allow us to choose \(\delta<\frac1{16}\) such that
\(\sup_{\alpha\in I}\Pp_\alpha(\mathsf F_s^\delta)\le\frac12\) for every
\(s\ge1\).
Define the event
\begin{equation}\label{eq:annular-barrier-event}
        \mathsf Z_s
        :=
        \{\text{no loop }\ell\in\LL_{\R^3}^\alpha\text{ with }
        \diam(\ell)>\delta s\text{ intersects }A(\frac{5s}{4},\frac{3s}{2})\}
        \cap(\mathsf F_s^\delta)^{\mathrm c}.
\end{equation}
By scale invariance and local finiteness, we get that
\[
        M_\delta:=
        \muLoop\{\ell:\diam(\ell)>\delta,\,
        \ell\cap A(\frac54,\frac32)\ne\varnothing\}<\infty .
\]
Since the two conditions in \eqref{eq:annular-barrier-event} use disjoint loop classes,
for every \(\alpha\in I\) and \(s\ge1\),
\begin{equation}\label{eq:annular-barrier-prob}
        \Pp_\alpha(\mathsf Z_s)
        \ge \frac12e^{-\alpha_+M_\delta}
        \ge c.
\end{equation}
By construction, \(\mathsf Z_s\) is independent of the loop collections in
\(B_s\) and \(B_{2s}^c\).
Now, we set
\[
        \mathsf G_r(s):=
        \mathsf H_{B_s}^+(1,\frac s2)
        \cap\mathsf Z_s\cap
        \mathsf H_{B_{2s}^c}^-(4s,r) .
\]
Lemma~\ref{lem:localized-robust-crossing},
\eqref{eq:annular-barrier-prob}, independence, and
Theorem~\ref{thm:utc} yield that
\[
        \Pp_\alpha(\mathsf G_r(s))
        \ge c\,p_s^\alpha p_{r/s}^\alpha
        \ge c\,p_r^\alpha .
\]

On the one hand, since \(\mathsf Z_s\) disconnects \(B_s\) from \(B_{2s}^c\), we have
\(\mathsf G_r(s)\subset \mathcal E_r^{\mathrm c}\).
On the other hand, we can choose a good outer piece \(J^+\) and
a good inner piece \(J^-\) on \(\mathsf G_r(s)\),
and any loop that intersects
both pieces creates \(\mathcal E_r\) and is therefore pivotal for \(\mathcal E_r\).
Restricting the Brownian pieces in the
proof of Lemma~\ref{lem:two-capacity-welding} to \(B_{16s}\) gives the lower bound
\[
        \muLoop\{\ell\subset B_{16s}:
        \ell\cap J^+\ne\varnothing,\ \ell\cap J^-\ne\varnothing\}
        \ge c .
\]
Such a loop has diameter between \(s\) and \(32s\), hence lies in
\(\mathscr P_s\) and is pivotal on \(\mathsf G_r(s)\).  Therefore, we obtain that
\[
\begin{aligned}
        \int_{\mathscr P_s}
        \Pp_\alpha(\operatorname{Piv}_\ell \mathcal E_r)\,\muLoop(\dint\ell)
        &\ge
        \Ee_\alpha\left[
        \1_{\mathsf G_r(s)}
        \muLoop\{\ell\in\mathscr P_s:
        \ell\cap J^+\ne\varnothing,\ \ell\cap J^-\ne\varnothing\}
        \right]\\
        &\ge c\,\Pp_\alpha(\mathsf G_r(s))
        \ge c\,p_r^\alpha ,
\end{aligned}
\]
as required.
\end{proof}

Now, we show that $\xi(\alpha)$ is strictly decreasing in $\alpha$.

\begin{proposition}\label{prop:xi-linear-separation}
For every compact interval \(I\subset\mathrm I_0\), there exists \(c>0\)
such that, for all \(\alpha<\beta\) in \(I\),
\(\xi(\alpha)-\xi(\beta)\ge c(\beta-\alpha)\).
\end{proposition}

\begin{proof}
Fix \(\alpha<\beta\) in \(I\).  For \(r\ge1024\), set
\(\mathcal N_r:=\{64^k:\ 16\le64^k\le \frac r{64}\}\).
Note that \(|\mathcal N_r|\ge c\log r\) for all large
\(r\).
Applying Lemma~\ref{lem:restricted-russo} to the finite union
\(\bigcup_{s\in\mathcal N_r}\mathscr P_s\), and then using
Lemma~\ref{lem:annular-pivotal-lower}, gives, for a.e. \(\lambda\in I\),
\[
        \frac{\dint}{\dint\lambda}p_r^\lambda
        \ge
        \sum_{s\in\mathcal N_r}
        \int_{\mathscr P_s}
        \Pp_\lambda(\operatorname{Piv}_\ell \mathcal E_r)\,\muLoop(\dint\ell)
        \ge
        c\,(\log r)\,p_r^\lambda .
\]
This implies \(\frac{\dint}{\dint\lambda}\log p_r^\lambda
        \ge c\log r\) for a.e. \(\lambda\in[\alpha,\beta]\).
Integration gives
\(r^{c(\beta-\alpha)}p_r^\alpha\le p_r^\beta\).
 Applying Theorem~\ref{thm:utc} at the two endpoints and letting
\(r\to\infty\) now yields
\(\xi(\alpha)-\xi(\beta)\ge c(\beta-\alpha)>0\), finishing the proof.
\end{proof}

We next give the results for small \(\alpha\).

\begin{theorem}\label{imp:thm:linear}
There exist \(\alpha_0>0\) and \(0<c<C<\infty\) such that, for
\(0<\alpha\le\alpha_0\),
\begin{equation}\label{imp:eq:linear}
 c\alpha\le1-\xi(\alpha)\le C\alpha.
\end{equation}
Moreover, uniformly for these intensities and all \(r\ge2\),
\begin{equation}\label{imp:eq:joint-one-arm}
 c\alpha r^{-1+c\alpha}\le p_r^{\alpha}
 \le C\alpha r^{-1+C\alpha}.
\end{equation}
\end{theorem}

\begin{proof}
Write \(\varkappa=\Capa(S_1)\), choose a fixed \(r_0\ge2\) so large that
\(\kappa_0r_0/2\ge\varkappa\), and put
\(r_n=2^nr_0\) and \(A_n=\Ee_\alpha[\Capa(\mathcal K_{r_n}^{\alpha})]\).
The proof bounds \(A_n\) by two geometric sequences.

For a set \(K\supset S_1\), write
\(\operatorname{rad}(K)=\sup_{x\in K}|x|\). Capacity monotonicity gives
\(\Capa(K)\le\varkappa\operatorname{rad}(K)\).
For \(j<n\) and \(t\in[r_j,r_{j+1}]\), we have
\(\Pp_\alpha(\operatorname{rad}(\mathcal K_{r_n}^{\alpha})>t)
\le p_{r_j}^{\alpha}\).
Hence,
\begin{equation}\label{imp:eq:capacity-upper-recursion}
 A_n\le\varkappa\int_0^{r_n}
       \Pp_\alpha\bigl(\operatorname{rad}(\mathcal K_{r_n}^{\alpha})>t\bigr)\,dt
     \le\varkappa r_0+\varkappa\sum_{j=0}^{n-1}r_jp_{r_j}^{\alpha}
     \le C_0+C_2\alpha\sum_{j=0}^{n-1}A_j.
\end{equation}
For the last inequality, \eqref{eq:terminal-capacity-expectation} and the
comparison \(p_r^{\alpha}\asymp g_r^{\alpha}\) from the proof of
Theorem~\ref{thm:main} give \(r_jp_{r_j}^{\alpha}\asymp\alpha A_j\).
We absorb the contribution \(\varkappa r_0\) below \(r_0\) into \(C_0\).

For the lower bound, fix \(j<n\) and consider the event in
\eqref{eq:terminal-event-probability} at scale \(r_j\), whose probability is
\(g_{r_j}^{\alpha}\). By \eqref{eq:terminal-loop-class}, on this event, \(\Capa(\mathcal K_{2r_j}^{\alpha})\ge\kappa_0r_j\).
Therefore for every \(0\le t<\kappa_0r_j\),
\(\Pp_\alpha(\Capa(\mathcal K_{r_n}^{\alpha})>t)\ge \Pp_\alpha(\Capa(\mathcal K_{2r_j}^{\alpha})>t)\ge g_{r_j}^{\alpha}\).
Integrating this bound over \([\kappa_0r_j/2,\kappa_0r_j]\), for \(0\le j<n\), we get
\begin{equation}\label{imp:eq:capacity-lower-recursion}
 A_n=\int_0^\infty
   \Pp_\alpha\bigl(\Capa(\mathcal K_{r_n}^{\alpha})>t\bigr)\,dt
 \ge\varkappa+\frac{\kappa_0}{2}\sum_{j=0}^{n-1}r_jg_{r_j}^{\alpha}
 \ge c_0+c_2\alpha\sum_{j=0}^{n-1}A_j.
\end{equation}
Here the last step uses \eqref{eq:terminal-capacity-expectation}.

Choose \(c_0,C_0\) so that \(c_0\le A_0\le C_0\). The recursions
\eqref{imp:eq:capacity-upper-recursion} and \eqref{imp:eq:capacity-lower-recursion} imply
\begin{equation}\label{imp:eq:capacity-gronwall}
 c_0(1+c_2\alpha)^n\le A_n\le C_0(1+C_2\alpha)^n,
 \qquad n\ge0.
\end{equation}
By \eqref{imp:eq:uniform-sharp}, for each fixed \(\alpha\),
\(\lim_{n\to\infty}\frac{\log A_n}{\log r_n}=1-\xi(\alpha)\).
Consequently, \eqref{imp:eq:capacity-gronwall} yields the more precise bounds
\begin{equation}\label{imp:eq:log-linear}
 \frac{\log(1+c_2\alpha)}{\log2}
 \le1-\xi(\alpha)
 \le\frac{\log(1+C_2\alpha)}{\log2}.
\end{equation}
This implies \eqref{imp:eq:linear} for sufficiently small \(\alpha\).
Finally, substituting \eqref{imp:eq:linear} into Theorem~\ref{imp:thm:uniform-sharp} proves
\eqref{imp:eq:joint-one-arm}.
\end{proof}

We next prove the left continuity of the one-arm probability before showing the same property for $\xi(\alpha)$.

\begin{lemma}\label{imp:lem:fixed-scale-left}
For every fixed \(R>1\) and deterministic domain \(D\),
\({\lambda\longmapsto\Pp_\lambda\bigl(S_1\xleftrightarrow{\LL_D^\lambda}S_R\bigr)}\)
is left-continuous at every positive intensity. This includes
\(p_R^{\lambda}\) when \(D=\R^3\).
\end{lemma}
\begin{proof}
Take a Poisson process on loops with intensity marks, of intensity measure
\(\mu_D^{\Loop}(d\ell)\,du\), and retain the loops whose marks satisfy
\(u<\lambda\) to form \(\LL_D^\lambda\).
If \(\lambda_n\uparrow\lambda\) and the connection occurs at intensity
 \(\lambda\) via a finite chain of loops, as required by the definition, then the maximum of these finitely
many marks is strictly less than \(\lambda\). Hence, the whole chain is present
at intensity \(\lambda_n\) for all sufficiently large \(n\),
and
\(\{S_1\xleftrightarrow{\mathcal L_D^{\lambda_n}}S_R\}
\uparrow\{S_1\xleftrightarrow{\mathcal L_D^{\lambda}}S_R\}\)
almost surely.
Monotone convergence proves the assertion.
\end{proof}

\begin{theorem}\label{imp:thm:left-continuity}
For every \(\alpha\in\mathrm I_0\), we have
\(\lim_{\beta\uparrow\alpha}\xi(\beta)=\xi(\alpha)\).
\end{theorem}
\begin{proof}
Since \(\xi\) is nonincreasing, the limit
\(\zeta:=\lim_{\beta\uparrow\alpha}\xi(\beta)\) exists and satisfies
\(\zeta\ge\xi(\alpha)\).
For \(R\ge2\), set
\(Q_R(\lambda)=\Pp_\lambda(S_1\xleftrightarrow{\LL_{B_{2R}}^\lambda}S_R)\).
Theorem~\ref{imp:thm:uniform-sharp}, applied with \(\bar\alpha=\alpha\), gives
\begin{equation}\label{imp:eq:left-continuity-power-bounds}
 C^{-1}\lambda R^{-\xi(\lambda)}
 \le Q_R(\lambda)\le C\lambda R^{-\xi(\lambda)},
 \qquad 0<\lambda\le\alpha,
\end{equation}
where \(C\ge1\) depends only on \(\alpha\).
For each fixed \(R\), Lemma~\ref{imp:lem:fixed-scale-left} gives
\(Q_R(\beta)\uparrow Q_R(\alpha)\) as \(\beta\uparrow\alpha\).
Taking \(\lambda=\beta\) and letting \(\beta\uparrow\alpha\) in
\eqref{imp:eq:left-continuity-power-bounds}, we obtain
\[
 C^{-1}\alpha R^{-\zeta}\le Q_R(\alpha)\le C\alpha R^{-\zeta}.
\]
Comparing with \eqref{imp:eq:left-continuity-power-bounds} at \(\lambda=\alpha\)
and letting \(R\to\infty\) gives \(\zeta=\xi(\alpha)\), as claimed.
\end{proof}

\begin{proof}[Proof of Theorem~\ref{thm:xi-strict-monotone}]
Proposition~\ref{prop:xi-linear-separation} gives
\eqref{eq:xi-linear-separation}, and hence strict decrease of \(\xi\).
Theorem~\ref{imp:thm:linear} gives \eqref{eq:xi-0}, which also implies
continuity at zero with \(\xi(0)=1\).
Theorem~\ref{imp:thm:left-continuity} gives left continuity on
 \(\mathrm I_0\). This completes the proof.
\end{proof}

\subsection{Equality of critical values}
\label{sec:continuum-thresholds}
\label{sec:b2b-identification}
We prove Theorem~\ref{thm:continuum-thresholds} in this subsection.
We first establish the easy equality
\(\alpha_\ast=\alpha_{\mathrm{B2B}}\), which relies on the following
property in the subcritical regime.

\begin{lemma}\label{lem:fixed-ratio-bdelta-connection}
For \(\alpha<\alpha_\ast\), we have
\(\Pp_\alpha\bigl(S_1\xleftrightarrow{\LL_{\R^3}^{\alpha}}S_2\bigr)<1\).
\end{lemma}

\begin{proof}
By the definition of \(\alpha_\ast\) and scale invariance, there is \(\delta\in(0,1)\)
such that
\(\Pp_\alpha\bigl(S_1\xleftrightarrow{\LL_{\R^3}^{\alpha,\le\delta}}S_2\bigr)<1\).
By \eqref{eq:large-loop-terminal-comparison}, the loop measure of loops of
diameter larger than \(\delta\) that intersect \(A(1,2)\) is finite. Thus the
event \(F_\delta\) that no such loop occurs
has strictly positive probability.  The event \(F_\delta\) and the failure of
the displayed small-loop crossing are decreasing, and their intersection
precludes a crossing from \(S_1\) to \(S_2\).  Lemma~\ref{lem:poisson-fkg}
therefore yields
\[
\Pp_\alpha\bigl(\{S_1\xleftrightarrow{\LL_{\R^3}^{\alpha}}S_2\}^{\mathrm c}\bigr)
\ge \Pp_\alpha\bigl(\{S_1\xleftrightarrow{\LL_{\R^3}^{\alpha,\le\delta}}S_2\}^{\mathrm c}\bigr)
\Pp_\alpha(F_\delta)>0,
\]
implying the result.
\end{proof}

\begin{lemma}\label{lem:unit-truncated-b2b-threshold}
We have \(\alpha_\ast=\alpha_{\mathrm{B2B}}\).
\end{lemma}

\begin{proof}
For every \(\alpha\in \mathrm{I}_0\), Lemma~\ref{lem:fixed-ratio-bdelta} gives
\(\Pp_\alpha\bigl(S_r\xleftrightarrow{\LL_{\R^3}^{\alpha,\le1}}S_{2r}\bigr)
\to0\) as \(r\to\infty\).
Hence \(\alpha\le\alpha_\ast\).  Taking the supremum over \(\mathrm{I}_0\) and using
\(\sup \mathrm{I}_0=\alpha_{\mathrm{B2B}}\) gives
\(\alpha_{\mathrm{B2B}}\le\alpha_\ast\).

Conversely, fix \(0<\alpha<\alpha_\ast\).  By
Lemma~\ref{lem:fixed-ratio-bdelta-connection}, one has \(p_2^{\alpha}<1\).  If
\(\alpha>\alpha_{\mathrm{B2B}}\),
\cite[Lemma~6.2]{JL26} gives
\(\Pp_\alpha\bigl(\forall r>1,\,
S_1\xleftrightarrow{\LL_{\R^3}^{\alpha}}S_r\bigr)=1\), contradicting \(p_2^{\alpha}<1\).
Therefore \(\alpha\le\alpha_{\mathrm{B2B}}\).  Letting
\(\alpha\uparrow\alpha_\ast\) proves the reverse inequality.
\end{proof}

Now, we focus on the inequality \(\alpha_{\mathrm{tr}}\le\alpha_{\mathrm{B2B}}\),
which is the key to the remaining assertions of
Theorem~\ref{thm:continuum-thresholds}.
Its proof relies on  an almost-sure local connection property of Brownian
loops inside an independent Brownian loop soup; see Lemma~\ref{lem:threshold-attachment}.
To establish this property, we first generalize the one-arm lower bound below, and then
 establish some regularity properties of Brownian motion.
In what follows, we fix \(\varepsilon\in\mathrm I_0\) and write
\(\xi=\xi(\varepsilon)\in(0,1)\). All constants are allowed to depend on \(\varepsilon\).

\begin{lemma}
\label{lem:threshold-seed-estimates}
For \(\kappa>0\), there exists $c>0$ such that the following holds. Let \(M\subset B_1\) be compact with
\(\operatorname{cap}(M)\ge\kappa\), let \(|w|=3\), and let
\(x\in B(w,1/8)\). Uniformly for \(0<r\le10^{-3}\) and compact
\(A\subset B(x,r)\) with \(\operatorname{cap}(A)\ge\kappa r\),
\[
 \mathbb P_\varepsilon
 \bigl(A\xleftrightarrow{\mathcal L_{B_{16}}^\varepsilon}M\bigr)
 \ge c\,
r^{\xi(\varepsilon)}.
\]
\end{lemma}

\begin{proof}
In fact, the proof of
Theorem~\ref{thm:main} extends uniformly to compact sets
\(A_0\subset B_1\) with \(\operatorname{cap}(A_0)\ge\kappa\),
which allows us to conclude that with probability of order $r^{\xi(\varepsilon)}$,
$A$ is connected to some loop-piece $J\subset A(x,1/2,1)$ with $\operatorname{cap}(J)\ge\kappa_0$.
Clearly, $J$ can
be further connected to $M$ by some loop in $\mathcal L_{B_{16}}^\varepsilon$ with positive probability depending only on $\kappa,\kappa_0,\varepsilon$.
 The FKG inequality then
yields the lower bound.
\end{proof}

Let \(W\) be the standard Brownian motion in $\mathbb R^3$ started at zero and
\(\tau_r=\inf\{t\ge0:|W_t|=r\}\).
There exists
\(\kappa_1>0\) such that
\begin{equation}\label{eq:bm-stopped-capacity-positive}
 p:=\mathbb P\left(
 \operatorname{cap}(W[0,\tau_1])>2\kappa_1
 \right)>0.
\end{equation}
For any
compact \(K\subset\mathbb R^3\), define the set of good points in \(K\) by
\begin{equation}\label{eq:good-point}
 G(K)=\left\{x\in K:
 \limsup_{n\to\infty}
 2^n\operatorname{cap}\bigl(K\cap B(x,2^{-n})\bigr)
 \ge2\kappa_1\right\}.
\end{equation}
For $j=1,2,3$, let \(\pi_j\) be the \(j\)-th coordinate projection.
Write \(\operatorname{Leb}(\cdot)\) for Lebesgue measure on \(\mathbb R\).
We will show that, almost surely, for every subsegment \(K=W([s,t])\), \(0\le s\le t<\infty\), the set of good points is ``full'' in each coordinate projection.

\begin{lemma}
\label{lem:threshold-bm-good-levels}
Almost surely,
all subsegments \(K=W([s,t])\), \(0\le s\le t<\infty\), simultaneously satisfy
\begin{equation}\label{eq:bm-good-projections}
 \operatorname{Leb}\bigl(\pi_j(K)\setminus\pi_j(G(K))\bigr)=0,
 \qquad j=1,2,3.
\end{equation}
\end{lemma}

\begin{proof}
By Brownian scaling, with \(\kappa_1,p\) in \eqref{eq:bm-stopped-capacity-positive}, we have
\[
 \mathbb P\left(
 \operatorname{cap}(W[0,\tau_{2^{-n}}])>2\kappa_1\,2^{-n}
 \right)=p,\qquad n\ge1.
\]
By Fatou's lemma, these inequalities hold infinitely often with probability at least
\(p\). Hence, Blumenthal's zero--one law gives that
for every \(t>0\),
\begin{equation}\label{eq:GW}
 \mathbb P\bigl(0\in G(W[0,t])\text{ for every }t>0\bigr)=1.
\end{equation}
For \(a\ge0\) and \(u\in\mathbb R\), define the first hitting time at level $u$ in the \(j\)-th coordinate
direction
\[
 H^j_{a,u}=\inf\{r\ge a:\pi_j(W_r)=u\},
 \qquad \inf\varnothing=\infty,
\]
and collect the exceptional levels over all rational intervals:
\begin{equation}\label{eq:Nj}
 N_j(W)
 =\bigcup_{\substack{a,b\in\mathbb Q\\0<a<b<\infty}}
 \left\{u:H^j_{a,u}<b,\ 
 W_{H^j_{a,u}}\notin G(W([a,b]))\right\}.
\end{equation}

By the strong Markov property and \eqref{eq:GW},
on the event \(\{H^j_{a,u}<b\}\), \(W_{H^j_{a,u}}\) is almost surely good for
\(W([a,b])\). For each fixed \(u\), the union over rational \(a,b\)
 is countable. Hence \(\mathbb P(u\in N_j(W))=0\), and we get
\[
 \mathbb E\bigl[\operatorname{Leb}(N_j(W))\bigr]
 =\int_{\mathbb R}\mathbb P(u\in N_j(W))\,du=0.
\]
Thus \(\operatorname{Leb}(N_j(W))=0\) for
\(j=1,2,3\) simultaneously, almost surely.

 Now we treat an arbitrary time interval $[s,t]$ using rational subintervals.
Fix a continuous sample path \(W\) satisfying
\(\operatorname{Leb}(N_j(W))=0\) for all \(j=1,2,3\), and take
\(K=W([s,t])\), \(0\le s<t<\infty\).
 For any level \(u\) in the interior of $\pi_j(K)$ with
\(u\notin N_j(W)\cup\{\pi_j(W_s)\}\), we have
\(H^j_{s,u}\in(s,t)\).
Choose rational \(a,b\) with \(s<a<H^j_{s,u}<b<t\). Since \(H^j_{s,u}\) is the first hit after
\(s\), we have \(H^j_{a,u}=H^j_{s,u}\). The definition of \(N_j(W)\) and
monotonicity of capacity then give
\[
 W_{H^j_{s,u}}\in G(W([a,b]))\subset G(K),
 \qquad \pi_j(W_{H^j_{s,u}})=u.
\]
It follows that
\[
 \pi_j(K)\setminus\pi_j(G(K))
 \subset N_j(W)\cup\{\pi_j(W_s)\}\cup\partial\pi_j(K).
\]
The right-hand side is null, since the last two sets contain at most
three levels. The case \(s=t\) is immediate, completing the proof.
\end{proof}

Now, we extend the result to the Brownian loop soup.
\begin{lemma}
\label{lem:threshold-good-points}
Let \(\kappa_1\) and \(G\) be as in Lemma~\ref{lem:threshold-bm-good-levels}.
For \(\mu_{\mathbb R^3}^{\mathrm{loop}}\)-almost every loop
\(\ell:[0,T]\to\mathbb R^3\), all subsegments
\(K=\ell([s,t])\), \(0\le s\le t\le T\), simultaneously satisfy
\begin{equation}\label{eq:loop-good-projections}
 \operatorname{Leb}\bigl(\pi_j(K)\setminus\pi_j(G(K))\bigr)=0,
 \qquad j=1,2,3.
\end{equation}
Moreover, for every fixed \(\alpha>0\), almost surely
\eqref{eq:loop-good-projections} holds simultaneously for every connected union
\(K=\bigcup_{i=1}^{q}\ell_i([s_i,t_i])\) of finitely many subsegments
of loops \(\ell_i\in\mathcal L_{\mathbb R^3}^{\alpha}\).
\end{lemma}

\begin{proof}
Since the Brownian bridge is locally absolutely continuous with respect to
the standard Brownian motion, Lemma~\ref{lem:threshold-bm-good-levels}
gives the first assertion.
By Lemma~\ref{lem:loop-soup-palm} (Palm's formula), almost surely
\eqref{eq:loop-good-projections} holds simultaneously for all subsegments
of every loop in \(\mathcal L_{\mathbb R^3}^{\alpha}\).
Write \(K_i=\ell_i([s_i,t_i])\) and \(K=\bigcup_{i=1}^qK_i\). Since $G(K_i)\subset G(K)$ and
\[
\pi_j(K)\setminus\pi_j(G(K))
 \subset\bigcup_{i=1}^q
 \bigl(\pi_j(K_i)\setminus\pi_j(G(K_i))\bigr),
\]
 the result follows immediately.
\end{proof}

We next use the good-point property in \eqref{eq:good-point} to obtain a uniform lower bound
on connection probabilities.

\begin{lemma}
\label{lem:threshold-attachment}
Fix \(\kappa>0\). Let \(M\subset B_1\) be compact with
\(\operatorname{cap}(M)\ge\kappa\), and let \(|w|=3\).
Let \(K\subset B(w,1/16)\) be a compact connected set
satisfying \eqref{eq:loop-good-projections}
and \(\operatorname{diam}(K)\ge3/64\).
For a Brownian loop soup \(\mathcal L_{B_{16}}^\varepsilon\) of intensity
\(\varepsilon\) in \(B_{16}\), we have
\[
 \mathbb P_\varepsilon
 \bigl(K\xleftrightarrow{\mathcal L_{B_{16}}^\varepsilon}M\bigr)
 \ge c(\varepsilon,\kappa)>0,
\]
uniformly over \(K,M,w\).
\end{lemma}

\begin{proof}
Since \(\operatorname{diam}(K)\ge3/64\),
we can choose a coordinate projection
\(J=\pi_j(K)\) with
\(|J|\ge\operatorname{diam}(K)/\sqrt3\ge3/(64\sqrt3)\).
By \eqref{eq:loop-good-projections} and the definition of \(G(K)\),
there is a set \(J_0\subset J\) with \(\operatorname{Leb}(J\setminus J_0)=0\) satisfying
the following property.
 For any $u\in J_0$, there are
\(x(u)\in G(K)\) and a dyadic \(r(u)\in(0,10^{-3}]\) such that
\(\pi_j(x(u))=u\) and \(\operatorname{cap}(A_u)>\kappa_1r(u)/4\),
where \(A_u=K\cap B(x(u),r(u)/2)\).
Set
\[
 E_u=\bigl\{A_u\xleftrightarrow{\mathcal L_{B_{16}}^\varepsilon}M\bigr\},
 \qquad
 Z=\frac1{|J|}\int_{J_0}r(u)^{-\xi}\mathbf 1_{E_u}\,du.
\]
Applying Lemma~\ref{lem:threshold-seed-estimates}, we obtain that
\[
 \mathbb E_\varepsilon Z
 =\frac1{|J|}\int_{J_0}r(u)^{-\xi}\mathbb P_\varepsilon(E_u)\,du\ge c.
\]
For \(u\ne v\), we have \(|x(u)-x(v)|\ge|u-v|\).
By the two-point estimate \eqref{eq:seeded-second-moment},
\[
 \frac{\mathbb P_\varepsilon(E_u\cap E_v)}{r(u)^\xi r(v)^\xi}
 \le C\bigl(|x(u)-x(v)|\vee r(u)\vee r(v)\bigr)^{-\xi}
 \le C|u-v|^{-\xi}.
\]
Moreover, since
\(\xi<1\), we get
\[
 \mathbb E_\varepsilon Z^2
 \le\frac C{|J|^2}\int_J\int_J|u-v|^{-\xi}\,du\,dv
 =\frac{2C |J|^{-\xi}}{(1-\xi)(2-\xi)}\le C'.
\]
Now the Paley--Zygmund inequality gives that
\[
 \mathbb P_\varepsilon(K\leftrightarrow M)
 \ge\mathbb P_\varepsilon(Z>0)
 \ge\frac{(\mathbb E_\varepsilon Z)^2}{\mathbb E_\varepsilon Z^2}
 \ge c'>0.
\]
The constants are independent of \(K,M,w\)
and of the chosen radii.
\end{proof}

We also need the following local connectedness property for nonempty open sets
when \(\alpha>\alpha_{\mathrm{B2B}}\).

\begin{lemma}
\label{lem:local-open-set-connections}
Fix \(\alpha>\alpha_{\mathrm{B2B}}\), \(z\in\mathbb R^3\), and \(h>0\).
 Almost surely, all pairs of nonempty open sets
\(U,V\subset B(z,h)\) are simultaneously connected in $\mathcal L_{B(z,10h)}^\alpha$.
\end{lemma}

\begin{proof}
For fixed nonempty open \(U,V\subset B(z,h)\), choose
\(B(x,a)\subset U\) and \(B(y,b)\subset V\), with
\(R=|x-y|>0\) and \(\max(a,b)<R/4\).
By \cite[Lemma~6.2]{JL26},
\(B(x,a)\) and \(S(x,R)\)  are connected in $\mathcal L_{\mathbb R^3}^{\alpha}$ almost surely.
By Lemma~\ref{lem:one-sided-surgery} and \(R\le2h\), with positive probability, there is no loop connecting $B(x,R)$ and \(\partial B(x,4h)\).
Independence therefore gives
\begin{equation}\label{eq:Bxa}
 \mathbb P_\alpha\bigl(
 B(x,a)\xleftrightarrow{\mathcal L_{B(x,4h)}^\alpha}S(x,R)
 \bigr)=1.
\end{equation}
Now, we choose $N$ rotations of \(B(y,b)\) about \(x\) that
cover \(S(x,R)\). By rotation invariance, the probabilities of
connecting \(B(x,a)\) to these targets within \(B(x,4h)\)
 are equal; denote their common value by \(p\).
Since at least one of the targets is connected to \(B(x,a)\) almost surely by \eqref{eq:Bxa},
the FKG inequality (Lemma~\ref{lem:poisson-fkg}) gives \(0\ge(1-p)^N\).
Thus \(p=1\).
Since \(B(x,4h)\subset B(z,10h)\), this proves
the desired connection for fixed \(U,V\).
Taking a countable intersection over pairs of balls in a countable
basis of \(B(z,h)\) proves the assertion simultaneously for all
nonempty open subsets.
\end{proof}

 Finally, we combine the above ingredients to complete the proof.
\begin{proof}[Proof of Theorem~\ref{thm:continuum-thresholds}]
The identity \(\alpha_\ast=\alpha_{\mathrm{B2B}}\) follows from
Lemma~\ref{lem:unit-truncated-b2b-threshold}.
According to \eqref{eq:four-critical}, it remains to show \(\alpha_{\mathrm{tr}}\le\alpha_{\mathrm{B2B}}\).
Fix \(\beta>\alpha_{\mathrm{B2B}}\), choose \(\varepsilon\in\mathrm I_0\)
 sufficiently small that \(\alpha=\beta-\varepsilon>\alpha_{\mathrm{B2B}}\),
and write \(\mathcal L_D^\beta=\mathcal L_D^\alpha\cup\mathcal L_D^\varepsilon\),
where \(\mathcal L_D^\alpha\) and \(\mathcal L_D^\varepsilon\) are independent
in a fixed nonempty connected open set \(D\).
Fix a ball \(B(z,h)\) with \(B(z,10h)\subset D\).
By Lemma~\ref{lem:local-open-set-connections}, almost surely every two
nonempty open subsets of \(B(z,h)\) are connected in \(\mathcal L_{B(z,10h)}^\alpha\).
We aim to show that, for \((\mu_D^{\Loop})^{\otimes2}\)-almost every pair
\((\gamma_1,\gamma_2)\) of disjoint loops meeting \(B(z,h)\), it holds that
\begin{equation}\label{eq:fixed-anchor-connection}
 \mathbb P_\beta\bigl(
 \gamma_1\xleftrightarrow{\mathcal L_D^\beta}\gamma_2
 \bigr)=1.
\end{equation}
The probability is taken over \(\mathcal L_D^\beta\), with
\(\gamma_1,\gamma_2\) fixed.

Before proving \eqref{eq:fixed-anchor-connection}, we first explain
how to get Theorem~\ref{thm:continuum-thresholds} from it.
Applying the Palm formula (Lemma~\ref{lem:loop-soup-palm})
to \(\mathcal L_D^\beta\), we get from \eqref{eq:fixed-anchor-connection} that almost surely all loops of
\(\mathcal L_D^\beta\) meeting \(B(z,h)\) belong to the same cluster.
Covering \(D\) by countably many such balls shows that all loops of \(\mathcal L_D^\beta\) belong to a single
cluster.
Since this holds for every fixed \(D\) and
\(\beta>\alpha_{\mathrm{B2B}}\), the local-uniqueness characterization
in \cite[Theorem~1.5]{JL26} gives
\(\alpha_{\mathrm{tr}}\le\alpha_{\mathrm{B2B}}\).
This proves all the claimed
equalities.

To prove \eqref{eq:fixed-anchor-connection}, we
use Lemma~\ref{lem:threshold-good-points} to select sufficiently
many pairs of good points on \(\gamma_1\) and \(\gamma_2\), respectively.
For each pair, we choose two small balls with centers slightly displaced
from the respective good points, and these balls are connected in
\(\mathcal L_{B(z,10h)}^\alpha\) by Lemma~\ref{lem:local-open-set-connections}.
Then, by Lemma~\ref{lem:threshold-attachment}, we can use the additional independent loop soup $\mathcal L_D^\varepsilon$
to provide the remaining desired connection to $\gamma_j$:
\begin{equation}\label{eq:trial-connection}
 \gamma_1
 \xleftrightarrow{\mathcal L_D^\varepsilon}K_{1,n}
 \xleftrightarrow{\mathcal L_{B(z,10h)}^\alpha}K_{2,n}
 \xleftrightarrow{\mathcal L_D^\varepsilon}\gamma_2,
\end{equation}
where \(K_{1,n}\) and \(K_{2,n}\) are unions of segments of loops near
\(\gamma_1\) and \(\gamma_2\), respectively, chosen from the same finite
loop chain in \(\mathcal L_{B(z,10h)}^\alpha\), as detailed below.
See Figure~\ref{fig:local-attachment-trial} for an illustration.

\begin{figure}[!t]
\centering
\begin{tikzpicture}[x=1cm,y=1cm,
  every node/.style={font=\small},
  anchorloop/.style={black!80,line width=0.9pt},
  middlechain/.style={blue!65!black,line width=1.1pt},
  localpiece/.style={blue!65!black,line width=2.0pt},
  anchorpiece/.style={green!45!black,line width=2.2pt},
  attachment/.style={orange!85!black,line width=1.4pt},
  neighborhood/.style={black!45,dashed,line width=0.6pt}]

\fill[black!2] (2.1,0) circle [radius=2.05];
\fill[black!2] (10.5,0) circle [radius=2.05];
\draw[neighborhood] (2.1,0) circle [radius=2.05];
\draw[neighborhood] (10.5,0) circle [radius=2.05];
\node[black!65] at (2.1,-2.36) {\(B(x_{1,n},16r_{1,n})\)};
\node[black!65] at (10.5,-2.36) {\(B(x_{2,n},16r_{2,n})\)};

\draw[anchorloop]
  (2.1,0) .. controls (1.86,0.31) and (2.12,0.88) .. (1.48,0.96)
  .. controls (0.84,1.04) and (0.87,0.75) .. (0.55,0.75)
  .. controls (0.23,0.75) and (-0.45,0.87) .. (-0.45,0.55)
  .. controls (-0.45,0.23) and (0.05,0.15) .. (0.05,-0.05)
  .. controls (0.05,-0.25) and (-0.40,-0.27) .. (-0.40,-0.55)
  .. controls (-0.40,-0.83) and (-0.15,-1.05) .. (0.30,-1.05)
  .. controls (0.75,-1.05) and (1.05,-0.85) .. (1.45,-0.65)
  .. controls (1.85,-0.45) and (2.34,-0.31) .. (2.1,0);
\draw[anchorloop]
  (10.5,0) .. controls (10.67,0.50) and (10.14,0.96) .. (9.77,0.78)
  .. controls (9.34,0.63) and (9.54,1.22) .. (8.92,1.14)
  .. controls (8.18,1.04) and (7.64,0.70) .. (7.91,0.23)
  .. controls (8.10,-0.06) and (8.94,0.15) .. (8.80,-0.38)
  .. controls (8.59,-1.18) and (9.40,-1.23) .. (9.78,-0.88)
  .. controls (10.15,-0.59) and (10.05,-0.11) .. (10.5,0);
\node at (-0.55,1.48) {\(\gamma_1\)};
\node at (8.40,1.28) {\(\gamma_2\)};

\draw[green!45!black,densely dotted,line width=0.8pt]
  (2.1,0) circle [radius=0.56];
\draw[green!45!black,densely dotted,line width=0.8pt]
  (10.5,0) circle [radius=0.56];
\begin{scope}
\clip (2.1,0) circle [radius=0.56];
\draw[anchorpiece]
  (2.1,0) .. controls (1.86,0.31) and (2.12,0.88) .. (1.48,0.96)
  .. controls (0.84,1.04) and (0.87,0.75) .. (0.55,0.75)
  .. controls (0.23,0.75) and (-0.45,0.87) .. (-0.45,0.55)
  .. controls (-0.45,0.23) and (0.05,0.15) .. (0.05,-0.05)
  .. controls (0.05,-0.25) and (-0.40,-0.27) .. (-0.40,-0.55)
  .. controls (-0.40,-0.83) and (-0.15,-1.05) .. (0.30,-1.05)
  .. controls (0.75,-1.05) and (1.05,-0.85) .. (1.45,-0.65)
  .. controls (1.85,-0.45) and (2.34,-0.31) .. (2.1,0);
\end{scope}
\begin{scope}
\clip (10.5,0) circle [radius=0.56];
\draw[anchorpiece]
  (10.5,0) .. controls (10.67,0.50) and (10.14,0.96) .. (9.77,0.78)
  .. controls (9.34,0.63) and (9.54,1.22) .. (8.92,1.14)
  .. controls (8.18,1.04) and (7.64,0.70) .. (7.91,0.23)
  .. controls (8.10,-0.06) and (8.94,0.15) .. (8.80,-0.38)
  .. controls (8.59,-1.18) and (9.40,-1.23) .. (9.78,-0.88)
  .. controls (10.15,-0.59) and (10.05,-0.11) .. (10.5,0);
\end{scope}
\node[green!45!black] at (1.25,-0.32) {\(M_{1,n}\)};
\node[green!45!black] at (9.50,-0.30) {\(M_{2,n}\)};
\draw[green!45!black,line width=0.45pt] (1.65,-0.32) -- (2.16296,-0.16744);
\draw[green!45!black,line width=0.45pt] (9.90,-0.30) -- (10.22,-0.18);
\fill (2.1,0) circle [radius=1.5pt];
\fill (10.5,0) circle [radius=1.5pt];
\node at (2.55,-0.88) {\(x_{1,n}\)};
\node at (10.95,-0.88) {\(x_{2,n}\)};

\draw[blue!50!black,densely dashed,line width=0.65pt]
  (3.3,0) circle [radius=0.34];
\draw[blue!50!black,densely dashed,line width=0.65pt]
  (11.7,0) circle [radius=0.34];
\draw[localpiece]
  (3.21,-0.2175) .. controls (3.345,-0.1125) and (3.195,0.06) .. (3.3375,0.2175);
\draw[localpiece]
  (11.61,-0.2175) .. controls (11.745,-0.1125) and (11.595,0.06) .. (11.7375,0.2175);
\node[blue!65!black] at (3.45,-0.65) {\(K_{1,n}\)};
\node[blue!65!black] at (11.85,-0.65) {\(K_{2,n}\)};

\draw[middlechain]
  (3.3375,0.2175) .. controls (4.15,1.34) and (5.45,0.75) .. (6.55,1.20)
  .. controls (7.75,1.82) and (8.73,1.60) .. (9.63,1.64)
  .. controls (10.64,1.69) and (11.95,1.30) .. (11.7375,0.2175);
\node[blue!65!black] at (6.90,2.02)
  {\(\mathcal L_{B(z,10h)}^\alpha\)};

\draw[attachment]
  (1.991875,0.269531) .. controls (2.35,0.50) and (2.51,-0.25) .. (2.84,0.03)
  .. controls (2.98,0.16) and (3.15,0.03) .. (3.2775,-0.045);
\draw[attachment]
  (10.52488,0.2904) .. controls (10.82,0.52) and (10.91,-0.25) .. (11.24,0.03)
  .. controls (11.38,0.16) and (11.55,0.03) .. (11.6775,-0.045);
\node[orange!85!black] at (2.72,0.83) {\(E_{1,n}\)};
\node[orange!85!black] at (11.12,0.83) {\(E_{2,n}\)};
\end{tikzpicture}
\caption{Illustration of \eqref{eq:trial-connection}. The blue curve represents a finite
chain of loops in \(\mathcal L_{B(z,10h)}^\alpha\) joining \(K_{1,n}\) and \(K_{2,n}\).
The green portions are \(M_{j,n}=\gamma_j\cap B(x_{j,n},r_{j,n})\),
with these balls indicated by green dotted circles.
 The orange curves represent connections in \(\mathcal L_D^\varepsilon\) that join \(M_{j,n}\) to
\(K_{j,n}\) inside disjoint gray balls $B(x_{j,n},16r_{j,n})$.
The blue circles indicate \(B(w_{j,n},r_{j,n}/16)\).}
\label{fig:local-attachment-trial}
\end{figure}
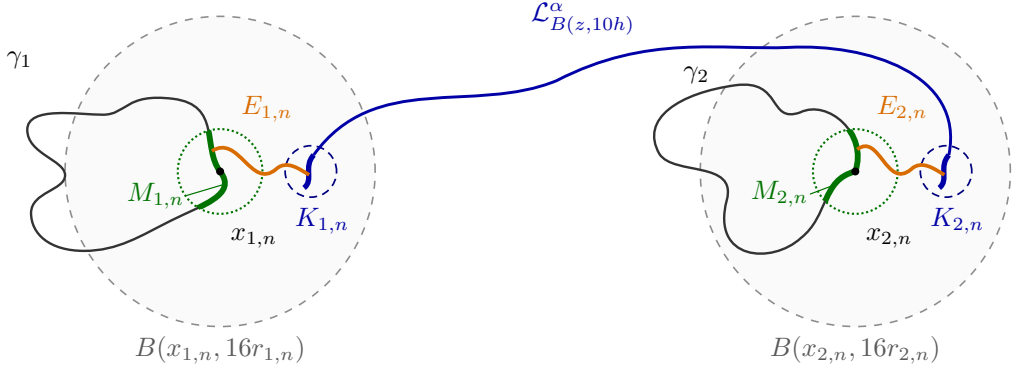

Concretely, Lemma~\ref{lem:threshold-good-points}
allows us to choose, for \(j=1,2\) and \(n\ge1\), centers
\(x_{j,n}\in\gamma_j\cap B(z,h)\) and dyadic radii \(r_{j,n}\) such that
all the balls \(B(x_{j,n},100r_{j,n})\) are disjoint
and contained in \(B(z,h)\), while
\[
 M_{j,n}:=\gamma_j\cap B(x_{j,n},r_{j,n}),
 \qquad \operatorname{cap}(M_{j,n})\ge\kappa_1r_{j,n}.
\]
Put \(w_{j,n}=x_{j,n}+3r_{j,n}e_1\), where \(e_1=(1,0,0)\).
By Lemma~\ref{lem:local-open-set-connections}, $\mathbb P_\alpha$-almost surely, for each \(n\), there is a finite chain of loops  in \(\mathcal L_{B(z,10h)}^\alpha\) joining
\(B(w_{1,n},r_{1,n}/256)\) to \(B(w_{2,n},r_{2,n}/256)\).
Hence, we can find along this chain two connected sets $(K_{j,n})_{j=1,2}$ made of finitely many segments
of loops in \(\mathcal L_{B(z,10h)}^\alpha\),  which satisfy
\[
 K_{j,n}\subset B(w_{j,n},r_{j,n}/16),
 \qquad \operatorname{diam}(K_{j,n})\ge3r_{j,n}/64,
\]
and belong to the same finite chain.
Next, we condition on \(\mathcal L_D^\alpha\).
For \(j=1,2\) and \(n\ge1\), define the connection event
\[
 E_{j,n}:=\left\{
 M_{j,n}\xleftrightarrow{\mathcal L_{B(x_{j,n},16r_{j,n})}^\varepsilon}K_{j,n}
 \right\}.
\]
By scaling and Lemma~\ref{lem:threshold-attachment},
\(\mathbb P_\beta(E_{j,n}\mid\mathcal L_D^\alpha)\ge c>0\),
uniformly in \(j,n\).
Since the balls \(B(x_{j,n},16r_{j,n})\) are disjoint, the events
\(E_{j,n}\) are conditionally independent.
On \(E_{1,n}\cap E_{2,n}\), the two loops \(\gamma_1\) and \(\gamma_2\)
are connected in \(\mathcal L_D^\beta\). Hence,
\[
 \mathbb P_\beta\left(
 \bigcap_{n=1}^{N}(E_{1,n}\cap E_{2,n})^{\mathrm c}
 \,\middle|\,\mathcal L_D^\alpha\right)
 \le(1-c^2)^N\longrightarrow0
 \qquad\text{as }N\to\infty.
\]
 Taking expectations over \(\mathcal L_D^\alpha\) gives the claim \eqref{eq:fixed-anchor-connection} and finishes the proof.
\end{proof}

\section{Sharpness for the unit-truncated model}
\label{sec:unit-truncated-sharpness}

This section presents the proof of Theorem~\ref{thm:unit-truncated-sharpness} as promised.
In Section~\ref{ut:sec:osss}, we use the standard OSSS inequality  (see, e.g.,
\cite{OSSS05,DRT17,DRT18}) to prove the differential inequality for the loop soup with a lower cutoff in Lemma~\ref{ut:prop:diff-cutoff}.
In Section~\ref{ut:sec:cutoff-sharpness}, we remove the lower cutoff and then prove the sharpness of the unit-truncated threshold in Proposition~\ref{ut:prop:analytic-application}.
Finally, Section~\ref{ut:sec:threshold-identification} proves Theorem~\ref{thm:unit-truncated-sharpness} by using ergodicity and the sharpness result.

\subsection{Influence, revealment, and the OSSS inequality}
\label{ut:sec:osss}

We regard every loop as rooted and split the loop soup according to the root
and diameter of each loop.
Write \(\ell(0)\) for the root of \(\ell\).  For \(k\ge0\) and
\(z\in\Z^3\), let \(Q_{k,z}:=2^{-k}(z+[0,1)^3)\), and define
\begin{equation}\label{ut:eq:dyadic-coordinate}
        \Gamma_{k,z}
        :=
        \Bigl\{\ell:
        2^{-(k+1)}<\diam(\ell)\le 2^{-k},\ 
        \ell(0)\in Q_{k,z}\Bigr\}.
\end{equation}
For \(N\ge0\), write $\LL_{\R^3}^{\alpha,N}
        :=
        \{\ell\in\LL_{\R^3}^\alpha:2^{-N-1}<\diam(\ell)\le1\}$  for the loop soup truncated from above and below.
For \(N\ge0\) and \(r>4\), we define the one-arm event and probability associated with $\LL_{\R^3}^{\alpha,N}$ by
\[
        \mathcal E_r^N
        :=
        \{S_1\xleftrightarrow{\LL_{\R^3}^{\alpha,N}}S_r\},
        \qquad
        \theta_r^N(\alpha):=\Pp_\alpha(\mathcal E_r^N).
\]
Define the finite
index set
\[
        \mathscr I_{N,r}
        :=
        \bigl\{(k,z):0\le k\le N,\ z\in\Z^3,
        Q_{k,z}\cap B_{r+2}\ne\varnothing\bigr\}.
\]
For \((k,z)\in\mathscr I_{N,r}\), set
\(\omega_{k,z}:=\LL_{\R^3}^\alpha\cap\Gamma_{k,z}\) and
\(\omega:=(\omega_{k,z})_{(k,z)\in\mathscr I_{N,r}}\).
Note that all \(\muLoop(\Gamma_{k,z})\) are the same and finite.  Since the classes
\(\Gamma_{k,z}\) are disjoint, these coordinates are independent.  Moreover,
\(\mathcal E_r^N\) is determined by $\omega$, the collection of loops in $\LL_{\R^3}^{\alpha,N}$
that intersect the ball $B_{r+2}$.

Let \(\omega_{k,z}'\) be an
independent copy of \(\omega_{k,z}\), independent of \(\omega\), and let
\(\omega^{(k,z)}\) be obtained by replacing \(\omega_{k,z}\) with
\(\omega_{k,z}'\) while leaving all other coordinates unchanged.  For an
increasing event \(A\) determined by these coordinates, write its indicator as
\(A(\omega)\in\{0,1\}\) and define the influence of coordinate \((k,z)\) by
\[
        \operatorname{Inf}_{k,z}^\alpha(A)
        :=
        \Pp_\alpha(A(\omega)\ne A(\omega^{(k,z)})).
\]
The next estimate relates this influence to the pivotal probabilities appearing
in Russo's formula.

\begin{lemma}\label{ut:lem:poisson-poincare}
For every increasing event \(A\) determined by $\omega$ and every
\((k,z)\in\mathscr I_{N,r}\),
\begin{equation}\label{ut:eq:influence-pivotal}
        \operatorname{Inf}_{k,z}^\alpha(A)
        \le
        2\alpha
        \int_{\Gamma_{k,z}}
        \Pp_\alpha(\operatorname{Piv}_\ell A)\,\muLoop(\dint\ell).
\end{equation}
\end{lemma}

\begin{proof}
Fix \(i=(k,z)\), and let \(\omega_{-i}\) denote the collection of all
coordinates other than \(i\).  Condition on \(\omega_{-i}\).  For a finite
loop configuration \(\eta\) on \(\Gamma_{k,z}\), set
\[
        F_{\omega_{-i}}(\eta)
        :=
        A(\omega_{-i},\eta).
\]
With \(\omega_{-i}\) fixed, write \(F=F_{\omega_{-i}}\).  Then \(F\) is an
increasing function of the finite loop soup \(\eta\) on \(\Gamma_{k,z}\), with
intensity \(\alpha\muLoop|_{\Gamma_{k,z}}\).  If \(\eta'\) is an independent
copy, then
\[
        \Pp_\alpha(F(\eta)\ne F(\eta'))
        =
        2\,\operatorname{Var}(F(\eta)).
\]
By the Poisson Poincar\'e inequality
\cite[Theorem~18.7]{LP17},
\[
        \operatorname{Var}(F(\eta))
        \le
        \alpha\int_{\Gamma_{k,z}}
        \Ee\bigl[(F(\eta+\delta_\ell)-F(\eta))^2\bigr]\,
        \muLoop(\dint\ell).
\]
Since \(F\) is an increasing \(\{0,1\}\)-valued function,
\[
        (F(\eta+\delta_\ell)-F(\eta))^2
        =
        F(\eta+\delta_\ell)-F(\eta).
\]
Averaging over the outside coordinates yields \eqref{ut:eq:influence-pivotal}.
\end{proof}

Applying Lemma~\ref{ut:lem:russo} (Russo's formula) to
\(\mathscr D=\bigcup_{(k,z)\in\mathscr I_{N,r}}\Gamma_{k,z}\) and summing
Lemma~\ref{ut:lem:poisson-poincare} over \((k,z)\in\mathscr I_{N,r}\), we obtain
\begin{equation}\label{ut:eq:sum-inf-russo}
        \sum_{(k,z)\in\mathscr I_{N,r}}
        \operatorname{Inf}_{k,z}^\alpha(A)
        \le
        2\alpha\,\theta_A'(\alpha),
\end{equation}
where we set \(\theta_A(\alpha):=\Pp_\alpha(A)\).

\medskip

We next use a random-sphere exploration as a randomized decision algorithm.
An algorithm \(T\) determining \(\mathcal E_r^N\) takes \(\omega\) as input and
reveals its coordinates \(\omega_{k,z}\) one by one.  The next coordinate may
depend on the coordinates revealed so far and on auxiliary randomness.  The
algorithm stops as soon as the indicator of \(\mathcal E_r^N\) is fixed,
regardless of the values of the remaining coordinates.  The probability below
averages over both the loop soup and the auxiliary randomness of \(T\).  The
revealment of coordinate \((k,z)\) is defined by
\[
        \delta_{k,z}(T)
        :=
        \Pp_\alpha\bigl(T\text{ reveals }\omega_{k,z}\bigr).
\]
We also extend \(\theta_r^N(\alpha)\) to \(0\le r\le4\) by setting
\(\theta_r^N(\alpha)=1\), and define
\[
        \Sigma_r^N(\alpha):=\sum_{j=0}^{\lfloor r\rfloor}\theta_j^N(\alpha).
\]

\begin{lemma}\label{ut:lem:revealment}
There is a universal constant \(C<\infty\) such that, for every \(N\ge0\) and
every \(r\ge8\), there is an algorithm \(T_{N,r}\) determining
\(\mathcal E_r^N\) for which the revealment of every coordinate
\((k,z)\in\mathscr I_{N,r}\)
satisfies
\begin{equation}\label{ut:eq:revealment}
        \delta_{k,z}(T_{N,r})
        \le
        C\frac{\Sigma_r^N(\alpha)}{r}.
\end{equation}
\end{lemma}

\begin{proof}
Choose \(J\) uniformly in \(\{2,\ldots,\lfloor r\rfloor-2\}\), independently
of the soup.  For each \((k,z)\), set
\[
        Q_{k,z}^+
        :=
        \{x:\dist(x,Q_{k,z})\le 2^{-k}\}.
\]
Every loop in \(\Gamma_{k,z}\) lies in \(Q_{k,z}^+\).  The algorithm proceeds
as follows.  First reveal every coordinate \((k,z)\) with
\(Q_{k,z}^+\cap S_J\ne\varnothing\), and mark exactly the revealed loops that
intersect \(S_J\).  Thereafter, whenever an unrevealed coordinate \((k,z)\) has
\(Q_{k,z}^+\) intersecting a marked loop, reveal it and mark every loop that is
connected to the marked set by a chain of loops
revealed so far.  Stop when no further coordinate can be revealed.  Since
\(\mathscr I_{N,r}\) is finite, the algorithm terminates and outputs \(1\) exactly when one
marked component intersects both \(S_1\) and \(S_r\), which determines
\(\mathcal E_r^N\).

Fix \((k,z)\) and condition on \(J=j\).  Let \(x_{k,z}\) be the center of
\(Q_{k,z}\), and let \(d_j=\dist(x_{k,z},S_j)\).  If \((k,z)\) is revealed,
then either \(d_j\le4\), or \(\LL_{\R^3}^{\alpha,N}\) connects
\(B(x_{k,z},3)\) to \(S_j\).  Hence, by the convention \(\theta_r^N=1\) for
\(r\le4\), translation invariance, and a union bound,
\[
        \Pp_\alpha((k,z)\text{ is revealed}\mid J=j)
        \le C\,\theta_{(d_j-4)_+}^N(\alpha).
\]
Averaging over \(J\) and using the monotonicity
of \(r\mapsto\theta_r^N(\alpha)\) yield
\[
        \delta_{k,z}(T_{N,r})
        \le
        \frac{C}{r}\sum_{j=2}^{\lfloor r\rfloor-2}
        \theta_{(d_j-4)_+}^N(\alpha)
        \le
        C\frac{\Sigma_r^N(\alpha)}{r},
\]
which proves \eqref{ut:eq:revealment}.
\end{proof}

Combining the previous results, we can now derive the following differential inequality.

\begin{lemma}
\label{ut:prop:diff-cutoff}
There is \(c>0\) such that, for every \(N\ge0\), every \(r\ge8\), and every
\(\alpha>0\),
\begin{equation}\label{ut:eq:diff-cutoff}
        \frac{\dint}{\dint\alpha}\theta_r^N(\alpha)
        \ge
        \frac{c}{\alpha}\,
        \frac{r}{\Sigma_r^N(\alpha)}
        \theta_r^N(\alpha)
        \bigl(1-\theta_r^N(\alpha)\bigr).
\end{equation}
\end{lemma}

\begin{proof}
Applying the OSSS inequality~\cite{OSSS05}, in the general product-space form of
\cite[Theorem~1.9]{DRT18}, to the algorithm \(T_{N,r}\) constructed in
Lemma~\ref{ut:lem:revealment}, we obtain that
\[
        \operatorname{Var}(\1_{\mathcal E_r^N})
        \le
        \sum_{(k,z)\in\mathscr I_{N,r}}\delta_{k,z}(T_{N,r})
        \operatorname{Inf}_{k,z}^\alpha(\mathcal E_r^N).
\]
By Lemma~\ref{ut:lem:revealment} and \eqref{ut:eq:sum-inf-russo},
\[
        \theta_r^N(1-\theta_r^N)
        \le C\frac{\Sigma_r^N}{r}
        \sum_{(k,z)\in\mathscr I_{N,r}}
        \operatorname{Inf}_{k,z}^\alpha(\mathcal E_r^N)
        \le 2\,C\,\alpha\,\frac{\Sigma_r^N}{r}
        \frac{\dint}{\dint\alpha}\theta_r^N(\alpha).
\]
Rearranging gives \eqref{ut:eq:diff-cutoff}.
\end{proof}

 Integrating the differential inequality compares the one-arm probabilities
at different intensities for the truncated loop soup.

\begin{lemma}
\label{ut:cor:diff-N}
For every \(N\ge0\), every \(r\ge8\), and every \(0<\alpha<\beta\),
\begin{equation}\label{ut:eq:integrated-N}
        \theta_r^N(\alpha)
        \le
        \theta_r^N(\beta)
        \exp\left\{
        -\,c\log\left(\frac\beta\alpha\right)\,
        \frac{r}{\Sigma_r^N(\beta)}
        \bigl(1-\theta_r^N(\beta)\bigr)
        \right\}.
\end{equation}
\end{lemma}

\begin{proof}
Dividing
\eqref{ut:eq:diff-cutoff} by \(\theta_r^N\) and integrating from \(\alpha\) to
\(\beta\), while using \(\Sigma_r^N(t)\le\Sigma_r^N(\beta)\) and
\(1-\theta_r^N(t)\ge1-\theta_r^N(\beta)\) for
\(t\in[\alpha,\beta]\), gives
\[
        \log\frac{\theta_r^N(\beta)}
                 {\theta_r^N(\alpha)}
        \ge
        c\log\left(\frac\beta\alpha\right)
        \frac{r}{\Sigma_r^N(\beta)}
        \bigl(1-\theta_r^N(\beta)\bigr).
\]
Exponentiating and rearranging give \eqref{ut:eq:integrated-N}.
\end{proof}

\subsection{Cutoff removal and sharpness}
\label{ut:sec:cutoff-sharpness}

We now remove the lower diameter cutoff by monotone convergence.
Let
\[
        \mathcal E_r^{\le1}
        :=
        \{S_1\xleftrightarrow{\LL_{\R^3}^{\alpha,\le1}}S_r\},
        \qquad
        \bar\theta_r(\alpha):=\Pp_\alpha(\mathcal E_r^{\le1}).
\]
For the one-arm sums we use the same small-radius convention as above:
\begin{equation}\label{ut:eq:theta-sigma-full}
        \theta_r(\alpha):=
        \begin{cases}
        1, & 0\le r\le4,\\
        \bar\theta_r(\alpha), & r>4,
        \end{cases}
        \qquad
        \Sigma_r(\alpha):=\sum_{j=0}^{\lfloor r\rfloor}\theta_j(\alpha).
\end{equation}
Then, we have the following estimate without the lower diameter cutoff.

\begin{lemma}\label{ut:lem:cutoff-limit}
For \(0<\alpha<\beta\) and \(r\ge8\),
\begin{equation}\label{ut:eq:integrated-full}
        \theta_r(\alpha)
        \le
        \theta_r(\beta)
        \exp\left\{
        -\,c\log\left(\frac\beta\alpha\right)\,
        \frac{r}{\Sigma_r(\beta)}
        \bigl(1-\theta_r(\beta)\bigr)
        \right\},
\end{equation}
with the same universal \(c>0\) as in
Lemma~\ref{ut:prop:diff-cutoff}.
\end{lemma}

\begin{proof}
Using
\(\Sigma_r^N(\beta)\le\Sigma_r(\beta)\) and
\(\theta_r^N(\beta)\le\theta_r(\beta)\) in \eqref{ut:eq:integrated-N}, we obtain
\[
        \theta_r^N(\alpha)
        \le
        \theta_r^N(\beta)
        \exp\left\{
        -\,c\log\left(\frac\beta\alpha\right)\,
        \frac{r}{\Sigma_r(\beta)}
        \bigl(1-\theta_r(\beta)\bigr)
        \right\}.
\]
Since \(\mathcal E_r^{\le1}=\bigcup_{N\ge0}\mathcal E_r^N\), letting
\(N\to\infty\) gives \eqref{ut:eq:integrated-full}.
\end{proof}

The following lemma adapts the analytic argument of
\cite[Lemma~3.1]{DRT17}.
However, in order to remove the cutoff, we need to work with the integrated estimate
\eqref{ut:eq:integrated-full} rather than the pointwise differential inequality \eqref{ut:eq:diff-cutoff}
as in \cite{DRT17}. Hence, the proof details differ from the original argument,
and we include a complete proof in Appendix~\ref{app:analytic-proof} for the reader's convenience.

\begin{lemma}\label{ut:lem:analytic}
For \(n\ge1\), let \(f_n:\R\longrightarrow[0,1]\) be increasing functions.
Assume that \(f_n(s)\) is nonincreasing in \(n\) and converges as
\(n\to\infty\) for every \(s\in\R\). Put
\(\Sigma_n(s):=\sum_{k=0}^{n-1}f_k(s)\), where \(f_0\equiv1\).
Assume that there are constants \(c_1>0\) and \(n_0<\infty\) such that, for
every \(n\ge n_0\) and every \(s<t\),
\begin{equation}\label{ut:eq:analytic-integrated}
        f_n(s)
        \le
        f_n(t)
        \exp\left\{
        -c_1(t-s)\frac{n}{\Sigma_n(t)}(1-f_n(t))
        \right\}.
\end{equation}
Then there exists \(s_1\in[-\infty,\infty]\) such
that:
\begin{enumerate}
\item if \(s<s_1\), then there exists \(c_s>0\) such that
\(f_n(s)\le e^{-c_s n}\) for all large \(n\);
\item if \(s>s_1\), then \(\lim_{n\to\infty}f_n(s)>0\).
\end{enumerate}
\end{lemma}

Finally, we apply the above lemma to the unit-truncated one-arm probabilities.

\begin{proposition}\label{ut:prop:analytic-application}
There exists \(\alpha_1\in[0,\infty]\) such that:
\begin{enumerate}
\item if \(\alpha<\alpha_1\), then there exist \(c_\alpha,C_\alpha>0\) such that for every \(r\ge1\)
\[
\bar\theta_r(\alpha)=
\Pp_\alpha(S_1\xleftrightarrow{\LL_{\R^3}^{\alpha,\le1}}S_r)
\le C_\alpha e^{-c_\alpha r} ;
\]
\item if \(\alpha>\alpha_1\), then
\(\lim_{r\to\infty}\bar\theta_r(\alpha)>0\).
\end{enumerate}
\end{proposition}

\begin{proof}
For integers \(n\ge1\), using the convention
\eqref{ut:eq:theta-sigma-full}, set
\(f_n(s):=\theta_n(e^s)\) and
\(\Sigma_n(s):=\sum_{k=0}^{n-1}f_k(s)\), with \(f_0\equiv1\).  Since
\(0\le f_n(t)\le1\le\Sigma_n(t)\), we have
\[
        \sum_{k=0}^{n}\theta_k(e^t)
        =\Sigma_{n+1}(t)
        =\Sigma_n(t)+f_n(t)
        \le 2\Sigma_n(t).
\]
Applying Lemma~\ref{ut:lem:cutoff-limit} with \(\alpha=e^s\),
\(\beta=e^t\), and \(r=n\), and using this bound for
\(\Sigma_r(\beta)\), yields \eqref{ut:eq:analytic-integrated} for every
\(n\ge8\) and \(s<t\), with \(c_1=\frac c2\).

The functions \(f_n\) are increasing in \(s\), nonincreasing in \(n\), and
take values in \([0,1]\); in particular, they converge as \(n\to\infty\).
Lemma~\ref{ut:lem:analytic} therefore yields \(s_1\in[-\infty,\infty]\)
such that \(f_n(s)\) decays exponentially for \(s<s_1\), whereas
\(\lim_{n\to\infty}f_n(s)>0\) for \(s>s_1\).  Define
\(\alpha_1:=e^{s_1}\).
For \(0<\alpha<\alpha_1\), the first conclusion of
Lemma~\ref{ut:lem:analytic}, applied at \(s=\log\alpha\), gives
\(c_\alpha>0\) such that
\(\bar\theta_n(\alpha)=f_n(\log\alpha)\le e^{-c_\alpha n}\) for all
sufficiently large integers \(n\).  By monotonicity in the radius,
this estimate extends to every real
\(r\ge1\), proving the first assertion.

Finally, if \(\alpha>\alpha_1\), then \(\log\alpha>s_1\), so the second
conclusion of Lemma~\ref{ut:lem:analytic} gives
\(\lim_{n\to\infty}f_n(\log\alpha)>0\).  Since
\(\bar\theta_r(\alpha)\) is nonincreasing in \(r\) and
\(\theta_n=\bar\theta_n\) for \(n>4\),
\[
        \lim_{r\to\infty}\bar\theta_r(\alpha)
        =\lim_{n\to\infty}f_n(\log\alpha)
        >0.
\]
This completes the proof of Proposition~\ref{ut:prop:analytic-application}.
\end{proof}

\subsection{Identification of the sharpness threshold}
\label{ut:sec:threshold-identification}

\begin{proof}[Proof of Theorem~\ref{thm:unit-truncated-sharpness}]
Proposition~\ref{ut:prop:analytic-application} gives exponential decay below
\(\alpha_1\) and a positive limiting probability above it.  It remains to
upgrade the latter to one and prove \(\alpha_1=\alpha_\ast\).
For \(a,r>0\), let \(\mathcal A_r^{(a)}\) be the event that \(S_r\) is connected
to infinity by the loop soup restricted to loops of diameter at most \(a\), namely
\[
        \mathcal A_r^{(a)}
        :=
        \{\forall R>r,\,
        S_r\xleftrightarrow{\LL_{\R^3}^{\alpha,\le a}}S_R\}.
\]
For the unit-truncated loop soup, write the limiting one-arm probability as
\begin{equation}\label{ut:eq:theta-unit-limit}
        \theta^{\le1}(\alpha)
        :=
        \lim_{R\to\infty}
        \Pp_\alpha(S_1\xleftrightarrow{\LL_{\R^3}^{\alpha,\le1}}S_R)
        =
        \Pp_\alpha(\mathcal A_1^{(1)}).
\end{equation}
For \(\alpha>\alpha_1\), Proposition~\ref{ut:prop:analytic-application} gives
\(\theta^{\le1}(\alpha)>0\).  Hence it suffices to prove that
\(\theta^{\le1}(\alpha)\in\{0,1\}\).
For every \(r\ge1\), \(\mathcal A_1^{(1)}\subseteq\mathcal A_r^{(1)}\),
whereas Brownian scaling and monotonicity give
\(\Pp_\alpha(\mathcal A_r^{(1)})
=\Pp_\alpha(\mathcal A_1^{(1/r)})
\le\Pp_\alpha(\mathcal A_1^{(1)})\).  It follows that, for every \(r\ge1\),
\begin{equation}\label{ut:eq:all-radii-one-arm}
        \Pp_\alpha(\mathcal A_r^{(1)})
        =\theta^{\le1}(\alpha).
\end{equation}
Let \(\widetilde{\mathcal A}
        :=\bigcup_{m=1}^{\infty}\mathcal A_m^{(1)}\),
which is translation invariant. The ergodicity result in~\cite[Lemma~6.1]{JL26}
then gives \(\Pp_\alpha(\widetilde{\mathcal A})\in\{0,1\}\).  Since
\(\mathcal A_m^{(1)}\) increases with \(m\), by
\eqref{ut:eq:all-radii-one-arm}, we have
\[
        \Pp_\alpha(\widetilde{\mathcal A})
        =\lim_{m\to\infty}\Pp_\alpha(\mathcal A_m^{(1)})
        =\theta^{\le1}(\alpha)\in\{0,1\}.
\]
Therefore, for every \(\alpha>\alpha_1\),
\begin{equation}\label{ut:eq:theta-unit-one}
        \theta^{\le1}(\alpha)=1.
\end{equation}

It remains to show  \(\alpha_1=\alpha_\ast\).  Suppose that
\(\alpha<\alpha_1\), and cover \(S_r\) by a deterministic \(\frac12\)-net
\(\mathcal N_r\) of size at most \(Cr^2\).  Every crossing from \(S_r\) to
\(S_{2r}\) intersects \(B(x,1)\) for some \(x\in\mathcal N_r\).  Since
\(\operatorname{dist}(x,S_{2r})\ge r\), the same crossing intersects both
\(\partial B(x,1)\) and \(\partial B(x,r)\).  Translation invariance, the
exponential estimate in Proposition~\ref{ut:prop:analytic-application}, and a
union bound yield
\[
        \Pp_\alpha(
        S_r\xleftrightarrow{\LL_{\R^3}^{\alpha,\le1}}S_{2r})
        \le C\,r^2\,e^{-cr}\to0 \quad \text{ as } r\to\infty.
\]
Thus \(\alpha_\ast\ge\alpha_1\).

If \(\alpha>\alpha_1\), then \eqref{ut:eq:theta-unit-limit} and
\eqref{ut:eq:theta-unit-one} give
\(\lim_{R\to\infty}
\Pp_\alpha(S_1\xleftrightarrow{\LL_{\R^3}^{\alpha,\le1}}S_R)=1\).
It follows that
\[
        \lim_{r\to\infty}
        \Pp_\alpha(
        S_r\xleftrightarrow{\LL_{\R^3}^{\alpha,\le1}}S_{2r})=1.
\]
Thus \(\alpha_\ast\le\alpha_1\).  Combined, this gives
\(\alpha_\ast=\alpha_1\).  The two conclusions of
Proposition~\ref{ut:prop:analytic-application} now prove
Theorem~\ref{thm:unit-truncated-sharpness}.
\end{proof}

\section{Random walk and metric-graph estimates}
\label{sec:rwls-companions}

This section proves the lattice and metric-graph results stated in
Section~\ref{sec:rwls-results}.
Section~\ref{rwls:sec:inputs} recalls the definition of loop soups on discrete graphs and metric graphs.
Section~\ref{rwls:sec:metric-small-loops} proves the  equality of critical values
in Theorem~\ref{thm:metric-critical-values} for \(d=3,4,5\).
In Section~\ref{rwls:sec:robust-reduction}, we obtain
the capacity lower bound for the crossing cluster in Theorem~\ref{rwls:cor:robust-crossing}
in dimensions \(d=3,4,5\),
which implies Theorem~\ref{thm:rwls-robust-summary}.
We also prove the uniform
small-intensity estimates of Theorem~\ref{thm:rwls-strict-summary}
in dimension three. Finally, we prove the strict comparison in Theorem~\ref{thm:rwls-strict-summary}
in Section~\ref{rwls:sec:strict-comparison} by using
the capacity estimate and Russo's formula.

\subsection{Loop soups and potential theory}
\label{rwls:sec:inputs}

We now give a more detailed account of the lattice and metric-graph loop soups
introduced in Section~\ref{sec:rwls-results}.  We fix
\(d\ge3\) and
\(\mathbb X\in\{\mathbb Z^d,\widetilde{\mathbb Z}^{d}\}\) below.

We first recall the loop measure on \(\mathbb Z^d\).  A rooted lattice loop is a
nearest-neighbor path
\(\dot\ell=(x_0,\ldots,x_n)\), with \(n\ge2\) and \(x_n=x_0\).
The rooted random walk loop measure assigns it mass
\(n^{-1}(2d)^{-n}\), and
\(\mu_{\mathbb Z^d}^{\Loop}\) is the pushforward of this measure after quotienting
by circular time-shift~\cite{LTF07}.

For the metric graph, let
\(\mathrm m\) be length measure, let \(\widetilde q_t\) be the transition
density of Brownian motion on \(\widetilde{\mathbb Z}^{d}\), and let
\(\widetilde{\mathbb P}_t^{x,x}\) be its bridge law from \(x\) to \(x\) in time
\(t\).   Following~\cite[Section~2]{Lup16}, we take loops modulo
time shift and define the loop measure
\[
 \mu_{\widetilde{\mathbb Z}^{d}}^{\Loop}(\dint\ell)
 =
 \int_{\widetilde{\mathbb Z}^{d}}\mathrm m(\dint x)
 \int_0^\infty \frac1t \widetilde q_t(x,x)
 \widetilde{\mathbb P}_t^{x,x}(\dint\ell)\,\dint t.
\]
Let \(\mathcal L_{\mathbb X}^\alpha\) be the loop soup on \(\mathbb X\),
which is the Poisson point
process on unrooted loops with intensity
\(\alpha\mu_{\mathbb X}^{\Loop}\).
Projecting the metric-graph loop soup onto \(\mathbb Z^d\) produces the random
walk loop soup at the same intensity, after loops with trivial projection are
discarded.
For \(D\subseteq\mathbb X\), let
\(\mu_D^{\Loop}\) and \(\mathcal L_D^\alpha\) be the loop measure
\(\mu_{\mathbb X}^{\Loop}\) and the loop soup \(\mathcal L_{\mathbb X}^\alpha\)
restricted to loops contained in
\(D\), respectively.

For \(D\subseteq\mathbb X\), let \(G_D\) be the Green function of the
underlying process killed on leaving \(D\).  For a nonempty set \(K\subset D\),
define
\begin{equation}\label{eq:capX}
 \operatorname{cap}_D(K)
 :=\left[
 \inf_{\nu}\iint G_D(x,y)\,\nu(\dint x)\nu(\dint y)
 \right]^{-1},
\end{equation}
where the infimum is over probability measures supported on \(K\).
We write \(\operatorname{cap}_{\mathbb X}\) when \(D=\mathbb X\).

\subsection{Equality of critical values}
\label{rwls:sec:metric-small-loops}

Fix \(d\in\{3,4,5\}\).  We prove
Theorem~\ref{thm:metric-critical-values} by first obtaining uniform
subcritical decay from OSSS, and then
using the supercritical result in~\cite{CDL24} to deduce
\(\alpha_\ast^{\widetilde{\mathbb Z}^{d}}\le\frac12\).

\begin{lemma}
\label{rwls:lem:uniform-cable-thinning}
For every \(0<\alpha<\frac12\) and \(L_0\in(0,1)\), there are \(c,C>0\) such
that, for every \(L\ge L_0\),
\(x\in\widetilde{\mathbb Z}^{d}\), and integer \(n\ge1\),
\[
\Pp_\alpha\bigl(
 \mathsf S(x,L)
 \xleftrightarrow{
 \mathcal L_{\widetilde{\mathbb Z}^{d}}^{\alpha,\le L}}
 \mathsf S(x,nL)\bigr)
 \le Ce^{-cn}.
\]
\end{lemma}

\begin{proof}
We first adapt the OSSS argument of
Section~\ref{sec:unit-truncated-sharpness} to obtain sharpness,
and then use
\cite[Theorem~1.2]{CD24} to locate the resulting threshold.
Take \(L\ge1\) and $x=0$.

 We begin with the necessary preparations as in Section~\ref{ut:sec:osss},
but omit repetitive details.
Fix \(\lambda>0\), \(L\ge1\), and an integer
\(N\ge1\), and introduce the lower-cutoff soup
\[
 \mathcal L_{\widetilde{\mathbb Z}^{d}}^{\lambda,N,L}
 :=\bigl\{\ell\in
 \mathcal L_{\widetilde{\mathbb Z}^{d}}^{\lambda}:
 2^{-N}L<\diam(\ell)\le L\bigr\}.
\]
For
\(0\le k<N\) and \(z\in\mathbb Z^d\), we set
\[
 \begin{aligned}
 Q_{k,z}^L
 &:=\widetilde{\mathbb Z}^{d}
 \cap2^{-k}L(z+[0,1)^d),\\
 \Gamma_{k,z}^L
 &:=\bigl\{\ell:2^{-(k+1)}L<\diam(\ell)\le2^{-k}L,
 \ \ell(0)\in Q_{k,z}^L\bigr\},\\
 Q_{(k,z)}^+
 &:=\bigl\{y\in\widetilde{\mathbb Z}^{d}:
 d(y,Q_{k,z}^L)
 \le2^{-k}L\bigr\}.
 \end{aligned}
\]
Then, we define the finite index set as before
\[
 \mathscr I_{N,L,n}
 :=\bigl\{i=(k,z):0\le k<N,\ Q_{k,z}^L\ne\varnothing,
 \ Q_i^+\cap\mathsf B_{(n+1)L}\ne\varnothing\bigr\}.
\]
For \(i=(k,z)\in\mathscr I_{N,L,n}\), define the Poisson coordinate
\(\omega_i:=\mathcal L_{
\widetilde{\mathbb Z}^{d}}^{\lambda,N,L}
\cap\Gamma_{k,z}^L\).  Since the classes \(\Gamma_{k,z}^L\) are disjoint, the
\(\omega_i\) are independent Poisson point processes.

Choose an integer \(C_0=C_0(d)\) large enough that
each \(Q_i^+\) lies in \(\mathsf B(y_i,C_0L/2)\) for some lattice
vertex \(y_i\).
We construct a randomized decision algorithm, as in
Lemma~\ref{ut:lem:revealment}, to determine the event
\(\{
 \mathsf S_{C_0L}
 \xleftrightarrow{
 \mathcal L_{\widetilde{\mathbb Z}^{d}}^{\lambda,N,L}}
 \mathsf S_{nL}\}\). Choose \(J\) uniformly
from \(\{2C_0,\ldots,n-2C_0\}\), independently of the loop soup.  First reveal
\(\omega_i\) for every \(i\in\mathscr I_{N,L,n}\) such that
\(Q_i^+\cap\mathsf S_{JL}\ne\varnothing\), and mark the revealed loops connected
to $\mathsf S_{JL}$.  Then, in a fixed order, repeatedly reveal every unrevealed
\(\omega_i\) such that \(Q_i^+\) meets the current marked union, adding all
newly connected revealed loops to that union.
The procedure
terminates after finitely many steps, since $\mathscr I_{N,L,n}$ is finite.
It outputs one exactly when a connected component of the final
marked union meets both \(\mathsf S_{C_0L}\) and
\(\mathsf S_{nL}\).

Let $\Theta_n(\lambda)=1$ when $0\le n\le4C_0$, and set
\begin{equation}\label{eq:theta}
\Theta_n(\lambda):=\sup_{L\ge1}
 \Pp_\lambda\bigl(
 \mathsf S_{C_0L}
 \xleftrightarrow{
 \mathcal L_{\widetilde{\mathbb Z}^{d}}^{\lambda,\le L}}
 \mathsf S_{nL}\bigr) \qquad \text{ when } n>4C_0.
\end{equation}
Put
\(r_i(j):=\left\lfloor
 (L^{-1}|\|y_i\|_\infty-jL|-C_0)_+
 \right\rfloor\).
As in Lemma~\ref{ut:lem:revealment}, integer translation to \(y_i\) gives
\(\Pp_\lambda(i\text{ is revealed}\mid J=j)
 \le\Theta_{r_i(j)}(\lambda)\).
For small indices, this uses the convention \(\Theta_k=1\) for
\(k\le4C_0\).
 Writing \(\Sigma_n(\lambda):=\sum_{j=0}^{n-1}\Theta_j(\lambda)\) and averaging over \(J\), we obtain
\begin{equation}
\label{rwls:eq:uniform-envelope-revealment}
 \delta_i(\lambda)
 :=\Pp_\lambda(i\text{ is revealed})
 \le\frac{C}{n-4C_0+1}
 \sum_{r=0}^{n+1}\Theta_r(\lambda)
 \le C\frac{\Sigma_n(\lambda)}n.
\end{equation}
Combining \eqref{rwls:eq:uniform-envelope-revealment} with
Russo's formula in Lemma~\ref{ut:lem:russo} and the OSSS inequality
\cite[Theorem~1.9]{DRT18}, and arguing as in Lemma~\ref{ut:lem:cutoff-limit},
we obtain constants \(c>0\) and \(n_0>0\) such that, for all \(n>n_0\),
\begin{equation}
\label{rwls:eq:uniform-envelope-integrated}
 \Theta_n(\lambda)
 \le \Theta_n(\gamma)
 \exp\left\{
 -c\log(\frac\gamma\lambda)
 \frac n{\Sigma_n(\gamma)}
 \bigl(1-\Theta_n(\gamma)\bigr)
 \right\}, \quad 0<\lambda<\gamma.
\end{equation}
By \eqref{rwls:eq:uniform-envelope-integrated},
Lemma~\ref{ut:lem:analytic}, applied to
\(f_n(s):=\Theta_n(e^s)\), yields a threshold
\(\widehat\alpha\in[0,\infty]\): below it \(\Theta_n\) decays exponentially,
while above it \(\Theta_n\) has a positive limit.

By \cite[Theorem~1.2]{CD24}, we have $\Theta_n(\tfrac12)\to0$ as $n\to\infty$.
This implies that \(\widehat\alpha\ge\frac12\).
Hence, for every
\(\alpha<\frac12\), we have $\Theta_n(\alpha)\le Ce^{-cn}$.
For arbitrary \(x\) and \(L\ge L_0\), we can compare with a nearest lattice
vertex and cutoff \(L'=\max\{1,L\}\). For all sufficiently large \(n\),
the probability in the statement is at most
\(\Theta_{\lfloor nL/L'\rfloor-2}(\alpha)\).
Since \(L/L'\ge L_0\), the result follows, with bounded \(n\)
absorbed in \(C\).
\end{proof}

The remaining lemma is specific to the metric graph and covers the cutoff
range that is not seen by the preceding estimate.

\begin{lemma}
\label{rwls:lem:cable-edge-barrier}
For every \(0<a<b<\infty\), there are \(L_0\in(0,1)\) and \(c,C>0\) such
that, for every \(0<\alpha\le\frac12\), \(0<L\le L_0\), \(r\ge2\), and
\(x\in\widetilde{\mathbb Z}^{d}\),
\[
 \Pp_\alpha\bigl(
 \mathsf S(x,ar)
 \xleftrightarrow{
 \mathcal L_{\widetilde{\mathbb Z}^{d}}^{\alpha,\le L}}
 \mathsf S(x,br)\bigr)
 \le C r^{d-1}e^{-cr/L}.
\]
\end{lemma}

\begin{proof}
Let \(J=(-2,2)\), \(u=-\frac12\), and \(v=\frac12\).  Since
\(G_J(s,t)=\frac12(s+2)(2-t)\) for \(-2<s\le t<2\),
\[
 \frac{G_J(u,v)}{\sqrt{G_J(u,u)G_J(v,v)}}
 =\frac{(u+2)(2-v)}
 {\sqrt{(u+2)(2-u)(v+2)(2-v)}}
 =\frac35.
\]
By~\cite[Proposition~5.2]{Lup16},
\begin{equation}
\label{rwls:eq:cable-gate-probability}
 q:=\Pp_{1/2}\bigl(u\xleftrightarrow{\mathcal L_J^{1/2}}v\bigr)
 =\frac2\pi\arcsin\left(\frac35\right)<1.
\end{equation}

For every unit edge \(e\) of \(\widetilde{\mathbb Z}^{d}\), choose a family
\((J_{e,k},u_{e,k},v_{e,k})_k\) of disjoint copies of \((LJ,Lu,Lv)\), with
pairwise distances greater than \(L\), and with each \(J_{e,k}\) at distance
greater than \(L\) from both endpoints of \(e\).  For every nondegenerate
closed subsegment \(I\subset e\), write \(|I|\) for its length and put
\(N_L(I):=\#\{k:J_{e,k}\subset I\}\).  The family may be chosen so that,
simultaneously for all such \(I\) and \(k\),
\begin{equation}
\label{rwls:eq:cable-gate-packing}
 N_L(I)\ge \left(c_0\frac{|I|}{L}-2\right)_+,
 \qquad
 d\bigl([u_{e,k},v_{e,k}],J_{e,k}^{\mathrm c}\bigr)=\frac{3L}{2}.
\end{equation}

Let \(G_{e,k}\) be the event that \(u_{e,k}\) and \(v_{e,k}\) are connected
by loops from \(\mathcal L_{J_{e,k}}^{\alpha,\le L}\).  By restriction,
Brownian scaling, monotonicity, and
\eqref{rwls:eq:cable-gate-probability},
\begin{equation}
\label{rwls:eq:cable-gate-open}
 \sup_{0<\alpha\le1/2}\Pp_\alpha(G_{e,k})\le q.
\end{equation}
The family \((G_{e,k})_{e,k}\) is independent.  Call a closed subsegment
\(I\subset e\) traversed if it lies in the range of a finite chain of loops
from \(\mathcal L_{\widetilde{\mathbb Z}^{d}}^{\alpha,\le L}\).  By
\eqref{rwls:eq:cable-gate-packing}, traversing \(I\) forces \(G_{e,k}\) for
every \(J_{e,k}\subset I\).  Hence
\begin{equation}
\label{rwls:eq:cable-segment-traversal}
 \sup_{0<\alpha\le1/2}\Pp_\alpha(I\text{ is traversed})
 \le q^{N_L(I)}\le Ce^{-c|I|/L}.
\end{equation}

Set \(\Delta=(b-a)r\).  Every annular crossing contains a simple path
\(\gamma\) in \(\widetilde{\mathbb Z}^{d}\), whose initial edge has at most
\(C_dr^{d-1}\) possible
choices.  Let \(E_{\le1}\) and \(E_{\ge2}\) be the events that such a path
exists and visits at most one or at least two lattice vertices, respectively.
On \(E_{\le1}\), the path traverses a closed edge segment \(I\) with
\(|I|\ge\frac\Delta2\).  Thus \eqref{rwls:eq:cable-segment-traversal} gives
\begin{equation}
\label{rwls:eq:cable-one-vertex}
 \Pp_\alpha(E_{\le1})\le C_dr^{d-1}e^{-c\Delta/L}.
\end{equation}

On \(E_{\ge2}\), each path \(\gamma\) contains a simple lattice path
\(\widehat\gamma\) of length \(m\ge\frac\Delta3\).  From a fixed initial edge
there are at most \(C_d(2d-1)^m\) choices for \(\widehat\gamma\).  Moreover, each
such path implies \(\bigcap_{e\in\widehat\gamma}\bigcap_kG_{e,k}\), whose
probability is at most \(e^{-c_1m/L}\) by
\eqref{rwls:eq:cable-gate-open} and independence.  Choose
\(L_0\in(0,1)\) so that
\((2d-1)e^{-c_1/L_0}<\frac12\).  Then
\begin{equation}
\label{rwls:eq:cable-many-vertices}
 \Pp_\alpha(E_{\ge2})
 \le C_dr^{d-1}
 \sum_{m\ge\Delta/3}((2d-1)e^{-c_1/L})^m
 \le C_dr^{d-1}e^{-c\Delta/L}.
\end{equation}
Combining \eqref{rwls:eq:cable-one-vertex} and
\eqref{rwls:eq:cable-many-vertices} proves the lemma.
\end{proof}

\begin{proof}[Proof of Theorem~\ref{thm:metric-critical-values}]
Fix \(0<\alpha<\frac12\) and \(0<a<b<\infty\), and let \(L_0\) be given by
Lemma~\ref{rwls:lem:cable-edge-barrier}.  Constants below may depend on
\(d,\alpha,a,b\).  Put \(c_{a,b}=\frac{b-a}{4}\).  It suffices to consider
\(0<\delta\le\min\{1,\frac{c_{a,b}}2\}\), since the remaining values are absorbed
by increasing \(C\).  Put \(L=\delta r\).  If \(L\ge L_0\), cover
\(\mathsf S_{ar}\) by at most
\(C_d\delta^{-(d-1)}\) balls of radius \(L\).  Every annular crossing then yields,
from one of these balls, a connection to distance at least \(c_{a,b}r\).
Taking \(n=\lfloor \frac{c_{a,b}}\delta\rfloor\),
Lemma~\ref{rwls:lem:uniform-cable-thinning} and a union bound give
\[
 \sup_{r:\,\delta r\ge L_0}
 \Pp_\alpha\bigl(
 \mathsf S_{ar}
 \xleftrightarrow{
 \mathcal L_{\widetilde{\mathbb Z}^{d}}^{\alpha,\le\delta r}}
 \mathsf S_{br}\bigr)
 \le C\delta^{-(d-1)}e^{-c/\delta}.
\]
If \(L<L_0\), then \(r<\frac{L_0}\delta\), and
Lemma~\ref{rwls:lem:cable-edge-barrier} gives the same bound.
\begingroup
After adjusting constants for \(\delta\) bounded away from zero, we obtain
\begin{equation}
\label{rwls:eq:metric-small-loop-bound}
 \chi_{a,b}^{\widetilde{\mathbb Z}^{d},\alpha}(\delta)
 \le C\delta^{-(d-1)}e^{-c/\delta},
 \qquad 0<\delta\le1.
\end{equation}
\endgroup
Since \(\alpha<\frac12\) and \(a<b\) were arbitrary,
we have
\(\alpha_\ast^{\widetilde{\mathbb Z}^{d}}\ge\frac12\).

\begingroup
For the reverse inequality, fix \(\alpha>\frac12\), \(0<a<b<\infty\),
and \(\delta>0\).
Let \(E_r\) be the event that \(\mathsf S_{ar}\) is connected to
\(\mathsf S_{br}\) in \(\mathcal L_{\widetilde{\mathbb Z}^{d}}^{\alpha}\), and let
\(E_r^\delta\) be the corresponding event in
\(\mathcal L_{\widetilde{\mathbb Z}^{d}}^{\alpha,\le\delta r}\).  Set
\[
 V_r:=\left\{\nexists\,\ell\in\mathcal L_{\widetilde{\mathbb Z}^{d}}^{\alpha}:
 \operatorname{diam}(\ell)>\delta r,\,
 \ell\cap\mathsf B_{br}\ne\varnothing\right\}.
\]
By~\cite[Lemma~2.7(a)]{CS16} (the estimate there applies to the metric graph as well),
there is a constant \(c\in(0,1)\) such that
\(\Pp_\alpha(V_r)\ge c\), uniformly in
\(r\ge2\).  Note that \(V_r\cap E_r\subseteq E_r^\delta\),
and the events \(V_r\) and \(E_r^\delta\) are independent.  Consequently,
\begin{equation}
\label{rwls:eq:supercritical-void-comparison}
 1-\Pp_\alpha(E_r^\delta)
 \le c^{-1}
 \bigl(1-\Pp_\alpha(E_r)\bigr).
\end{equation}
By~\cite{CDL24}, the full loop soup percolates at intensity \(\alpha\),
and hence \(\lim_{r\to\infty}\Pp_\alpha(E_r)=1\).
Keeping \(\delta\) fixed in
\eqref{rwls:eq:supercritical-void-comparison} gives
\(\Pp_\alpha(E_r^\delta)\to1\), and therefore
\(\chi_{a,b}^{\widetilde{\mathbb Z}^{d},\alpha}(\delta)=1\).
Since this holds for every \(\alpha>\frac12\), we obtain
\(\alpha_\ast^{\widetilde{\mathbb Z}^{d}}\le\frac12\).  Together with the lower bound,
this proves the theorem for \(d=3,4,5\).
\endgroup
\end{proof}

\subsection{\texorpdfstring{Capacity lower bounds}{Capacity lower bounds}}
\label{rwls:sec:robust-reduction}

Throughout this subsection, fix \(d\in\{3,4,5\}\) and
\(\mathbb X\in\{\mathbb Z^d,\widetilde{\mathbb Z}^{d}\}\).  Recall from
Theorem~\ref{thm:metric-critical-values} that
\(\alpha_\ast^{\widetilde{\mathbb Z}^{d}}=\frac12\), while
the corresponding lattice thresholds remain unknown.

\begingroup
Recall \(a_d(r)\) from \eqref{eq:rwls-capacity-scale}.  For
\(\eta,\kappa>0\) and \(r\ge2\), we define the event
\begin{equation}\label{rwls:eq:rooted-robust-event-def}
\mathsf H_r^{\mathbb X,\eta,\kappa}
:=
\left\{
\begin{gathered}
\text{the cluster }\mathcal C(0)\text{ in }\mathcal L_{\mathbb X}^{\alpha}
\text{ crosses }\mathsf A(1,r)\text{ and contains }J \\
\text{ such that }J\subset\mathsf A(0,\frac r2,2r),\,
\diam(J)\ge\eta r,\,
\operatorname{cap}_{\mathbb X}(J)\ge\kappa a_d(r)
\end{gathered}
\right\}.
\end{equation}
\endgroup

\begin{theorem}[Capacity lower bound]
\label{rwls:cor:robust-crossing}
\begingroup
There are \(\eta_0,\kappa_0\in(0,1)\) such that, for every compact interval
\(I=[\alpha_-,\alpha_+]\subset\mathrm{I}_0^{\mathbb X}\), there is
\(c>0\) such that for all
\(\alpha\in I\) and \(r\ge2\),
\begin{equation}\label{rwls:eq:robust-crossing}
 \Pp_\alpha\bigl(\mathsf H_r^{\mathbb X,\eta_0,\kappa_0}\bigr)
 \ge c p_r^{\mathbb X,\alpha}.
\end{equation}
\endgroup
\end{theorem}

\begin{proof}[Proof of Theorem~\ref{thm:rwls-robust-summary}]
Applying Theorem~\ref{rwls:cor:robust-crossing} for \(\alpha\in\mathrm{I}_0^{\mathbb X}\), we get
\(\Pp_\alpha(\mathsf H_r^{\mathbb X,\eta_0,\kappa_0})
\ge c p_r^{\mathbb X,\alpha}\).  The monotonicity of capacity gives
\(\operatorname{cap}_{\mathbb X}(\mathcal C(0)\cap\mathsf B_{2r})
\ge\operatorname{cap}_{\mathbb X}(J)\ge\kappa_0a_d(r)\), and the claim
\eqref{eq:rwls-robust-summary} follows.
\end{proof}

We first give the counterpart of Lemma~\ref{lem:one-sided-surgery}.
For \(\kappa>0\), \(r\ge32\), and \(K\subset\mathsf B_r\), define
\begin{equation}
\label{rwls:eq:terminal-loop-class}
 \mathscr T_r^{\mathbb X,\kappa}(K)
 :=
 \left\{
 \begin{gathered}
 \ell\subset\mathsf B_{2r}:\ \ell\cap K\ne\varnothing\text{ and there is }
 J\subset\ell\cap\mathsf A(0,r,2r)\\
 \text{joining }
 \mathsf S_{5r/4}\text{ to }\mathsf S_{7r/4}
 \text{ such that }\operatorname{cap}_{\mathbb X}(J)\ge\kappa a_d(r)
 \end{gathered}
 \right\}.
\end{equation}

\begin{lemma}
\label{rwls:lem:terminal-loop-comparison}
There are constants \(\kappa_0,c,C>0\) such that, for every \(r\ge32\) and
\(K\) as above,
\begin{equation}
\label{rwls:eq:terminal-loop-comparison}
 c r^{2-d}\operatorname{cap}_{\mathbb X}(K)
 \le
 m_r^{\mathbb X}(K)
 :=\mu_{\mathsf B_{2r}}^{\Loop}
 \bigl(\mathscr T_r^{\mathbb X,\kappa_0}(K)\bigr)
 \le
 C r^{2-d}\operatorname{cap}_{\mathbb X}(K).
\end{equation}
Moreover, for every \(\delta\in(0,\frac12)\), there is
\(C_\delta<\infty\) such that, uniformly for \(r\ge32\) and every \(K\) as
above,
\begin{equation}
\label{rwls:eq:large-loop-terminal-comparison}
 \mu_{\mathbb X}^{\Loop}\bigl(
 \ell\cap K\ne\varnothing,\ \diam(\ell)>\delta r
 \bigr)
 \le C_\delta r^{2-d}\operatorname{cap}_{\mathbb X}(K)
 \le C_\delta m_r^{\mathbb X}(K).
\end{equation}
\end{lemma}

\begin{proof}
For the lattice case, the estimates  can be obtained by adapting the argument of
\cite[Lemma~2.7]{CS16}. On the metric graph, they also follow directly from
Lupu's construction~\cite[Section~2]{Lup16}. We omit the details.
\end{proof}

The next two proofs follow the same strategy as the preceding  arguments for
the Brownian loop soup in \(\mathbb R^3\), so we keep the exposition brief.

\begin{proof}[Proof of Theorem~\ref{rwls:cor:robust-crossing}]
We can follow exactly the same route as in the proof of Theorem~\ref{thm:main}.
Fix a compact interval
\(I=[\alpha_-,\alpha_+]\subset\mathrm{I}_0^{\mathbb X}\).
For \(r\ge32\), let \(\mathcal K_r(0)\) be the union of the origin and all
clusters in \(\mathcal L_{\mathsf B_r}^{\alpha}\) that intersect it, and set
\begin{equation}\label{rwls:eq:rooted-terminal-probability}
 g_r^{\mathbb X,\alpha}
 :=\Pp_\alpha\bigl(
 \mathcal L_{\mathbb X}^{\alpha}\cap
 \mathscr T_r^{\mathbb X,\kappa_0}(\mathcal K_r(0))\ne\varnothing
 \bigr).
\end{equation}
The proof of Lemma~\ref{lem:robust-prop}, with
Lemma~\ref{rwls:lem:terminal-loop-comparison} in place of
Lemma~\ref{lem:one-sided-surgery}, gives an \(a\in(0,1)\) such that,
uniformly in \(\alpha\in I\) and \(r\ge32\),
\begin{equation}\label{rwls:eq:rooted-terminal-propagation}
 g_{2r}^{\mathbb X,\alpha}\ge ag_r^{\mathbb X,\alpha}
 \quad \text{and} \quad
 \inf_{32\le\rho<64}g_\rho^{\mathbb X,\alpha}>0.
\end{equation}
Similarly, the proof of Proposition~\ref{lem:residual-extension} gives, for
every \(\delta\in(0,\frac18)\), a constant \(C_\delta\ge1\) such that,
uniformly in \(\alpha\in I\) and \(r\ge32\),
\begin{equation}\label{rwls:eq:rooted-residual-recursion}
 p_r^{\mathbb X,\alpha}
 \le C_\delta g_r^{\mathbb X,\alpha}
 +\chi_{5/8,3/4}^{\mathbb X,\alpha_+}(\delta)
 p_{r/2}^{\mathbb X,\alpha}.
\end{equation}
Since \(\alpha_+\in\mathrm{I}_0^{\mathbb X}\), we can choose
\(\delta\in(0,\frac18)\) so that
\(\chi_{5/8,3/4}^{\mathbb X,\alpha_+}(\delta)<a/2\).
The remainder of the proof is the same as the dyadic induction in the proof of
Theorem~\ref{thm:main}.
\end{proof}

\begin{proof}[Proof of \eqref{rwls:eq:small-intensity-bounds} in Theorem~\ref{thm:rwls-strict-summary}]
Fix \(\bar\alpha\in(0,\frac12)\). By
Theorem~\ref{thm:metric-critical-values} and domination,
 \(\bar\alpha\in\mathrm I_0^{\mathbb X}\) for both choices of
\(\mathbb X\in\{\mathbb Z^3,\widetilde{\mathbb Z}^{3}\}\).
All constants below are uniform in the model and in
\(0<\alpha\le\bar\alpha\). Retain \(\mathcal K_r(0)\) and
\(g_r^{\mathbb X,\alpha}\) from \eqref{rwls:eq:rooted-terminal-probability},
and write
\[
 A(r):=\Ee_\alpha[\operatorname{cap}_{\mathbb X}(\mathcal K_r(0))],
 \qquad p_r:=p_r^{\mathbb X,\alpha},
 \qquad g_r:=g_r^{\mathbb X,\alpha}.
\]
Repeating the proof of Lemma~\ref{lem:robust-prop}, with
Lemma~\ref{rwls:lem:terminal-loop-comparison} in place of
Lemma~\ref{lem:one-sided-surgery} and using
\(\operatorname{cap}_{\mathbb X}(\{0\})>0\), we have
\begin{equation}\label{rwls:eq:uniform-terminal-capacity}
 g_r\asymp\frac{\alpha A(r)}r,
 \qquad g_{2r}\ge a g_r,
 \qquad r\ge32,
\end{equation}
for some fixed \(a\in(0,1)\), together with
\(p_2\le C\alpha\) and \(p_\rho\le Cg_\rho\) for \(32\le\rho<64\).
Choosing the cutoff at intensity \(\bar\alpha\) and applying the dyadic
iteration from the proof of Theorem~\ref{thm:main}, we obtain
\begin{equation}\label{rwls:eq:uniform-arm-capacity}
 p_r\asymp g_r\asymp\frac{\alpha A(r)}r,
 \qquad r\ge32.
\end{equation}

Choose a fixed \(r_0\ge32\) such that
\(\kappa_0r_0/2\ge\operatorname{cap}_{\mathbb X}(\{0\})\),
and set \(r_n=2^nr_0\) and \(A_n=A(r_n)\).
Applying the tail integrations in
\eqref{imp:eq:capacity-upper-recursion}--\eqref{imp:eq:capacity-lower-recursion}
with seed \(\{0\}\) and
\(\operatorname{cap}_{\mathbb X}(\mathsf B_t)\le C(1+t)\),
and using
\(\operatorname{cap}_{\mathbb X}(\mathcal K_{2r_j}(0))\ge\kappa_0r_j\) together with
\eqref{rwls:eq:uniform-arm-capacity}, we conclude that
\begin{equation}\label{rwls:eq:small-intensity-capacity-recursion}
 c_0+c_2\alpha\sum_{j=0}^{n-1}A_j
 \le A_n\le
 C_0+C_2\alpha\sum_{j=0}^{n-1}A_j.
\end{equation}
As in \eqref{imp:eq:capacity-gronwall}, choosing
\(c_0\le A_0\le C_0\) and applying induction gives
\[
 c_0(1+c_2\alpha)^n\le A_n\le C_0(1+C_2\alpha)^n.
\]
Substituting into \eqref{rwls:eq:uniform-arm-capacity} and using
\(\log(1+t)\asymp t\) as \(t\downarrow0\) proves
\eqref{rwls:eq:small-intensity-bounds} at the scales \(r_n\) for sufficiently
small \(\alpha\). We conclude the proof by possibly adjusting the constants.
\end{proof}

\subsection{One-arm comparisons}
\label{rwls:sec:strict-comparison}
From this point onward, we return to \(d=3\) and fix
\(\mathbb X\in\{\mathbb Z^3,\widetilde{\mathbb Z}^{3}\}\).
Throughout, we fix \(I=[\alpha_-,\alpha_+]\subset(0,\alpha_\ast^{\mathbb X})\)
as in Section~\ref{rwls:sec:robust-reduction}, and allow constants to depend on \(I\).
Our goal is to prove the following
quantitative version of \eqref{eq:rwls-strict-summary}.

\begin{theorem}
\label{rwls:thm:lacunary-comparison}
There exist \(c,C>0\) such that, for every
\(\alpha<\beta\) in \(I\) and $r\ge2$,
\begin{equation}\label{rwls:eq:lacunary-comparison}
 p_r^{\mathbb X,\alpha}
 \le C
 r^{-c(\beta-\alpha)}p_r^{\mathbb X,\beta}.
\end{equation}
\end{theorem}

\begingroup
\begin{proof}[Proof of \eqref{eq:rwls-strict-summary}]
For \(0<\alpha<\beta<\alpha_\ast^{\mathbb X}\), we apply
Theorem~\ref{rwls:thm:lacunary-comparison} on \(I=[\alpha,\beta]\) and conclude the proof.
\end{proof}
\endgroup

The proof of Theorem~\ref{rwls:thm:lacunary-comparison} is based on the same strategy
as that of Proposition~\ref{prop:xi-linear-separation}.
We consider two localized one-arm events with capacity bounds
in disjoint annuli. A loop joining the corresponding good outer and inner pieces is pivotal for the one-arm event
\(\mathcal E_r^{\mathbb X}\), and
 Russo's formula accumulates its mass over a number of disjoint scales of order \(\log r\).
This will yield \eqref{rwls:eq:lacunary-comparison}.

We begin by introducing the localized one-arm events
\(\mathcal E_s^{+}\) and \(\mathcal E_{s,r}^{-}\),
and, following Section~\ref{sec:strict-monotonicity}, denote
the events with additional diameter and capacity requirements by
\(\mathsf H_{\mathsf B_{2s}}^{+}(s)\) and
\(\mathsf H_{\mathsf B_{s/2}^{\rm c}}^{-}(s,r)\).
After decreasing the constants \(\eta_0,\kappa_0\) from
Theorem~\ref{rwls:cor:robust-crossing} if necessary, we use the same constants
for these localized events.

\begin{definition}[Localized one-arm events with good pieces]
Fix \(2\le s\le \frac r2\).  Let \(\mathcal E_s^{+}\) be the event
that \(0\) is connected to \(\mathsf S_s\) in
\(\mathcal L_{\mathsf B_{2s}}^\alpha\), and let
\(\mathcal E_{s,r}^{-}\) be the event that \(\mathsf S_s\) is
connected to \(\mathsf S_r\) in
\(\mathcal L_{\mathsf B_{s/2}^{\rm c}}^\alpha\).
Let \(\mathsf H_{\mathsf B_{2s}}^{+}(s)\) be the event that
\(\mathcal E_s^{+}\) is realized by one cluster containing a compact
connected set \(J^+\subset\mathsf A(0,\frac s2,2s)\) with diameter at least
\(\eta_0s\) and capacity at least \(\kappa_0s\).  Let
\(\mathsf H_{\mathsf B_{s/2}^{\rm c}}^{-}(s,r)\) be the event that
\(\mathcal E_{s,r}^{-}\) is realized by one cluster containing a compact
connected set
\(J^-\subset\mathsf A(0,\frac s2,2s)\) satisfying the same two conditions.
We call \(J^+\) and \(J^-\) a good outer piece and a good inner piece, respectively.
\end{definition}

We record the following localized version of Theorem~\ref{rwls:cor:robust-crossing},
which can be proved in a similar fashion by restricting the loop soup to
the corresponding inner and exterior domains. We hence omit its proof for brevity.

\begin{theorem}[Localized capacity lower bounds]
\label{thm:rwls-localized-robust-summary}
There is \(c>0\) such that, for every \(\alpha\in I\) and
\(2\le s\le \frac r2\),
\[
\Pp_\alpha(\mathsf H_{\mathsf B_{2s}}^{+}(s))
\ge c\Pp_\alpha(\mathcal E_s^{+}),\qquad
\Pp_\alpha(\mathsf H_{\mathsf B_{s/2}^{\rm c}}^{-}(s,r))
\ge c\Pp_\alpha(\mathcal E_{s,r}^{-}).
\]
\end{theorem}

To compare the two localized one-arm probabilities with the full one-arm
probability, we need a lattice and metric-graph version of the splitting estimate in Lemma~\ref{lem:splitting}.
For \(K_1,K_2\subseteq D\subseteq\mathbb X\), write
\[
 \mathcal B(K_1,K_2;D)
 :=\mu_D^{\Loop}
 (\ell\cap K_1\ne\varnothing,\ell\cap K_2\ne\varnothing).
\]

\begin{lemma}
\label{rwls:lem:deterministic-factorization}
There is \(C<\infty\) such that, for every \(s\ge2\) and all sets
\(K_1,K_2\), finite in the lattice case and compact in the metric-graph case, with
\(K_1\subset \mathsf{B}_{2s}\) and
\(K_2\subset\mathsf B_{16s}^{\rm c}\),
\begin{equation}\label{rwls:eq:deterministic-factorization}
 \mathcal B(K_1,K_2;\mathbb X)
 \le C\,\mathcal B(K_1,\mathsf S_{3s};\mathsf B_{4s})
 \mathcal B
 (\mathsf S_{12s},K_2;\mathsf B_{8s}^{\rm c}).
\end{equation}
\end{lemma}

\begin{proof}
\begingroup
As in Lemma~\ref{lem:splitting}, the alternating-entrance decomposition
\cite[Equation~(3.7)]{Vog25} at \(\mathsf S_{10s},\mathsf S_{12s}\) gives
\begin{equation}\label{rwls:eq:global-loop-mass-upper-bound}
 \mathcal B(K_1,K_2;\mathbb X)
 \lesssim\frac{\operatorname{cap}_{\mathbb X}(K_1)}s\,
 \mathcal B(\mathsf S_{12s},K_2;\mathsf B_{8s}^{\rm c}).
\end{equation}
It follows from \cite[Lemma~2.7]{CS16} that
\begin{equation}\label{rwls:eq:confined-factorization-lower-bound}
 \mathcal B(K_1,\mathsf S_{4s};\mathbb X)
 \lesssim\frac{\operatorname{cap}_{\mathbb X}(K_1)}s
 \lesssim\mathcal B(K_1,\mathsf S_{3s};\mathsf B_{4s}).
\end{equation}
Combining \eqref{rwls:eq:global-loop-mass-upper-bound} and
\eqref{rwls:eq:confined-factorization-lower-bound} proves
\eqref{rwls:eq:deterministic-factorization}.
\endgroup
\end{proof}

We now use the preceding events at scales \(2s\) and \(16s\), and combine the deterministic
factorization with Poisson BK.

\begin{lemma}
\label{rwls:lem:localized-splitting}
There is \(C<\infty\) such that for all $\alpha\in I$ and $64\le s\le \frac r{64}$,
\[
 p_r^{\mathbb X,\alpha}
 \le C\Pp_\alpha(\mathcal E_{2s}^{+})
 \Pp_\alpha(\mathcal E_{16s,r}^{-}).
\]
\end{lemma}

\begin{proof}
We split \(\mathcal E_r^{\mathbb X}\) into four cases according to whether
\(\mathcal E_{2s}^{+}\) and
\(\mathcal E_{16s,r}^{-}\) occur.

\smallskip
\noindent\textbf{Case (1):}
\(\mathcal E_r^{\mathbb X}\cap\mathcal E_{2s}^{+}
\cap\mathcal E_{16s,r}^{-}\).
Since the two localized events depend on disjoint loop classes, by independence, we have
\[
 \Pp_\alpha(\mathcal E_r^{\mathbb X}\cap\mathcal E_{2s}^{+}
 \cap\mathcal E_{16s,r}^{-})
 \le\Pp_\alpha(\mathcal E_{2s}^{+})
 \Pp_\alpha(\mathcal E_{16s,r}^{-}).
\]

For Cases~(2)--(4), let \(K^+\) be the cluster of \(0\) generated by loops
contained in \(\mathsf B_{2s}\), and let \(K^-\) be \(\mathsf S_r\) together
with the clusters that intersect it and are generated by loops contained in
\(\mathsf B_{16s}^{\rm c}\).
On \(\mathcal E_r^{\mathbb X}\setminus\mathcal E_{2s}^{+}\),
there is a loop \(\ell_+\) intersecting both \(K^+\) and \(\mathsf S_{4s}\),
while on \(\mathcal E_r^{\mathbb X}\setminus\mathcal E_{16s,r}^{-}\),
there is a loop \(\ell_-\) intersecting both \(K^-\) and \(\mathsf S_{8s}\).
Conditionally on \(K^\pm\), we put
\[
\begin{aligned}
 \widehat\lambda_+
 &:=\mathcal B(K^+,\mathsf S_{4s};\mathbb X),
 &\lambda_+
 &:=\mathcal B(K^+,\mathsf S_{3s};\mathsf B_{4s}),\\
 \widehat\lambda_-
 &:=\mathcal B(\mathsf S_{8s},K^-;\mathbb X),
 &\lambda_-
 &:=\mathcal B
 (\mathsf S_{12s},K^-;\mathsf B_{8s}^{\rm c}).
\end{aligned}
\]
Equation~\eqref{rwls:eq:confined-factorization-lower-bound} gives
\(\widehat\lambda_+\lesssim\lambda_+\). Marking only the outer hit in the
same decomposition as in~\eqref{rwls:eq:global-loop-mass-upper-bound} gives
\begin{equation}\label{rwls:eq:outer-confined-factorization}
 \widehat\lambda_-
 =\mathcal B(\mathsf S_{8s},K^-;\mathbb X)
 \lesssim\mathcal B(\mathsf S_{12s},K^-;\mathsf B_{8s}^{\rm c})
 =\lambda_-.
\end{equation}
Since these masses are uniformly bounded, we have, uniformly for \(\alpha\in I\),
\begin{equation}\label{rwls:eq:defect-loop-comparison}
 1-e^{-\alpha\widehat\lambda_\pm}
 \le C(1-e^{-\alpha\lambda_\pm}).
\end{equation}
The four base loop classes are disjoint, and a local attachment realizes the
corresponding event.  Hence,
\begin{equation}\label{rwls:eq:defect-loop-expectations}
 \Ee_\alpha[1-e^{-\alpha\lambda_+}]
 \le\Pp_\alpha(\mathcal E_{2s}^{+}),
 \qquad
 \Ee_\alpha[1-e^{-\alpha\lambda_-}]
 \le\Pp_\alpha(\mathcal E_{16s,r}^{-}).
\end{equation}

\smallskip
\noindent\textbf{Case (2):}
\((\mathcal E_r^{\mathbb X}\setminus\mathcal E_{2s}^{+})
\cap\mathcal E_{16s,r}^{-}\).
On \(\mathcal E_r^{\mathbb X}\setminus\mathcal E_{2s}^{+}\),
there is a loop \(\ell_+\) intersecting both \(K^+\) and \(\mathsf S_{4s}\).
By independence, \eqref{rwls:eq:defect-loop-comparison} and \eqref{rwls:eq:defect-loop-expectations},
\[
 \Pp_\alpha\bigl((\mathcal E_r^{\mathbb X}\setminus\mathcal E_{2s}^{+})\cap\mathcal E_{16s,r}^{-}\bigr)
 \le \Pp_\alpha(\mathcal E_{16s,r}^{-})\Ee_\alpha[1-e^{-\alpha\widehat\lambda_+}]
 \le C\Pp_\alpha(\mathcal E_{2s}^{+})\Pp_\alpha(\mathcal E_{16s,r}^{-}).
\]

\smallskip
\noindent\textbf{Case (3):}
\((\mathcal E_r^{\mathbb X}\setminus\mathcal E_{16s,r}^{-})
\cap\mathcal E_{2s}^{+}\).
On \(\mathcal E_r^{\mathbb X}\setminus\mathcal E_{16s,r}^{-}\),
there is a loop \(\ell_-\) intersecting both \(K^-\) and \(\mathsf S_{8s}\).
As in Case (2), we conclude that
\[
 \Pp_\alpha\bigl((\mathcal E_r^{\mathbb X}\setminus\mathcal E_{16s,r}^{-})\cap\mathcal E_{2s}^{+}\bigr)
 \le \Pp_\alpha(\mathcal E_{2s}^{+})\Ee_\alpha[1-e^{-\alpha\widehat\lambda_-}]
 \le C\Pp_\alpha(\mathcal E_{2s}^{+})\Pp_\alpha(\mathcal E_{16s,r}^{-}).
\]

\smallskip
\noindent\textbf{Case (4):}
\(\mathcal E_r^{\mathbb X}\setminus
(\mathcal E_{2s}^{+}\cup\mathcal E_{16s,r}^{-})\).
Given \(K^+\) and \(K^-\), if we can find distinct \(\ell_+\) and \(\ell_-\)
that satisfy the previous conditions in Cases (2) and (3), then
the BK inequality (Lemma~\ref{lem:poisson-bk}, which holds for general Poisson point processes as mentioned)
and \eqref{rwls:eq:defect-loop-comparison} bound their disjoint occurrence by
\[
 C(1-e^{-\alpha\lambda_+})(1-e^{-\alpha\lambda_-}).
\]
If they coincide, they intersect both \(K^+\) and \(K^-\).  By
Lemma~\ref{rwls:lem:deterministic-factorization},
\(\mathcal B(K^+,K^-;\mathbb X)\le C\lambda_+\lambda_-\), and the probability of this event is bounded by
\[
 1-e^{-\alpha\mathcal B(K^+,K^-;\mathbb X)}
 \le C(1-e^{-\alpha\lambda_+})(1-e^{-\alpha\lambda_-}).
\]
By independence and
\eqref{rwls:eq:defect-loop-expectations}, together with the same upper bound
for each subcase above, we obtain that
\[
 \Pp_\alpha\bigl(\mathcal E_r^{\mathbb X}\setminus
 (\mathcal E_{2s}^{+}\cup\mathcal E_{16s,r}^{-})\bigr)
 \le C \Ee_\alpha[(1-e^{-\alpha\lambda_+})(1-e^{-\alpha\lambda_-})] \le C\Pp_\alpha(\mathcal E_{2s}^{+})
 \Pp_\alpha(\mathcal E_{16s,r}^{-}).
\]
Summing the four cases proves the result.
\end{proof}

Let \(\mathscr P_{4s}^{\mathbb X}\) be the
\(\mathbb X\)-analogue of \(\mathscr P_{4s}\) in
\eqref{eq:annular-pivotal-loop-class}; that is,
\[
 \mathscr P_{4s}^{\mathbb X}
 :=\{\ell:\operatorname{range}(\ell)
 \subset \mathsf B_{64s},\ 4s\le\diam(\ell)\le128s\}.
\]
Then, we have the following key estimate, which is parallel to
Lemma~\ref{lem:annular-pivotal-lower}.
\begin{proposition}
\label{rwls:prop:localized-pivotal}
There is \(c>0\) such that, uniformly for \(\alpha\in I\) and
\(64\le s\le \frac r{64}\),
\[
 \int_{\mathscr P_{4s}^{\mathbb X}}
 \Pp_\alpha(\operatorname{Piv}_\ell \mathcal E_r^{\mathbb X})
 \,\mu_{\mathbb X}^{\Loop}(\dint\ell)
 \ge cp_r^{\mathbb X,\alpha}.
\]
\end{proposition}

\begin{proof}
We follow the same proof strategy as in Lemma~\ref{lem:annular-pivotal-lower},
which uses an explicit construction of pivotal loops.
Let \(\mathsf F_s^{\delta,\mathbb X}\) be the event that
\(\mathsf S_{5s}\) is connected to \(\mathsf S_{7s}\) using only loops of \(\LL^{\alpha,\le\delta s}_{\mathbb X}\)
that are contained in \(\mathsf A(0,4s,8s)\).
Since \(\alpha_+<\alpha_\ast^{\mathbb X}\), we can choose
\(0<\delta<\frac14\) such that
\(\sup_{s\ge2}\Pp_{\alpha_+}(\mathsf F_s^{\delta,\mathbb X})\le\frac12\).
By monotonicity, the same bound holds for every \(\alpha\in I\).
Now, define the \(\mathbb X\)-analogue of \(\mathsf Z_s\) in
\eqref{eq:annular-barrier-event} by
\[
 \mathsf Z_s^{\mathbb X}
 :=(\mathsf F_s^{\delta,\mathbb X})^{\rm c}
 \cap
 \left\{\nexists\,\ell\in\mathcal L_{\mathbb X}^{\alpha}:
 \diam(\ell)>\delta s,
 \ell\cap \mathsf A(0,5s,7s)
 \ne\varnothing\right\}.
\]
By \eqref{rwls:eq:large-loop-terminal-comparison}, we have
\[
 m_{\delta,s}^{\mathbb X}
 :=\mu_{\mathbb X}^{\Loop}\bigl(
 \diam(\ell)>\delta s,\,
 \ell\cap \mathsf A(0,5s,7s)
 \ne\varnothing\bigr),
 \qquad
 M_{\delta}^{\mathbb X}:=\sup_{s\ge2}m_{\delta,s}^{\mathbb X}<\infty.
\]
Note that the two events involved in the definition of $\mathsf Z_s^{\mathbb X}$ are independent.
This leads to
\begin{equation}\label{rwls:eq:annular-barrier-prob}
 \Pp_\alpha(\mathsf Z_s^{\mathbb X})
 =(1-\Pp_\alpha(\mathsf F_s^{\delta,\mathbb X}))
 e^{-\alpha m_{\delta,s}^{\mathbb X}}
 \ge\frac12e^{-\alpha_+M_{\delta}^{\mathbb X}}.
\end{equation}
Next, consider the following event
\[
 \mathsf G_r^{\mathbb X}(s)
 :=\mathsf H_{\mathsf B_{4s}}^{+}(2s)
 \cap\mathsf Z_s^{\mathbb X}
 \cap\mathsf H_{\mathsf B_{8s}^{\rm c}}^{-}(16s,r).
\]
The three events defining \(\mathsf G_r^{\mathbb X}(s)\) depend on disjoint loop
classes.  Hence, by independence,
Theorem~\ref{thm:rwls-localized-robust-summary}, \eqref{rwls:eq:annular-barrier-prob},
and Lemma~\ref{rwls:lem:localized-splitting}, we obtain that
\[
\begin{aligned}
 \Pp_\alpha(\mathsf G_r^{\mathbb X}(s))
 &=\Pp_\alpha\bigl(\mathsf H_{\mathsf B_{4s}}^{+}(2s)\bigr)
 \Pp_\alpha(\mathsf Z_s^{\mathbb X})
 \Pp_\alpha\bigl(\mathsf H_{\mathsf B_{8s}^{\rm c}}^{-}(16s,r)\bigr)\\
 &\ge c\Pp_\alpha(\mathcal E_{2s}^{+})
 \Pp_\alpha(\mathcal E_{16s,r}^{-})
 \ge cp_r^{\mathbb X,\alpha}.
\end{aligned}
\]
Note that \(\mathsf G_r^{\mathbb X}(s){\subseteq (\mathcal E_r^{\mathbb X})^{\mathrm c}}\).
Moreover, on \(\mathsf G_r^{\mathbb X}(s)\), any loop intersecting both the good outer piece
\(J^+\) and the good inner piece \(J^-\), given by $\mathsf H_{\mathsf B_{4s}}^{+}(2s)$ and
$\mathsf H_{\mathsf B_{8s}^{\rm c}}^{-}(16s,r)$, respectively, is pivotal for the one-arm event
$\mathcal E_r^{\mathbb X}$.
By \cite[Lemma~3.5]{Vog25}, we have
\[
 \mathcal B(J^+,J^-;\mathsf B_{64s})
 \ge cs^{-2}\operatorname{cap}_{\mathbb X}(J^+)
 \operatorname{cap}_{\mathbb X}(J^-)\ge c.
\]
Such a pivotal loop has diameter between \(4s\) and \(128s\), hence lies in
\(\mathscr P_{4s}^{\mathbb X}\).  Therefore,
\[
\begin{aligned}
 \int_{\mathscr P_{4s}^{\mathbb X}}
 \Pp_\alpha(\operatorname{Piv}_\ell\mathcal E_r^{\mathbb X})
 \,\mu_{\mathbb X}^{\Loop}(\dint\ell)
 &\ge \Ee_\alpha\left[
 \1_{\mathsf G_r^{\mathbb X}(s)}\,
 \mu_{\mathbb X}^{\Loop}\{\ell\in\mathscr P_{4s}^{\mathbb X}:
 \ell\cap J^+\ne\varnothing,\ \ell\cap J^-\ne\varnothing\}
 \right]\\
 &\ge c\Pp_\alpha(\mathsf G_r^{\mathbb X}(s))
 \ge cp_r^{\mathbb X,\alpha},
\end{aligned}
\]
as required.
\end{proof}

\begin{proof}[Proof of Theorem~\ref{rwls:thm:lacunary-comparison}]
Set \(s_k=64^k\).  The loop collections
\(\mathscr P_{4s_k}^{\mathbb X}\) are disjoint, and the number of indices with
\(64\le s_k\le \frac r{64}\) is at least \(c_0\log r-C_0\).
By Lemma~\ref{lem:restricted-russo} (which applies to \(\mu_{\mathbb X}^{\Loop}\))
and Proposition~\ref{rwls:prop:localized-pivotal}, we have for
almost every \(\lambda\in[\alpha,\beta]\),
\[
\begin{aligned}
 \frac{\dint}{\dint\lambda}p_r^{\mathbb X,\lambda}
 &\ge\sum_{k:64\le s_k\le r/64}
 \int_{\mathscr P_{4s_k}^{\mathbb X}}
 \Pp_\lambda(\operatorname{Piv}_\ell \mathcal E_r^{\mathbb X})
 \,\mu_{\mathbb X}^{\Loop}(\dint\ell)\\
 &\ge(c_0\log r-C_0)p_r^{\mathbb X,\lambda}.
\end{aligned}
\]
Integrating over \([\alpha,\beta]\), we obtain
\[
 p_r^{\mathbb X,\alpha}
 \le e^{C_0(\beta-\alpha)}r^{-c_0(\beta-\alpha)}
 p_r^{\mathbb X,\beta}
 \le C
 r^{-c_0(\beta-\alpha)}p_r^{\mathbb X,\beta},
\]
where \(C=e^{C_0\operatorname{diam}(I)}\).  This
proves the result with \(c=c_0\).
\end{proof}
\appendix
\section{\texorpdfstring{Proof of the analytic lemma}{Proof of the analytic lemma}}
\label{app:analytic-proof}

\begin{proof}[Proof of Lemma~\ref{ut:lem:analytic}]
Define
\[
        s_1:=
        \inf\left\{
        s\in\R:
        \limsup_{n\to\infty}
        \frac{\log \Sigma_n(s)}{\log n}\ge1
        \right\},
\]
with the convention \(\inf\varnothing=+\infty\).
Since \(\Sigma_n\) is increasing in \(s\), every
\(s>s_1\) belongs to it unless \(s_1=+\infty\).

For the first assertion, let \(t<s_1\).  Choose \(\delta>0\) so small that
\(s:=t+2\delta<s_1\), and
set \(s':=t+\delta\).
Since \(s\) is outside the defining set of \(s_1\), there are \(a\in(0,1)\) and
\(N<\infty\) such that, for every \(n\ge N\),
\[
        \Sigma_n(s)\le n^{1-a}.
\]
For every \(n\ge N\), the monotonicity of \(f_n(s)\) in \(n\) gives
\[
        f_n(s)
        \le \frac1n\sum_{k=0}^{n-1}f_k(s)
        =
        \frac{\Sigma_n(s)}{n}
        \le n^{-a}.
\]
For all large \(n\) and all \(u\le s\),
\[
        f_n(u)\le f_n(s)\le\frac12,
        \qquad
        \Sigma_n(u)\le\Sigma_n(s)\le n^{1-a}.
\]
Applying \eqref{ut:eq:analytic-integrated} between \(s'\) and \(s\) gives
\[
        f_n(s')
        \le f_n(s)\exp\left\{-\frac{c_1\delta n^a}{2}\right\}
        \le \exp\left\{-\frac{c_1\delta n^a}{2}\right\}.
\]
This stretched-exponential bound is summable in \(n\), so
\[
        \Sigma_*:=\sup_n\Sigma_n(s')<\infty .
\]
For all sufficiently large \(n\), applying
\eqref{ut:eq:analytic-integrated} between \(t\) and \(s'\), and using
\(\Sigma_n(s')\le\Sigma_*\) and \(1-f_n(s')\ge\frac12\), gives
\[
        f_n(t)\le
        \exp\left\{-\frac{c_1\delta n}{2\Sigma_*}\right\}.
\]
This proves the exponential estimate for every \(t<s_1\).

For the second assertion, fix \(s>s_1\), put
\(f(u):=\lim_{n\to\infty}f_n(u)\), and suppose for a
contradiction that \(f(s)=0\).  Choose \(s'\in(s_1,s)\).  Since
\(f_i(s)\downarrow0\), there is \(i_0\ge n_0\) such that, for every
\(i\ge i_0\) and \(u\le s\),
\begin{equation}\label{ut:eq:fi-half}
        f_i(u)\le f_i(s)\le\frac12.
\end{equation}
Define, for \(n>i_0\),
\[
        T_n(u):=\frac1{\log n}\sum_{i=i_0}^n\frac{f_i(u)}{i}.
\]
For every \(u\le s\), monotonicity gives
\(0\le f(u)\le f(s)=0\), and logarithmic Ces\`aro summability therefore yields
\(T_n(u)\to 0\).

We recover a differential estimate from the integrated hypothesis.  At a
point \(u\le s\) where the monotone function \(f_i\) is differentiable, apply
\eqref{ut:eq:analytic-integrated} to \(u-h<u\) and let \(h\downarrow0\).   If
\(f_i(u)=0\), the resulting lower bound is trivial; otherwise, combining it
with \eqref{ut:eq:fi-half} gives
\[
        f_i'(u)
        \ge
        c_1\,\frac{i}{\Sigma_i(u)}f_i(u)(1-f_i(u))
        \ge
        \frac{c_1}{2}\,\frac{i f_i(u)}{\Sigma_i(u)}.
\]
Hence, for a.e. \(u\le s\),
\[
        T_n'(u)
        \ge
        \frac{c_1}{2\log n}\sum_{i=i_0}^n
        \frac{f_i(u)}{\Sigma_i(u)} .
\]
Since \(\Sigma_{i+1}(u)=\Sigma_i(u)+f_i(u)\) and
\(\log(1+x)\le x\),
\[
        \frac{f_i(u)}{\Sigma_i(u)}
        \ge
        \log\Sigma_{i+1}(u)-\log\Sigma_i(u).
\]
Therefore
\begin{equation}\label{ut:eq:Tn-derivative-lower}
        T_n'(u)
        \ge
        \frac{c_1}{2\log n}
        \bigl(\log\Sigma_{n+1}(u)-\log\Sigma_{i_0}(u)\bigr).
\end{equation}
Since \(s'\) belongs to the defining set of \(s_1\),
\begin{equation}\label{ut:eq:s1-limsup}
        \limsup_{n\to\infty}
        \frac{\log\Sigma_{n+1}(s')}{\log n}\ge1 .
\end{equation}
Since \(T_n\) is non-decreasing, it is differentiable almost everywhere and
satisfies
\[
        T_n(s)-T_n(s')\ge\int_{s'}^s T_n'(u)\,\dint u .
\]
Substituting \eqref{ut:eq:Tn-derivative-lower} into this integral and using
\(\Sigma_{n+1}(u)\ge\Sigma_{n+1}(s')\) and
\(\Sigma_{i_0}(u)\le i_0\), we obtain
\[
        T_n(s)-T_n(s')
        \ge
        \frac{c_1}{2}(s-s')
        \frac{\log\Sigma_{n+1}(s')-\log i_0}{\log n}.
\]
By \eqref{ut:eq:s1-limsup}, there is a subsequence \(n_j\to\infty\) such that
\[
        \liminf_{j\to\infty}
        \frac{\log\Sigma_{n_j+1}(s')-\log i_0}{\log n_j}\ge1.
\]
Along this subsequence, the right-hand side has limit inferior at least
\(\frac{c_1(s-s')}{2}>0\), whereas
\(T_{n_j}(s)-T_{n_j}(s')\to0\) by the convergence established above.  This
contradiction proves \(f(s)>0\) and completes the proof.
\end{proof}

\bibliographystyle{alpha}
\bibliography{references}

\end{document}